\documentclass[12pt,a4paper]{paper}
\usepackage{amsmath}
\usepackage{amssymb}
\usepackage{amsthm}
\usepackage{amsfonts}
\usepackage{latexsym}

\theoremstyle{plain}
\newtheorem{thm}{Theorem}[section]
\newtheorem{cor}{Corollary}[section]
\newtheorem{lem}{Lemma}[section]
\newtheorem{prop}{Proposition}[section]
\theoremstyle{definition}
\newtheorem{defn}{Definition}[section]

\theoremstyle{remark}

\newtheorem{rem}{Remark}[section]

\title{Local scaling asymptotics for the Gutzwiller trace formula
in Berezin-T\"{o}plitz quantization}
\author{Roberto Paoletti }
\date{}

\begin{document}

\maketitle

\noindent{\bf Address:}
Dipartimento di Matematica e Applicazioni, Universit\`a degli Studi
di Milano Bicocca, Via R. Cozzi 55, 20125 Milano,
Italy; 

\noindent
\textbf{e-mail:}  roberto.paoletti@unimib.it 
\begin{abstract}
Under certain hypothesis on the underlying classical Hamiltonian flow,
we produce local scaling asymptotics in the semiclassical regime for a Berezin-T\"{o}plitz version of the Gutzwiller
trace formula on a quantizable compact K\"{a}hler manifold, in the spirit of the near-diagonal scaling asymptotics 
of Szeg\"{o} and T\"{o}plitz kernels.
More precisely, we consider an analogue of the \lq Gutzwiller-T\"{o}plitz kernel\rq\, 
previously introduced in this setting by Borthwick, Paul and Uribe, and study how it
asymptotically concentrates along the appropriate classical loci
defined by the dynamics, with an explicit description of the exponential decay along normal directions.
These local scaling asymptotics probe into the concentration behavior of the eigenfunctions
of the quantized Hamiltonian flow. When globally integrated, they yield the analogue of the Gutzwiller trace formula. 
\end{abstract}

\section{Introduction}

Let $(M,J)$ be a connected and compact complex manifold, of complex dimension $d$. 
Suppose that there exists on $M$ a positive Hermitian holomorphic line bundle
$(A,h)$, and let $\Theta$ be the curvature form of its unique compatible covariant derivative $\nabla$.
Then $\omega=:(i/2)\,\Theta$ is a K\"{a}hler form on $(M,J)$, and
$\mathrm{d}V_M=\omega^{\wedge d}/d!$ is a volume form on $M$.
We shall denote by $g$ the Riemannian structure associated to $\omega$ and $J$.

This paper is concerned with the semiclassical aspects of the Berezin-T\"{o}plitz
quantization of the K\"{a}hler manifold $(M,J,2\omega)$, a theme which has attracted 
great attention in recent years 
(see for instance, of course with no pretense of completeness whatsoever, 
\cite{ber}, \cite{cgr}, \cite{bms}, \cite{bdm-g}, \cite{bpu1}, \cite{bpu2}, \cite{z1}, \cite{ks}, \cite{ch1},
\cite{mz}, 
\cite{mm1}, \cite{mm2}, \cite{bms}, \cite{s}, \cite{g-star}). As we point out at the end of this introduction,
the present results generalize to the almost complex setting, where one considers a quantizable compact symplectic
manifold $(M,\omega)$, endowed with a compatible almost complex structure $J$. 

In \cite{bpu2}, in particular, an analogue of the Gutzwiller trace formula was investigated in
the Berezin-T\"{o}plitz setting. Roughly speaking, this formula deals with the asymptotics of
the trace of the distributional kernel of a smoothed spectral projector relative to a spectral band
of width $O(\hbar)$ about an energy value $E$.   
The basic motivation for the present work
is to revisit the general theme of \cite{bpu2} in light of the approach to scaling asymptotics
in geometric quantization 
more recently emerged in the series of papers
\cite{z2}, \cite{bsz1}, \cite{sz}, \cite{bsz2} and built on the microlocal theory of the Szeg\"{o} kernel in \cite{bs}.

In order to clarify the scope of the present analysis, let us recall a Definition:

\begin{defn}
 \label{defn:compatible}
 A function $f\in \mathcal{C}^\infty(M,\mathbb{R})$ will be called \textit{compatible} (with the K\"{a}hler structure)
 if the Hamiltonian flow associated to it, $\phi^M_\tau:M\rightarrow M$ is a 1-parameter group of holomorphic
 automorphisms of $(M,J)$. 
\end{defn}

These functions are also called \textit{quantizable} in the literature \cite{cgr}. Compatible functions
are of course quite special, but play a very important role in complex geometry and geometric quantization, as they are
tightly related to Hamiltonian group actions preserving the quantization set-up. 

Let then $f$ be a compatible function on $M$, and consider its associated Hamiltonian vector field $\upsilon_f$ on 
$(M,2\omega)$. It generates the flow $\phi^M_\tau:M\rightarrow M$, which heuristically represents the classical evolution
of the system with Hamiltonian $f$.

On the quantum side of the picture, we replace $(M,J,\omega)$ with the spaces of global holomorphic sections 
$H^0\left(M,A^{\otimes k}\right)$ of higher tensor powers of $A$. Given the volume form $\mathrm{d}V_M$ and the Hermitian metric
$h$, we may consider the Hilbert space $L^2\left(M,A^{\otimes k}\right)$ of square summable sections of $A^{\otimes k}$, and
$H^0\left(M,A^{\otimes k}\right)$ sits in it as a closed subspace. In geometric quantization, one views 
$H^0\left(M,A^{\otimes k}\right)$ as the quantum Hilbert space associated to $(M,J,2\omega)$, at Planck's constant $\hbar=1/k$, 
$k=1,2,\ldots$. The semiclassical regime corresponds to studying the asymptotic aspects of this picture for $k\rightarrow+\infty$.

The quantum counterpart of the observable $f$ is often taken to be the self-adjoint Berezin-T\"{o}plitz operator 
\begin{equation}
 \label{eqn:usual toeplitz}
 T_f'^{(k)}:s\in H^0\left(M,A^{\otimes k}\right)\mapsto P_k(f\,s)\in  H^0\left(M,A^{\otimes k}\right),
\end{equation}
where $P_k:L^2\left(M,A^{\otimes k}\right)\rightarrow H^0\left(M,A^{\otimes k}\right)$ denotes the orthogonal projector.
This is a (zeroth-order) Berezin-T\"{o}plitz operator 
(see Definition \ref{defn:toplitz} below), 
and its eigenvalues $\eta_{kj}$ satisfy the elementary bound
$a_f\le \eta_{kj}\le A_f$, where $a_f=:\min (f)$ and $A_f=:\max(f)$ (see e.g. Lemma 2.1 of \cite{pao_imrn}). 

When $f$ is compatible,
however, there is a more natural choice for the quantization of $f$.
Before explaining this, 
let us reformulate the picture in the frame of Hardy spaces and Szeg\"{o} kernels, following the general approach in
\cite{bdm-g}, \cite{bpu1}, \cite{bpu2}, \cite{z2}, \cite{bsz2}, \cite{sz}, \cite{bsz2}.

Let $ A^\vee\supseteq X\stackrel{\pi}{\rightarrow }M$ be the unit circle bundle in the dual line bundle 
(with the naturally induced Hermitian structure).
The covariant derivative $\nabla$ corresponds on $X$ to a connection 1-form $\alpha$, satisfying $\mathrm{d}\alpha=2\,\pi^*(\omega)$.
Thus $(X,\alpha)$ is a contact manifold, and $\mathrm{d}V_X(x)=:(\alpha/2\pi)\wedge \pi^*(\mathrm{d}V_M)$
is a volume form on $X$. The horizontal tangent bundle 
$\mathcal{H}=:\ker(\alpha)\subseteq TX$, with the pull-back to it of the complex structure $J$,
is a CR structure on $X$. 

The corresponding Hardy space $H(X)\subseteq L^2(X)$ then splits equivariantly and unitarily under the $S^1$-action as an
orthogonal direct sum of isotypical components,
$$
H(X)=\bigoplus_{k\ge 0}H(X)_k.
$$
There is natural unitary isomorphism $H(X)_k\cong H^0\left(M,A^{\otimes k}\right)$
for every $k$.

The orthogonal projector
$
\Pi:L^2(X)\rightarrow H(X)$ is called the \textit{Szeg\"{o} projector} on $X$; its distributional kernel 
$\Pi\in \mathcal{D}'(X\times X)$ is the \textit{Szeg\"{o} kernel} of $X$.
If $(e_{kj})_j$ is any orthonormal basis of $H(X)_k$, then $\Pi=\sum_k\Pi_k$, where for each $k$
\begin{equation}
 \label{eqn:spectral szego}
\Pi_k(x,y)=\sum_{j=1}^{N_k} e_{kj}(x)\cdot \overline{e_{kj}(y)}\,\,\,\,\,\,\,(x,y\in X);
\end{equation}
$\Pi_k$ is the $\mathcal{C}^\infty$ $k$-th equivariant Szeg\"{o} kernel\footnote{We shall generally use the same notation for an operator and its Schwartz
kernel.}, that is, the Schwartz kernel of the orthogonal
projector $\Pi_k:L^2(X)\rightarrow H(X)_k$.

Consider an Hamiltonian vector field $\upsilon_f$ on $M$, with $f$ not necessarily compatible; there is
a natural lift of $\upsilon_f$ to a contact vector field $\widetilde{\upsilon}_f$ on $X$, given by
\begin{equation}
 \label{eqn:defn of contact vf}
 \widetilde{\upsilon}_f=:\upsilon_f^\sharp-f\,\frac{\partial}{\partial \theta}.
\end{equation}
Here notation is as follows: $\upsilon_f^\sharp:x\mapsto \upsilon_f^\sharp(x)\in \mathcal{H}_x$ 
is the horizontal lift of $\upsilon_f$ to $X$ (with respect to $\alpha$),
and $\partial/\partial \theta$ is the generator of the standard $S^1$-action on $X$ (in 
local coordinates the latter action will be expressed by translation in an angular coordinate); in addition, we have
written $f$ for the pull-back $\pi^*(f)$.
While $\upsilon_f=\upsilon_{ f+c}$, with $c$ a real constant,
clearly $\widetilde{\upsilon}_f\neq \widetilde{\upsilon}_{f+c}$ unless $c=0$: 
the same Hamiltonian flow has many contact lifts to $X$.

Consider the 1-parameter group of contact transformations $\phi^X_\tau:X\rightarrow X$ generated by $\widetilde{\upsilon}_f$.  
Given that $f$ is compatible, $\phi^X_\tau$ preserves the CR structure and the Hardy space $H(X)$.
This corresponds to a 1-parameter group of holomorphic automorphism 
of $A^{\otimes k}$ for every $k$, preserving the Hermitian structure. 

We have
a 1-parameter group of unitary automorphisms of $H(X)_k$, induced by pull-back:
\begin{equation}
 \label{eqn:defn Utau}
 U(\tau)_k:H(X)_k\rightarrow H(X)_k,\,\,\,\,\,
 s\mapsto \left(\phi^X_{-\tau}\right)^*(s)=:s\circ \phi^X_{-\tau}.
\end{equation}

The latter is evidently the 1-parameter group generated by the self-adjoint operator 
\begin{equation}
 \label{eqn:1st order toeplitz operator}
T(f)_k=:\left.i\,\widetilde{\upsilon}_f\right|_{H(X)_k}:H(X)_k\rightarrow H(X)_k,
\end{equation}
where $\widetilde{\upsilon}_f$ is viewed as a differential operator.
In other words,
\begin{equation}
 \label{eqn:defn Utau1}
 U(\tau)_k=e^{i\tau\,T(f)_k}:H(X)_k\rightarrow H(X)_k.
\end{equation}

Thus, if $\lambda_{kj}$ are the eigenvalues of $T(f)_k$, repeated according to multiplicity,
the eigenvalues of $U(\tau)_k$ are given by
$e^{i\tau\,\lambda_{kj}}$. 

Let us recall the following definition from \cite{bdm-g}:

\begin{defn}
\label{defn:toplitz}
 A T\"{o}plitz operator of degree $k$ on $X$ is an operator of the form
 $T=\Pi\circ Q\circ \Pi$, where $Q$ is a pseudodifferential operator of degree $k$
 on $X$; its principal symbol $\sigma_T$ is the restriction of the principal symbol of $Q$
 to the closed symplectic cone sprayed by the connection 1-form,
 $$
 \Sigma=:\big\{(x,r\alpha_x)\,:\,x\in X,\,r>0\big\}\subseteq T^*X\setminus (0).
 $$
 Here $T^*X\setminus (0)$ is the complement of the zero section in the cotangent bundle of $X$.
 
\end{defn}

A T\"{o}plitz operator is commonly viewed as a 
 (generally unbounded) linear endomorphism of the Hardy space.
For instance, if $Q=-i\,\partial/\partial\theta$, then $T$ is the so-called \lq number operator\rq\,
$\mathcal{N}=:\bigoplus_{k\ge 0}k\cdot \mathrm{id}_{H(X)_k}$.

Therefore, the formally
self-adjoint operator $$T_f=:\bigoplus _{k\ge 0}T(f)_k:H(X)\rightarrow H(X)$$ is a first order T\"{o}plitz  
operator in the sense of \cite{bdm-g}, with
principal symbol 
$$
\sigma_{T(f)}(x,r\,\alpha_x)=r\,f\big(\pi(x)\big);
$$
one can see from this (see e.g. Corollary 2.1 of \cite{pao_imrn}) that
$$
a_f\,k+O(1)\le \lambda_{kj}\le A_f\,k+O(1).
$$

The $\mathcal{C}^\infty$ distributional kernels of 
$U(\tau)_k$ and $T(f)_k$ may be described in terms of spectral data, as follows.
Let $(e_{kj})_j$ be an orthonormal basis of $H(X)_k$, composed of eigenvectors of $T(f)_k$ 
relative to the $\lambda_{kj}$'s.
Then
\begin{equation}
 \label{eqn:spectral toeplitz}
T(f)_k(x,y)=\sum_{j=1}^{N_k} \lambda_{kj}\,e_{kj}(x)\cdot \overline{e_{kj}(y)},
\end{equation}
and
\begin{equation}
 \label{eqn:spectral unitary}
U(\tau)_k(x,y)=\sum_{j=1}^{N_k}e^{i\tau\,\lambda_{kj}}\,e_{kj}(x)\cdot \overline{e_{kj}(y)};
\end{equation}
here $N_k=h^0\left(M,A^{\otimes k}\right)=\dim H(X)_k$.

Let us now consider $\chi\in \mathcal{C}^\infty_0(\mathbb{R})$ and a real number 
$E\in \mathbb{R}$. Let 
$\mathcal{G}_k=\mathcal{G}_k^{(\chi)}:H(X)_k\rightarrow H(X)_k$
be the operator
with Schwartz kernel 
\begin{equation}
 \label{eqn:spectral gutzwiller}
\mathcal{G}_k(x,y)=:\sum_{j=1}^{N_k} \widehat{\chi}(kE-\lambda_{kj})\,e_{kj}(x)\cdot \overline{e_{kj}(y)};
\end{equation}
we shall call $\mathcal{G}_k$ a level-$k$
Gutwziller-T\"{o}plitz kernel. Although this will be mostly left implicit for notational simplicity, $\mathcal{G}_k$ 
depends on $f$, $\chi$ and $E$ (see \cite{bpu2}) (however, in the proof of Theorem \ref{thm:trace of G_k}
below we shall need to emphasize the dependence on $\chi$). 

Operators of this kind in the Berezin-T\"{o}plitz context were studied in \cite{bpu2}.
Heuristically,  $\mathcal{G}_k$ is a smoothed spectral projector, 
associated to a spectral band of $T(f)_k$ about $kE$, of width $O(1)$. If $T(f)$ is turned
into a zeroth-order T\"{o}plitz operator $T^\sharp(f)_k$ \label{pageTsharp} by composing it with a parametrix for the number operator,
we should speak of a spectral band about $E$, of width $O(1/k)$.


Theorem 1.1 of \cite{bpu2} deals with an operator of type 
(\ref{eqn:usual toeplitz}), in the Hardy space formulation,
and describes the asymptotics as $k\rightarrow +\infty$ of 
$\mathcal{G}_k(x,y)$ at fixed $x,y\in X$ ($f$ is not required to be compatible). 
One has $\mathcal{G}_k(x,y)=O\left(k^{-\infty}\right)$,
unless $m_x=:\pi(x),\,m_y=:\pi(y)\in M$ satisfy the following conditions:
\begin{enumerate}
 \item there exists $\tau\in \mathrm{supp}(\chi)$ such that $m_x=\phi^M_\tau(m_y)$;
 \item $f(m_x)=E$.
\end{enumerate}

Furthermore, let $\mathcal{R}_\chi (f,E)\subseteq M\times M$ \label{pageRchifE}
be the locus of points satisfying these two conditions, 
and denote by 
$\widetilde{\mathcal{R}}_\chi (f,E)$ 
its inverse image in $X\times X$.
Suppose $(x,y)\in  \widetilde{\mathcal{R}}(f,E)$ and that $f$ is submersive at $m_x$. Then 
$\mathcal{G}_k(x,y)$ is
described by an asymptotic expansion in descending integer powers of $k$, with leading order term of degree
$d-1/2$. The leading coefficient depends on Poincar\'{e} type data of the
Hamiltonian flow,
and the subprincipal symbol of an appropriate pseudodifferential operator. 

The arguments in \cite{bpu2} are based on the theory of Fourier-Hermite
distributions and their symbolic calculus in terms of symplectic spinors.

This \lq jumping behavior\rq\, in the asymptotics of
$\mathcal{G}_k(x,y)$ as $(m_x,m_y)$ moves away from the classically defined special locus
$\mathcal{R}_\chi (f,E)$ motivates studying the
asymptotic concentration of $\mathcal{G}_k$
in shrinking neighborhoods of 
$\widetilde{\mathcal{R}}(f,E)$.
In other words, one is naturally led to study the behavior of $\mathcal{G}_k$ not at fixed points,
but at sequences $(x_k,y_k)\rightarrow \widetilde{\mathcal{R}}(f,E)$ at appropriate rates.
We shall attack this problem working in rescaled loal coordinates at points $(x,y)\in \widetilde{\mathcal{R}}(f,E)$;
in doing so we shall build, rather than on Fourier-Hermite distributions, 
on the approach to the near-diagonal scaling asymptotics of level-$k$ Szeg\"{o} kernels 
developed in the series of papers \cite{z2}, \cite{bsz1}, \cite{sz}, \cite{bsz2}.

The present analysis is restricted to compatible Hamiltonians and to first order T\"{o}plitz operators
 of type (\ref{eqn:1st order toeplitz operator}). Compatible Hamiltonians are of course quite special, but nonetheless 
of exceptional importance in complex geometry and harmonic analysis, given their tight relation to holomorphic Hamiltonian
actions. 
In addition, the proofs, the geometric significance and the dynamical
interpretation of the scaling asymptotics are particularly transparent in this case,
and can thus serve as a guide for future developments.

In order to state our results, we first need to introduce
some notation.

\begin{defn}
 \label{defn:general periods}
 Let $\mu:G\times D\rightarrow D$, $(g,d)\mapsto \mu(g,d)=\mu_g(d)$, be an action of
 a group $G$ on a set $D$. For any $d\in D$, the set of \textit{periods} of $d$ in $G$
 is the stabilizer subgroup of $d$:
 $$
 \mathrm{Per}^G_D(d)=:\big\{g\in G\,:\,\mu(g,d)=d\big\}.
 $$
 If $(d,d')\in D\times D$, we shall set 
 $$
 \mathrm{Per}^G_D(d'\mapsto d)=:\big\{g\in G:\,\mu(g,d')=d\big\}=\big\{g\in G:\,d'=d_g\big\},
 $$
 where $d_g=:\mu_{g^{-1}}(d)$. 
\end{defn}


\begin{defn}
 \label{defn:periods subsets}
With the hypothesis and notation of Definition \ref{defn:general periods}, 
given a $G$-invariant subset $S\subseteq D$ let us set
$$
\mathrm{Per}^G_M(S)=:\bigcup_{d\in S}\mathrm{Per}^G_D(d).
$$
\end{defn}

In the present picture, we have the following built-in actions:
\begin{enumerate}
 \item the Hamiltonian action $\phi^M:\mathbb{R}\times M\rightarrow M$;
 \item the contact action $\phi^X:\mathbb{R}\times X\rightarrow X$;
 \item the action of 
$\mathbb{R}\times S^1$ on $X$ obtained by composing $\phi^X$
with the structure action $r:S^1\times X\rightarrow X$.
\end{enumerate}

If $m\in M$, let us set $X_m=:\pi^{-1}(m)\subseteq X$. If $(m,n)\in M\times M$ and $(x,y)\in X_m\times X_n$,
then there is a natural bijection 
$$
\mathrm{Per}^{\mathbb{R}}_M(n\mapsto m)\cong \mathrm{Per}^{\mathbb{R}\times S^1}_X(y\mapsto x),
\,\,\,\,\tau\mapsto (\tau,g_\tau),
$$
which for $n=m$ and $x=y$ is a group isomorphism (\S \ref{sctn:periods} below).

The scaling asymptotics of $\mathcal{G}_k$ at a point $(m,n)\in \mathcal{R}(f,E)$ with $\upsilon_f(m)\neq 0$
are controlled by a universal exponent, given by a quadratic function $\mathcal{Q}$ on pairs of tangent vectors
(see (\ref{eqn:defn of Q}) below).
The latter is defined in terms of a natural orthogonal decomposition of the tangent spaces
$T_mM$ and $T_nM$ of $M$ at $m$ and $n$, respectively.

Let us set 
\begin{equation}
\label{eqn:MEXE}
M_E=:f^{-1}(E)\subseteq M,\,\,\,\,\,\,\,\,\,\,X_E=:\pi^{-1}(M_E)\subseteq X.
\end{equation}

Suppose that $m\in M_E$, and that $f$ is submersive at $m$. 
The (Riemannian) normal space to $M_E$ at $m$
is 
\begin{equation*}
  N_m(M_E)=\mathrm{span}_{\mathbb{R}}\{ J_m\big(\upsilon_f(m)\big)\}.
\end{equation*}
Let 
\begin{equation}
 \label{eqn:RmSm}
 R_m=:\mathrm{span}_\mathbb{R}\big(\upsilon_f(m),\,J_m\upsilon_f(m)\big)=
\mathrm{span}_\mathbb{C}\big(\upsilon_f(m)\big),\,\,\,\,\,\,S_m=:R_m^\perp,
\end{equation}
where the orthocomplement in $T_mM$ may be taken equivalently in the Riemannian or in the Hermitian sense.

Then $\mathrm{span}_\mathbb{R}\big(\upsilon_f(m)\big)=T_mM_E\cap \mathrm{span}_\mathbb{C}\big(\upsilon_f(m)\big)$
and $T_mM_E$ splits as a Riemannian orthogonal direct sum 
$$
T_mM_E=\mathrm{span}_\mathbb{R}\big(\upsilon_f(m)\big)\oplus S_m.
$$

\begin{defn}
 \label{defn:tvh components}
 Let us now provide an orthocomplement direct sum decomposition of $T_mM$ which
 reflects the local geometry induced by the Hamiltonian flow. 
 Given the orthogonal direct sum decomposition
$$
T_mM=\mathrm{span}_\mathbb{R}\big(J_m\big(\upsilon_f(m) \big)
\oplus \mathrm{span}_\mathbb{R}\big(\upsilon_f(m)\big)\oplus S_m,
$$ 
at any $m\in M_E(\sigma)$ with $\mathrm{d}_mf\neq 0$,
we can uniquely write any $\mathbf{w}\in T_mM$ as 
$\mathbf{w}=\mathbf{w}_\mathrm{t}+\mathbf{w}_\mathrm{v}+\mathbf{w}_\mathrm{h}$, where
$\mathbf{w}_\mathrm{t}\in \mathrm{span}_\mathbb{R}\big(J_m\big(\upsilon_f(m) \big)$,
$\mathbf{w}_\mathrm{v}\in \mathrm{span}_\mathbb{R}\big(\upsilon_f(m)\big)$,
$\mathbf{w}_\mathrm{h}\in S_m$. 
Thus $\mathbf{w}_\mathrm{t}=b\,J_m\big(\upsilon_f(m)\big)$ and $\mathbf{w}_\mathrm{v}=a\,\upsilon_f(m)$
for certain $a,\,b\in \mathbb{R}$.
The suffix  $\mathrm{t}$ stands for \lq transverse\rq\, (to $M_E$),
$\mathrm{v}$ stands for \lq vertical\rq\, and $\mathrm{h}$
for \lq horizontal\rq\,.  
\end{defn}

Verticality and horizontality heuristically 
refer to the quotient fibration locally induced by
the orbits of the $\mathbb{R}$-action in $M_E$.

Let us now recall a definition from \cite{sz}.

\begin{defn}
\label{defn:psi2}
 Let $(V,h)$ be an Hermitian vector space, and write $g=\Re (h)$ and $\omega=-\Im(h)$
 (real and imaginary parts).
Let $\|\cdot\|$ be the induced norm, and  set 
$$
\psi_2(v_1,v_2)=:-i\,\omega(v_1,v_2)-\frac{1}{2}\,\|v_1-v_2\|^2\,\,\,\,\,\,\,\,(v_1,v_2\in V).
$$
\end{defn}

This is the universal exponent appearing in near-diagonal scaling asymptotics of level-$k$ Szeg\"{o} kernels
when the latter are expressed in Heisenberg local coordinates.

Let us now define $\mathcal{Q}:T_mM\times T_mM\rightarrow \mathbb{C}$ by
\begin{eqnarray}
 \label{eqn:defn of Q}
 \mathcal{Q}(\mathbf{w}_1,\mathbf{w}_2)&=:&\psi_2(\mathbf{w}_{1\mathrm{h}},\mathbf{w}_{2\mathrm{h}})\\
 &&+\Big[i\,\big(\omega(\mathbf{w}_{1\mathrm{v}},\mathbf{w}_{1\mathrm{t}})
 -\omega(\mathbf{w}_{2\mathrm{v}},\mathbf{w}_{2\mathrm{t}})\big)
 -\left(\|\mathbf{w}_{1\mathrm{t}}\|^2+\|\mathbf{w}_{2\mathrm{t}}\|^2\right)\Big].\nonumber
\end{eqnarray}


The scaling asymptotics of $\mathcal{G}_k$ below
will be expressed in Heisenberg local coordinates (HLC for short), which were introduced in \cite{sz}.
These local coordinates render the universal nature of the near-diagonal scaling asymptotics of
$\Pi_k$ especially transparent. We refer to \cite{sz} for a precise definition and discussion, and
to \S \ref{sctn:HLC} below a quick review. 

A Heisenberg local coordinate chart for $X$ centered at $x$,
$\gamma_x:(-\pi,\pi)\times B_{2d}(\mathbf{0},\delta)\rightarrow X'\subseteq X$, is often denoted in additive notation,
$\gamma_x(\theta,\mathbf{v})=x+(\theta,\mathbf{v})$; it satisfies the following properties:

\begin{enumerate}\label{pageHLC}
 \item Let $V_x=:\ker(\mathrm{d}_x\pi),\,H_x=:\ker(\alpha_x)$ be the
vertical and horizontal tangent spaces of $X$\footnote{these should not be confused with the \lq vertical\rq \,and \lq horizontal\rq\,
distributions defined above on $M$ along $M_E$}; then the differential 
$$\mathrm{d}_{(0,\mathbf{0})}\gamma_x:\mathbb{R}\oplus \mathbb{C}^d\rightarrow T_xX=V_x\oplus H_x$$ is unitary and preserves
the direct sum (in particular, the isomorphism $\mathbb{C}^d\rightarrow H_x$ is $\mathbb{C}$-linear).
\item The standard circle action on $X$ is expressed in HLC by translation in the angular coordinate, that is, whenever both terms 
are well-defined we have
$$
r_{e^{i\vartheta}}\big(x+(\theta,\mathbf{v})\big)=x+(\vartheta+\theta,\mathbf{v});
$$
we shall generally write $r_\vartheta$ for $r_{e^{i\vartheta}}$.
\item In particular, if $m=:\pi(x)$ then in additive notation
$$
\mathbf{v}\in B_{2d}(\mathbf{0},\delta)\mapsto m+\mathbf{v}=:\pi\Big(x+(\theta,\mathbf{v})\Big)\in M'=:\pi (X')\subseteq M
$$
is a well-defined coordinate chart on $M$ centered at $m$ (that is, it does not depend on $\vartheta$), and its
differential at $\mathbf{0}\in \mathbb{R}^{2d}$ 
determines a unitary isomorphism $\mathbb{C}^d\rightarrow T_mM$. This coordinate system is called \textit{preferred}, 
or \textit{adapted}.
\end{enumerate}
\noindent
We shall generally write $x+\mathbf{v}$ for $x+(0,\mathbf{v})$.

Lastly, let us give the following Definition:

\begin{defn}
 \label{defn:supported orbit}

If $m\in M$ and $\chi\in \mathcal{C}^\infty_0(\mathbb{R})$, we shall denote by $m^\chi$ 
the portion of the $\mathbb{R}$-orbit of $m$ determined by the support of $\chi$:
\begin{equation*}
 m^\chi=:\left\{m_\tau\,:\,\tau\in \mathrm{supp}(\chi)\right\}.
\end{equation*}

\end{defn}

The locus $\mathcal{R}_\chi(f,E)$ \label{pageRchifE1} (page \pageref{pageRchifE}) of pairs $(m,n)\in M\times M$ over which 
$\mathcal{G}_k$ asymptotically concentrates for $k\rightarrow +\infty$ is determined by
the conditions
$$
\mathrm{dist}_M\left(n,m^\chi\right)=0,\,\,\,\,\,f(m)-E=0,
$$
where $\mathrm{dist}_M$ is the distance function on $M$ associated to the K\"{a}hler metric.

We shall now formulate a quantitative estimate of the rate of concentration in terms of these two
quantities.

\begin{thm}
\label{thm:rapid decrease}
Let $f\in \mathcal{C}^\infty(M)$ be a compatible Hamiltonian and suppose
 $\chi\in \mathcal{C}^\infty_0(\mathbb{R})$.
Let $\mathcal{G}_k=\mathcal{G}^\chi_k\in \mathcal{C}^\infty(X\times X)$ be the level-$k$ Gutzwiller-T\"{o}plitz
kernel associated to the first order T\"{o}plitz operator (\ref{eqn:1st order toeplitz operator}).
For any $\epsilon,\,C>0$, we have 
$$
\mathcal{G}_k(x,y)=O\left (k^{-\infty}\right)
$$
uniformly for
$$
\max\left\{\mathrm{dist}_M\left(m_y,m_x^\chi\right),\,|f(m_x)-E|\right\}\ge C\,k^{\epsilon-\frac{1}{2}}.
$$
\end{thm}

A more explicit formulation is the following: there exist constants 
$$C_j=C_j(C,\epsilon)>0,\,\,\,\,\,\,\,\,j=1,2,\ldots,$$ 
such that
for any choice of sequences $m_k,\,n_k\in M$ satisfying 
$$
\max\left\{\mathrm{dist}_M\left(n_k,m_k^\chi\right),\,|f(m_k)-E|\right\}\ge C\,k^{\epsilon-\frac{1}{2}} 
$$
for $k=1,2,\ldots$ we have
$$
\big|\mathcal{G}_k(x_k,y_k)\big|\le C_j\,k^{-j}
$$
whenever $(x_k,y_k)\in X_{m_k}\times X_{n_k}$.

\begin{cor}
 \label{cor:diagonal case}
 In the hypothesis of Theorem \ref{thm:rapid decrease}, we have
 $
\mathcal{G}_k(x,x)=O\left (k^{-\infty}\right)
$
uniformly for
$$
|f(m_x)-E|\ge C\,k^{\epsilon-\frac{1}{2}}.
$$
\end{cor}

Let us heuristically interpret $\mathcal{G}_k$ as a smoothed spectral projector, associated
to a spectral band of the zeroth order T\"{o}plitz operator 
$T^\sharp(f)_k$ (page \pageref{pageTsharp}),
of width $O\left(k^{-1}\right)$ about $E$. Then, in the range 
of Corollary \ref{cor:diagonal case}, all the eigenvalues in the band
stay at a distance $O\left(k^{\epsilon-1/2}\right)$ from $f(m_x)$. 
Thus Corollary \ref{cor:diagonal case} tallies in spirit 
with the results of \cite{pao_imrn}, where it is proved that the eigenfunctions
$e_{kj}$ pertaining to eigenvalues $\lambda_{kj}$ with $\big|\lambda_{kj}-f(m_x)\big|\ge C_1\,k^{\epsilon-1/2}$
contribute negligibly for $k\rightarrow +\infty$ to the equivariant Szeg\"{o} kernel at $x$, $\Pi_k(x,x)$.

Given this, we look for explicit asymptotic expansions for $\mathcal{G}_k$ in shrinking neighborhoods
of $\widetilde{\mathcal{R}}_\chi (f,E)$ (pages \pageref{pageRchifE} and \pageref{pageRchifE1}). We shall furthermore restrict our analysis to the open sublocus 
\begin{equation}
 \label{eqn:defnRchiprime}
 \mathcal{R}_\chi(f,E)'=:\big\{(m,n)\in \mathcal{R}_\chi(f,E)\,:\,\upsilon_f(m)\neq 0\big\},
\end{equation}
and to its inverse image $\widetilde{\mathcal{R}}_\chi (f,E)'\subseteq X\times X$. Since $f$ is $\phi^M$-invariant, 
if $n=m_\tau$ and
$\upsilon_f(m)\neq 0$, then also $\upsilon_f(n)\neq 0$.

Let us fix $(x,y)\in \widetilde{\mathcal{R}}_\chi (f,E)'$,
and HLC systems centered at $x$ and $y$, respectively, that we shall denote additively as above.
Then
$\mathrm{Per}^\mathbb{R}_M(m_y\mapsto m_x)$ is a discrete subset of $\mathbb{R}$, either infinite or reduced to
a point, and so $\mathrm{Per}^\mathbb{R}_M(m_y\mapsto m_x)\cap \mathrm{supp}(\chi)$ is a finite set.
We shall write 
\begin{equation}
 \label{eqn:period set}
\mathrm{Per}^\mathbb{R}_M(m_y\mapsto m_x)=\{\tau_a\}_\mathcal{A}
\end{equation}
where the index set $\mathcal{A}$
depends on $x$; either
$\mathcal{A}=\{0\}$, or else $\mathcal{A}=\mathbb{Z}$. Accordingly,
$$
\mathrm{Per}^{\mathbb{R}\times S^1}_X(y\mapsto x)=\left\{\left(\tau_a,e^{i\vartheta_a}\right)\right\}_\mathcal{A}.
$$
Thus, for every $a\in \mathcal{A}$ we have (see Definition \ref{defn:general periods})
\begin{equation}
 \label{eqn:lifted period}
  x_{\tau_a}=\phi^X_{-\tau_a}(x)=e^{i\vartheta_a}\,y.
\end{equation}

We shall, furthermore, denote by $A_a$ ($a\in \mathcal{A}$) 
the unitary (that is, symplectic and orthogonal) $2d\times 2d$
matrix representing the differential
\begin{equation}
 \label{eqn:differentialHLC}
\mathrm{d}_{m_x}\phi^{M}_{-\tau_a}:T_{m_x}M\rightarrow T_{m_y}M
\end{equation}
with respect to the given HLC systems.

\begin{thm}
 \label{thm:scaling asymptotics}
 Adopt the hypothesis and notation of Theorem \ref{thm:rapid decrease},
 and the notation (\ref{eqn:period set}).
 Suppose $(x,y)\in \widetilde{\mathcal{R}}_\chi(f,E)'$ (see (\ref{eqn:defnRchiprime})), 
 and adopt HLC systems centered at $x$ and $y$.
Fix $C>0$ and $\epsilon\in (0,1/6)$. Then uniformly in $\mathbf{v}_1,\,\mathbf{v}_2\in \mathbb{R}^{2d}$
satisfying $\|\mathbf{v}_1\|,\,\|\mathbf{v}_2\|\le C\,k^{\epsilon}$ the following asymptotic expansion
holds for $k\rightarrow +\infty$:
\begin{eqnarray*}
 \lefteqn{\mathcal{G}_k\left(x+\left(\theta_1,\frac{\mathbf{v}_1}{\sqrt{k}}\right),
 y+\left(\theta_2,\frac{\mathbf{v}_2}{\sqrt{k}}\right)\right)}\\
&\sim&\dfrac{\sqrt{2}}{\|\upsilon_f(m_x)\|}\,\left(\frac{k}{\pi}\right)^{d-1/2}\,e^{ik(\theta_1-\theta_2)}\,
\sum_{a\in \mathcal{A}}\,\mathbf{G}_k^{(a)}
\left(x+\frac{\mathbf{v}_1}{\sqrt{k}},y+\frac{\mathbf{v}_2}{\sqrt{k}}\right),
\end{eqnarray*}
where for each $a$
\begin{eqnarray}
 \label{eqn:a-th summand expansion}
\lefteqn{\mathbf{G}_k^{(a)}
\left(x+\frac{\mathbf{v}_1}{\sqrt{k}},y+\frac{\mathbf{v}_2}{\sqrt{k}}\right)}\\
&\sim&e^{ik(\vartheta_a-\tau_a E)+\mathcal{Q}(A_a\mathbf{v}_1,\mathbf{v}_2)}\,\left[\chi(\tau_a)
+\sum_{j=1}^{+\infty}k^{-j/2}\,P_{aj}^\chi(m_x,m_y;
\mathbf{v}_1,\mathbf{v}_2)\right];  \nonumber
\end{eqnarray}
here $\mathcal{Q}$ is as in (\ref{eqn:defn of Q}) and
$P_{aj}^\chi$ is a polynomial in $(\mathbf{v}_1,\mathbf{v}_2)$, of degree $\le 3j$ and parity
$(-1)^j$, whose coefficients depend smoothly on $m_x$ and $m_y$, and which vanishes identically when 
$\tau_a\not\in \mathrm{supp}(\chi)$. More precisely $P_{aj}^\chi$ is, as a function of $\chi$, the evaluation
at $\tau_a$ of a differential polynomial in $\chi$ of degree $\le j$.
\end{thm}

Notice that $\|\upsilon_f(m_x)\|=\|\upsilon_f(m_y)\|$ and that, by the previous remarks, the sum over $a$ in the statement of
Theorem \ref{thm:scaling asymptotics} is finite. In addition, the $j$-th summand in (\ref{eqn:a-th summand expansion})
satisfies
\begin{equation}
 \label{eqn:bound ath summand}
 \left |e^{ik(\vartheta_a-\tau_a E)+\mathcal{Q}(A_a\mathbf{v}_1,\mathbf{v}_2)}\,k^{-j/2}\,P_j(m_x,m_y;
\mathbf{v}_1,\mathbf{v}_2)\right|\le
D_a\,k^{-3j\,(1/6-\epsilon)},
\end{equation}
so that (\ref{eqn:a-th summand expansion}) is indeed an asymptotic expansion.

In view of the previous parity statement, we recover an asymptotic expansion 
at fixed points akin to the one in \cite{bpu2}, by setting $\theta_j=0$ and $\mathbf{v}_j=0$.

\begin{cor}
 \label{cor:scaling asymptotics}
Suppose $(x,y)\in \widetilde{\mathcal{R}}_\chi(f,E)'$. Then the following asymptotic expansion
holds for $k\rightarrow +\infty$:
\begin{eqnarray*}
 \mathcal{G}_k\left(x,y\right)\sim\dfrac{\sqrt{2}}{\|\upsilon_f(m_x)\|}\,\left(\frac{k}{\pi}\right)^{d-1/2}\,
\sum_{a\in \mathcal{A}}\,\mathbf{G}_k^{(a)}
\left(x,y\right),
\end{eqnarray*}
where for each $a$
\begin{eqnarray*}
 \label{eqn:a-th summand expansion 1}
\mathbf{G}_k^{(a)}
\left(x,y\right)\sim e^{ik(\vartheta_a-\tau_a  E)}\,\left[\chi(\tau_a)+\sum_{j=1}^{+\infty}k^{-j}\,A_{aj}(m_x,m_y)\right]; 
\end{eqnarray*}
here 
$A_{aj}=A_{aj}^\chi$ is a $\mathcal{C}^\infty$ function on $\mathcal{R}_\chi(f,E)$.
\end{cor}

Let us consider the special case $x=y$, so that $\mathcal{A}$ indexes the periods of $x\in X_E$
(defined in (\ref{eqn:MEXE})). We can heuristically interpret Corollary \ref{cor:scaling asymptotics}
as saying that the eigensections of the zeroth order T\"{o}plitz operator
$T^\sharp(f)_k$ (defined on page \pageref{pageTsharp}) 
associated to a spectral band of width $O\left(k^{-1}\right)$ about 
$E$ yield a contribution to $\Pi_k(x,x)$ which is $O\left(k^{d-1/2}\right)$.
This should be contrasted with the results in \cite{pao_imrn}, which imply that the eigensections
pertaining to a \lq slowly shrinking\rq\, spectral band of $T^\sharp(f)_k$
of width $O\left(k^{\epsilon-1/2}\right)$ 
about $E$ asymptotically capture all of 
$\Pi_k(x,x)$, up to a negligible contribution; recall that 
$\Pi_k(x,x)=O\left(k^d\right)$ (\S \ref{sctn:HLC}).

It is in order to show how,
by integrating the local rescaled asymptotic expansion in Theorem \ref{thm:scaling asymptotics} 
term by term, one obtains a asymptotic expansion for the trace of $\mathcal{G}_k$,
analogue to the one of Theorem 1.2 of \cite{bpu2};
this expansion probes into the asymptotic clustering, for $k\rightarrow +\infty$,
of the eigenvalues $\lambda_{kj}$
in a spectral band of width $O\left(1\right)$
centered at $kE$. 

To this end, we shall make the stronger assumption that $E$ be a regular value of $f$, which implies that
$\mathcal{R}_\chi(f,E)'=\mathcal{R}_\chi(f,E)$ in (\ref{eqn:defnRchiprime}); 
under this condition, by Proposition \ref{prop:hypersurface periods} below the set 
$\mathrm{Per}^\mathbb{R}_M(M_E)$ (Definition \ref{defn:periods subsets}) 
is a discrete subset of $\mathbb{R}$.
In fact, either $\mathrm{Per}^\mathbb{R}_M(M_E)=\{0\}$ (when there are no closed orbits of $\phi^M$
on $M_E$), or else it is infinite countable (and drifting to infinity).

Let us write \label{pagefixedloci}
$\mathrm{Per}^\mathbb{R}_M(M_E)=\{\sigma_b\}_{b\in \mathcal{B}}$.
For any $b\in \mathcal{B}$, let $M(\sigma_b)\subseteq M$ be the fixed point locus
of $\phi^M_{\sigma_b}$, and let $M(\sigma_b)_l$, \label{page connected components Msigmabl} $1\le l\le n_b$, 
be its connected components.
Each $M(\sigma_b)_l$ is a compact and connected complex submanifold of $M$, of complex dimension 
$d_{bl}=d-c_{bl}$.

Given the unitarity of $\mathrm{d}_m\phi^M_{\sigma_b}$, 
the tangent and normal spaces $T_mM(\sigma_b)$ and
$N_m\big(M(\sigma_b)\big)=T_mM(\sigma)^\perp$ at each $m\in M(\sigma_b)_l$ are given by, respectively,
\begin{equation}
 \label{eqn:tangent and normal}
 T_mM(\sigma_b)_l=\ker\left(\mathrm{d}\phi^M_{\sigma_b}-\mathrm{id}_{T_mM}\right),
\,\,\,\,\,N_mM(\sigma_b)_l=\mathrm{im}\left(\mathrm{d}\phi^M_{\sigma_b}-\mathrm{id}_{T_mM}\right),
\end{equation}
where $\mathrm{im}(h)$ denotes the image of a map $h$.

The $\mathbb{C}$-linear map $\mathrm{id}_{T_mM}-\mathrm{d}\phi^M_{-\sigma_b}$
determines a complex linear automorphism of the holomorphic normal bundle
of $M(\sigma_b)_l$, whose determinant 
\begin{equation}
 \label{defn:dbl}
D(b,l)=:\det \left (\left.\mathrm{id}_{T_mM}-\mathrm{d}\phi^M_{-\sigma_b}\right|_{N_mM(\sigma_b)_l}:
N_mM(\sigma_b)_l\rightarrow N_mM(\sigma_b)_l\right)
\end{equation}
is therefore a non-zero constant 
on $M(\sigma_b)_l$. \label{pageDeterminant}

One can see that if $E$ is a regular value of $f$, then $M_E$ and each $M(\sigma_b)_l$ meet transversally
(\S \ref{sctn:transverse intersections}). 
We shall set 
\begin{equation}
 \label{eqn:defn of MEsigma}
 M_E(\sigma_b)=:M_E\cap M(\sigma_b),\,\,\,\,\,M_E(\sigma_b)_l=:M_E\cap M(\sigma_b)_l;
\end{equation}
this is a submanifold of $M$, perhaps not connected. Its (real) dimension is
$2\,d_{bl}-1$.

For each $b\in \mathcal{B}$ and and $l=1,\ldots,n_b$, let
\begin{equation}
 \label{eqn:defn of XEsigma}
 X_E(\sigma_b)=:\pi^{-1}\big(M_E(\sigma_b)\big),\,\,X_E(\sigma_b)_l=:\pi^{-1}\big(M_E(\sigma_b)_l\big)\subseteq X.
\end{equation}

Then, in view of the connectedness of $M(\sigma_b)_l$, 
there is a unique \label{pagegbl} $g_{bl}=e^{i\vartheta_{bl}}\in S^1$, such that
$x_{\sigma_b}=r_{g_{bl}}(x)$ for every $x\in X(\sigma_b)_l$. 

With this notation, we can now formulate the following global spectral counterpart of Theorem \ref{thm:scaling asymptotics}
(see Theorem 1.2 of \cite{bpu2}):

\begin{thm}
 \label{thm:trace of G_k}
Adopt the hypothesis of Theorem \ref{thm:scaling asymptotics}, and suppose 
in addition that $E$ is a regular value of $f$. Then the following asymptotic expansion holds
for $k\rightarrow +\infty$:
\begin{eqnarray*}
 \lefteqn{\sum_j\widehat{\chi}(kE-\lambda_{kj})}\\
&\sim&\sum_{b\in \mathcal{B}}\,e^{-ik\sigma_b E}\,\sum_{l=1}^{n_b}\left(\frac{k}{\pi}\right)^{d_{bl}-1}\,
\dfrac{e^{ik\vartheta_{bl}}}{D(b,l)}\,\int_{M_E(\sigma_b)_l}\dfrac{1}{\|\upsilon_f(m)\|}\,\mathrm{d}V_{M_E(\sigma_b)_l}(m)\\
&&\cdot \left[\chi(\sigma_b)+\sum_{j\ge 1}k^{-j}\,A_j(b,l)\right],
\end{eqnarray*}
where $A_j(b,l)=A_j(b,l)^\chi\in \mathbb{C}$ are appropriate constants, which vanish for
$\sigma_b\not\in \mathrm{supp}(\chi)$, and $\mathrm{d}V_{M_E(\sigma_b)_l}$ is the Riemannian density on 
$M_E(\sigma_b)_l$.
\end{thm}

Again, the sum over $b$ is finite since $\mathcal{B}\subseteq \mathbb{R}$ is discrete
and $\chi$ has compact support.

It is in order to conclude this introduction by remarking that there is a wider scope for the 
previous results. Namely, consider a compact symplectic manifold $(M,\omega)$,  
with a compatible almost complex structure $J$ \cite{mcds}. 
The definition of compatible Hamiltonian readily generalizes to this almost K\"{a}hler setting.
By way of example, given a Hamiltonian action of a compact Lie group $G$ on $(M,\omega)$,
with moment map $\Phi:M\rightarrow \mathfrak{g}^\vee$,
one can find a $G$-invariant compatible almost complex structure $J$
on $(M,\omega)$ \cite{mcds}. Then any component of $\Phi$ is a compatible Hamiltonian for
the almost K\"{a}hler structure $(M,\omega,J)$.

Suppose now that $(M,\omega)$ is quantizable, and let $(A,h)$ be a quantizing
Hermitian line bundle. Then the theory of \cite{bdm-g}
provides analogues $H^0_J(M,A^{\otimes k})$ of the spaces
of global holomorphic sections of the integrable case, lying in the range of a generalized Szeg\"{o} projector,
which is the first step of a resolution generalizing the $\overline{\partial}$-complex of the integrable case 
(see especially \S A.5 of \cite{bdm-g};
furthermore, a detailed review of this construction is given in \cite{sz}). 
The Hilbert space direct sum of the $H^0_J(M,A^{\otimes k})$'s corresponds in the usual manner to
a generalized Hardy space on the unit circle bundle $X$.

The spaces $H^0_J(M,A^{\otimes k})$
are not fully intrinsic, but depend on some non-canonical choices; it is then not a priori clear whether a compatible
Hamiltonian induces an action on the generalized Hardy space of $(A,h)$. However, 
given as above a compact group action, we may assume that 
$G$ acts on $(A,h)$ preserving all the data involved, and so it acts unitarily on the spaces 
$H^0_J(M,A^{\otimes k})$ (\cite{bdm-g}, \S A.5, Theorem 5.9). The same holds for the contact flow of any component
of the moment map. 
It is thus natural to consider analogues the previous results in this wider picture.

In \cite{sz}, the microlocal description of the Szeg\"{o} kernel in \cite{bs} is generalized to the almost complex 
situation, and based on this near diagonal scaling asymptotics for the equivariant components of the generalized Szeg\"{o}
kernel are provided.
Granting the latter extension, the arguments in the present paper 
go over unchanged to this more general setting.

\section{Preliminaries}

\subsection{Periods}

\label{sctn:periods}

If $m\in M$ is a critical point of $f$, then $\upsilon_f(m)=0$ and $m$ is a fixed point of the Hamiltonian flow
$\phi^M$, that is, $\mathrm{Per}^\mathbb{R}_M(m)=\mathbb{R}$. If on the other hand
$\mathrm{d}_mf\neq 0$ then $\mathrm{Per}^\mathbb{R}_M(m)$ is a discrete subgroup of $\mathbb{R}$. 
In fact, suppose $\tau\in \mathrm{Per}^{R}_M(m)$ and choose a system of local coordinates 
centered at $m$. Then, in additive notation, we have for $\delta\sim 0$:
$$
\phi^M_{\tau+\delta}(m)=\phi^M_{\delta}(m)=m+\delta\,\upsilon_f(m)+O\left(\delta^2\right),
$$
which equals $m$ only for $\delta=0$.

If $\mathrm{d}_mf\neq 0$ and $\mathrm{Per}^{\mathbb{R}}_M(m)$
is not trivial, then
it is isomorphic to $\mathbb{Z}$. More precisely, if $\tau_1\in \mathbb{R}$ is the least positive element
of $\mathrm{Per}^{\mathbb{R}}_M(m)$, then $\mathrm{Per}^{\mathbb{R}}_M(m)=\mathbb{Z}\cdot \tau_1$.
We shall write $\tau_a=a\,\tau_1$.

Given any $x\in X_m=:\pi^{-1}(m)$, there exists a unique $g_1\in S^1$ such that
$x=g_1\cdot \phi^X_{\tau_1}(x)$, that is $x_{\tau_1}=r_{g_1}(x)$.
Thus, if $g_a=:g_1^a$ then
$$
\mathrm{Per}^{R\times S^1}_X(x)=\left\{\left(a\,\tau_1,g_1^a\right)\,:\,a\in \mathbb{Z}\right\}
=\left\{\left(\tau_a,g_a\right)\,:\,a\in \mathbb{Z}\right\}.
$$

Similarly, suppose $m,\,n\in M$ lie in the same $\mathbb{R}$-orbit.
Then $\mathrm{Per}^{\mathbb{R}}_M(n\mapsto m)$ is a non-empty coset of $\mathrm{Per}^{\mathbb{R}}_M(n)$.
If $x\in X_m$ and $y\in Y_n$, then 
$\mathrm{Per}^{\mathbb{R}\times S^1}_X(y\mapsto x)$ is a non-empty coset of
$\mathrm{Per}^{\mathbb{R}\times S^1}_X(y)$. For any $\tau\in \mathrm{Per}^{\mathbb{R}}_M(n\mapsto m)$,
there is a unique $g_\tau\in S^1$ such that 
$x=g_\tau\cdot \phi^X_\tau(y)$. Thus 
$$\tau\in \mathrm{Per}^{\mathbb{R}}_M(n\mapsto m)\mapsto (\tau,g_\tau)\in
\mathrm{Per}^{\mathbb{R}\times S^1}_X(y\mapsto x)$$ is a natural bijection.

Let us now consider the periods on the hypersurface $M_E$. By a standard compactness argument, $\mathrm{Per}^\mathbb{R}_M(M_E)$
is a closed subset of $\mathbb{R}$. 

\begin{prop}
 \label{prop:hypersurface periods}
 Suppose that $E$ is a regular value of $f$.
 Then $\mathrm{Per}^\mathbb{R}_M(M_E)$ is a discrete subset of $\mathbb{R}$.
\end{prop}

\begin{proof}[Proof of Proposition \ref{prop:hypersurface periods}]
 Given $\tau_0\in \mathrm{Per}^\mathbb{R}_M(M_E)$, consider a sequence $\tau_j\in \mathrm{Per}^\mathbb{R}_M(M_E)$,
 $j=1,2,\ldots$, with $\tau_j\rightarrow \tau_0$ for $j\rightarrow+\infty$.
 Thus $\epsilon_j=:\tau_j-\tau_0\rightarrow 0$ for $j\rightarrow +\infty$.
 We aim to prove that $\epsilon_j=0$ for all $j\gg 0$.
 
 For every $j\ge 1$  there exists by definition $m_j\in M_E$ such that $\phi^M_{\tau_j}(m_j)=m_j$.
 By compactness of $M_E$, perhaps after passing to a subsequence we are reduced to the case where
 $m_j\rightarrow m_0$ for some $m_0\in M_E$. By continuity, we then have
 $\phi^M_{\tau_0}(m_0)=m_0$.
 
 Let us fix a unitary isomorphism $T_{m_0}M\cong \mathbb{R}^{2d}$. 
 On an open neighborhood $U\subseteq M$ of $m_0$, we have a geodesic local coordinate system
 centered at $m_0$
 \begin{eqnarray}
  \label{eqn:moving HLC}
\Gamma:\mathbf{v}\in B_{2d}(\mathbf{0},\delta)\subseteq \mathbb{R}^{2d}
\mapsto m_0+\mathbf{v}=:\exp_{m_0}(\mathbf{v})\in U.
 \end{eqnarray}
%

 
\noindent For any $j=1,2,\ldots$ we have $m_j=m_0+\mathbf{v}_j$, for a unique $\mathbf{v}_j\in B_{2d}(\mathbf{0},\delta)
$, and
 $\mathbf{v}_j\rightarrow \mathbf{0}$ for $j\rightarrow+\infty$.
 
 For $\tau\sim 0\in \mathbb{R}$ we have
 \begin{equation}
  \label{eqn:small tau displacement}
  \phi^M_\tau(m)=m+\big(\tau\,\upsilon_f(m)+R_2(\tau)\big),
 \end{equation}
 where, here and in the following, $R_j$ is a smooth function on some open neighborhood of the
 origin in a Euclidean space, vanishing to
 $j$-th order at the origin; generally dependence on additional variables will be left implicit.
 
 Since $\Gamma $ is a system of geodesic local coordinates centered at
$m_0$, and $\phi^M_{\tau_0}:M\rightarrow M$ is a Riemannian isometry fixing $m_0$, we have
\begin{equation}
 \label{eqn:riemannian isometry}
 \phi^M_{\tau_0}(m_0+\mathbf{v}_j)=m_0+A_{\tau_0}^{-1}\mathbf{v}_j,
\end{equation}
where $A_{\tau_0}$ is the matrix representing $\mathrm{d}_{m_0}\phi^M_{-\tau_0}$.

We deduce from (\ref{eqn:small tau displacement}) and (\ref{eqn:riemannian isometry})
that
\begin{eqnarray}
 \label{eqn:mixed displacements}
m_0+\mathbf{v}_j&=& m_j=\phi^M_{\tau_j}(m_j)=\phi^M_{\tau_0+\epsilon_j}(m_0+\mathbf{v}_j)\nonumber \\
&=&\phi^M_{\epsilon_j}\left(m_0+A_{\tau_0}^{-1}\mathbf{v}_j\right)\nonumber\\
 &=&m_0+\left(A_{\tau_0}^{-1}\mathbf{v}_j+\epsilon_j\,\upsilon_f(m_0)+R_2(\epsilon_j)+\epsilon_j\,R_1(\mathbf{v}_j)\right).
\end{eqnarray}
This implies
\begin{equation}
 \label{eqn:v_jA_j}
 \mathbf{v}_j-A_{\tau_0}^{-1}\mathbf{v}_j=\epsilon_j\,\upsilon_f(m)+R_2(\epsilon_j)+\epsilon_j\,R_1(\mathbf{v}_j).
\end{equation}
Since $\upsilon_f(m_0)\in \ker\left(I-A_{\tau_0}^{-1}\right)$ (in local coordinates), and on the other hand
$\mathrm{im}\left(I-A_{\tau_0}^{-1}\right)\perp \ker\left(I-A_{\tau_0}^{-1}\right)$ in view of the unitarity
of $A_{\tau_0}$, we deduce by taking the scalar product with $\upsilon_f(m_0)$ on both
sides of (\ref{eqn:v_jA_j}) that
$$
0=\epsilon_j\,\left[\|\upsilon_f(m)\|^2+R_1(\epsilon_j)+R_1(\mathbf{v}_j)\right],
$$
whence that $\epsilon_j=0$ for any $j\gg 0$.

\end{proof}

Thus, assuming that $E$ is a regular value of $f$, $\mathrm{Per}^\mathbb{R}_M(M_E)$
is a discrete even subset of $\mathbb{R}$. It is trivial (that is, reduced to the origin)
precisely when $\phi^M$ has no closed orbits on the level hypersurface $M_E$.
Otherwise, it is infinite countable, and
can be arranged in increasing order $\ldots<\sigma_{-1}<\sigma_0=0<\sigma_1<\ldots$,
with $\sigma_{-j}=-\sigma_j$. 

We shall write $\mathrm{Per}^\mathbb{R}_M(M_E)=\{\sigma_b\}_{b\in \mathcal{B}}$,\label{pageperiods}
with $\mathcal{B}=\{0\}$ or $\mathcal{B}=\mathbb{Z}$ (and depends on $E$).

\subsection{Transversality issues and normal bundles}
\label{sctn:transverse intersections}

As on page \pageref{pagefixedloci}, for any $\sigma\in \mathbb{R}$ 
let us denote by $M(\sigma)\subseteq M$ the fixed locus of $\phi^M_\sigma$.
Then $M(\sigma)$ is a $\phi^M$-invariant closed complex submanifold of $M$.
Therefore, for any $m\in M(\sigma)$ we have 
\begin{equation}
 \label{eqn:hamiltonian in tangent space}
 \upsilon_f(m),\,J_m\upsilon_f(m)\in T_mM(\sigma).
\end{equation}

The tangent and normal spaces to $M(\sigma)$ are given (with $\sigma=\sigma_b$) by
(\ref{eqn:tangent and normal}). 

On the other hand, since $J(\upsilon_f)$ is the Riemannian gradient of $f$, 
at any $m\in M_E=f^{-1}(E)$ with $\mathrm{d}_mf\neq 0$ the Riemannian normal space to $M_E$ is
\begin{equation}
\label{eqn:normal space energy level}
 N_m(M_E)=\mathrm{span}\big(J_m\upsilon_f(m)\big).
\end{equation}

Clearly, $\sigma\in \mathrm{Per}^\mathbb{R}_M(M_E)$ if and only if $M_E\cap M(\sigma)\neq \emptyset$.

\begin{lem}
 \label{lem:transverse intersections}
For any $\sigma\in \mathrm{Per}^\mathbb{R}_M(M_E)$, $M_E$ and $M(\sigma)$ meet transversally.
\end{lem}

\begin{proof}
 [Proof of Lemma \ref{lem:transverse intersections}]
Suppose $m\in M(\sigma)\cap M_E$. By (\ref{eqn:hamiltonian in tangent space}) and
(\ref{eqn:normal space energy level}), we have
\begin{eqnarray*}
 N_m(M_E)\cap N_m\big(M(\sigma)\big)&=&
\mathrm{span}\big(J_m\upsilon_f(m)\big)\cap T_mM(\sigma)^{\perp_g}\\
&\subseteq&T_mM(\sigma)\cap T_mM(\sigma)^{\perp_g}=\{\mathbf{0}\}.
\end{eqnarray*}
Here $\perp_g$ denotes the Riemannian orthocomplement.
\end{proof}

Thus $M_E(\sigma)=:M_E\cap M(\sigma)$ is $1$-codimensional submanifold of $M(\sigma)$,
and at any $m\in M_E(\sigma)$ we have
$T_mM_E(\sigma)=T_mM_E\cap T_mM(\sigma)$. To be more explicit, let $R_m,\,S_m\subseteq T_mM$
be as in (\ref{eqn:RmSm}); then
\begin{equation}
 \label{eqn:decomposition}
T_mM_E=\mathrm{span}_\mathbb{R}\big(\upsilon_f(m) \big)\oplus S_m,\,\,\,\,\,
T_m M(\sigma)=\mathrm{span}_\mathbb{C}\big(\upsilon_f(m)\big)\oplus 
\widetilde{T}_mM(\sigma),
\end{equation}
where $\widetilde{T}_mM(\sigma)=:S_m\cap T_m M(\sigma)$. 
Thus, with notation (\ref{eqn:defn of MEsigma}),
\begin{equation}
 \label{eqn:M_E(sigma)}
 T_mM_E(\sigma)=\mathrm{span}_\mathbb{R}\big(\upsilon_f(m)\big)\oplus 
\widetilde{T}_mM(\sigma).
\end{equation}

For every $b\in \mathcal{B}$, 
as on page \pageref{page connected components Msigmabl}
let $M(\sigma_b)_l$, $l=1,\ldots,n_b$, be
the connected components of $M(\sigma_b)$, 
For each $l$, let $d_{bl}$ be the complex dimension of $M(\sigma_b)_l$,
and $c_{bl}=d-d_{bl}$ its complex codimension in $M$.
Then by transversality, $M_E(\sigma_b)_l=:M_E\cap M(\sigma_b)_l$ is a submanifold of $M$, 
and if non-empty it has
(real) dimension $2d_{bl}-1$.

For any $m\in M_E(\sigma_b)_l$, in view of transversality the normal space
to $M_E(\sigma_b)_l$ at $m$ is given by the orthogonal direct sum 
\begin{equation}
 \label{eqn:normal space bl}
 N_m\big(M_E(\sigma_b)_l\big)=\mathrm{span}_\mathbb{R}\big(J_m\upsilon_f(m)\big)
 \oplus \mathrm{im}\left(\mathrm{d}\phi^M_{\sigma_b}-\mathrm{id}_{T_mM}\right).
\end{equation}

Now for any $m_0\in M_E(\sigma_b)_l$ we can find an open neighborhood
$U\subseteq M$ of $m$ and a smoothly varying family of geodesic local coordinates on
$M$ centered at points $m\in U$, as in (\ref{eqn:moving HLC}). We can also find a local 
$\mathcal{C}^\infty$ unitary trivialization
of the normal bundle $N_m\big(M_E(\sigma_b)_l\big)\cong \mathbb{R}\oplus\mathbb{C}^{c_{bl}}\cong 
\mathbb{R}^{1+2c_{bl}}$.
In terms of these data, we obtain a local parametrization of a tubular neighborhood of $M_E(\sigma_b)_l$, 
and we shall write this additively in the form \label{page tubular neighborhood}
\begin{equation}
 \label{eqn:local parametrization}
 (m,\lambda,\mathbf{n}):
U'\times (-\delta,\delta)\times B_{2c_{bl}}(\mathbf{0},\delta)\mapsto m+\left(\lambda\,\frac{J_m\upsilon_f(m)}{\|\upsilon_f(m)\|}+
\mathbf{n}\right),
\end{equation}
where $U'$ is an open neighborhood of $m_0$ in $M_E(\sigma_b)_l$, $\delta>0$ is suitably small,
and $\mathbf{n}\in \mathbb{R}^{2c_{bl}}$
is identified with its image in $\mathrm{im}\left(\mathrm{id}_{T_mM}-\mathrm{d}_m\phi^M_\sigma\right)$.


\begin{rem}
 \label{rem:orthogonal decomposition}
In the notation of Definition \ref{defn:tvh components}, in (\ref{eqn:local parametrization}) we have
$\mathbf{n}=\mathbf{n}_\mathrm{h}$.
\end{rem}

\subsection{Szeg\"{o} kernel asymptotics}
\label{sctn:HLC}

It is in order to make some recalls on the off-diagonal scaling asymptotics
of the equivariant Szeg\"{o} kernels $\Pi_k$, expressed in HLC
(\cite{bsz1}, \cite{bsz2}, \cite{sz}).
We refer in particular to \cite{sz} for a detailed discussion and definition of Heisenberg local coordinates, which we have
touched upon on page \pageref{pageHLC}. As is well-known, alternative approaches to the general theme
of near-diagonal asymptotics of Bergman-Szeg\"{o} kernels have been developed by other authors; see for instance 
\cite{mz}, \cite{mm1},
\cite{mm2}, \cite{ch2}, and references therein.

To begin with, given any fixed $C>0$ and $\epsilon\in (0,1/6)$,
we have the following well-known rapid decrease away from slowly shrinking neighborhoods of the inverse image in
$X\times X$ of the diagonal in $M\times M$:

\begin{prop}
 \label{prop:rapid decrease}
Uniformly for
$\mathrm{dist}_M(m_x,m_y)\ge C\,k^{\epsilon-1/2}$, we have $$\Pi_k(x,y)=O\left(k^{-\infty}\right).$$
\end{prop}

This may be derived directly from the representation of $\Pi$ as in FIO in
\cite{bs}, but see also the discussion in \cite{christ} for much more precise statements.

To describe the near-diagonal asymptotics, 
let us consider $x\in X$ and a HLC system centered at $x$; any $y$ close to the $S^1$-orbit of $x$ may be written in the form
$y=x+(\theta,\mathbf{v})$. Let $\psi_2$ be as in Definition \ref{defn:psi2}.
Then for fixed $\epsilon\in (0,1/6)$ and $C>0$ the following asymptotic expansion holds
for $k\rightarrow+\infty$ in the range $\|\mathbf{v}_1\|,\,\|\mathbf{v}_2\|\le C\,k^{\epsilon}$:
\begin{eqnarray}
 \label{eqn:scaling asymptotics rescaled}
 \lefteqn{\Pi_k\left(x+\left(\theta_1,\frac{\mathbf{v}_1}{\sqrt{k}}\right),
 x+\left(\theta_2,\frac{\mathbf{v}_2}{\sqrt{k}}\right)\right)}\\
 &\sim&\left(\frac{k}{\pi}\right)^{d}\,e^{ik(\theta_1-\theta_2)+\psi_2(\mathbf{v}_1,\mathbf{v}_2)}
 \cdot\left[1+\sum_{j=1}^{+\infty}k^{-j/2}\,P_j(x;\mathbf{v}_1,\mathbf{v}_2)\right], \nonumber
\end{eqnarray}
where $\psi_2$ is as in Definition \ref{defn:psi2}, and $P_j(x;\cdot,\cdot)$ is a polynomial,
with coefficients depending $\mathcal{C}^\infty$-wise on $x$, of degree $\le 3j$ and parity $j$
(see for instance the first two lines of (81) in \cite{sz}).
If we write $P_j$ as the sum of its homogeneous components, we obtain
\begin{eqnarray}
 \label{eqn:scaling asymptotics rescaled 1}
 \lefteqn{\Pi_k\left(x+\left(\theta_1,\frac{\mathbf{v}_1}{\sqrt{k}}\right),
 x+\left(\theta_2,\frac{\mathbf{v}_2}{\sqrt{k}}\right)\right)}\\
 &\sim&\left(\frac{k}{\pi}\right)^{d}\,e^{ik(\theta_1-\theta_2)+\psi_2(\mathbf{v}_1,\mathbf{v}_2)}
 \cdot\left[1+\sum_{j=1}^{+\infty}k^{-j/2}\,\sum_{l=0}^{3j}Q_{j,l}(x;\mathbf{v}_1,\mathbf{v}_2)\right], \nonumber
\end{eqnarray}
where now $Q_l(x,\cdot,\cdot)$ is a homogeneous polynomial of degree $l$, and vanishes unless $l-j$ is even.

Passing to unrescaled coordinates, we multiply $\mathbf{v}_j$ by $\sqrt{k}$, and obtain
\begin{eqnarray}
 \label{eqn:scaling asymptotics rescaled 2}
 \lefteqn{\Pi_k\big(x+\left(\theta_1,\mathbf{v}_1\big),
 x+\left(\theta_2,\mathbf{v}_2\right)\right)}\\
 &\sim&\left(\frac{k}{\pi}\right)^{d}\,e^{k[i(\theta_1-\theta_2)+\psi_2(\mathbf{v}_1,\mathbf{v}_2)]}
 \cdot\left[1+\sum_{j=1}^{+\infty}\,\sum_{l=0}^{3j}k^{(l-j)/2}Q_{j,l}(x;\mathbf{v}_1,\mathbf{v}_2)\right], \nonumber
\end{eqnarray}
where only integer powers of $k$ contribute to the inner summation. 
The asymptotic expansion (\ref{eqn:scaling asymptotics rescaled 2}) holds in the range
$\|\mathbf{v}_j\|\le C\,k^{\epsilon-1/2}$.

It is convenient to rearrange the inner summands 
in (\ref{eqn:scaling asymptotics rescaled 2}) in the following 
manner. Let us set, for $j-l$ even, $b=:(l-j)/2$; then $b$ is an integer, and $\lceil -j/2\rceil \le b\le j$. Thus we have
\begin{eqnarray}
 \label{eqn:scaling asymptotics rescaled 3}
 \lefteqn{\Pi_k\big(x+\left(\theta_1,\mathbf{v}_1\big),
 x+\left(\theta_2,\mathbf{v}_2\right)\right)}\\
 &\sim&\left(\frac{k}{\pi}\right)^{d}\,e^{k[i(\theta_1-\theta_2)+\psi_2(\mathbf{v}_1,\mathbf{v}_2)]}
 \cdot\left[1+\sum_{j=1}^{+\infty}\,\sum_{b=\lceil -j/2\rceil}^{j}k^{b}Q_{j,j+2b}(x;\mathbf{v}_1,\mathbf{v}_2)\right] \nonumber\\
&=&\left(\frac{k}{\pi}\right)^{d}\,e^{k[i(\theta_1-\theta_2)+\psi_2(\mathbf{v}_1,\mathbf{v}_2)]}
 \cdot\left[1+\sum_{j=1}^{+\infty}\,A_j(k,x;\mathbf{v}_1,\mathbf{v}_2)\right], \nonumber
\end{eqnarray}
where
\begin{equation}
 \label{eqn:Aj definition}
A_j(k,x;\mathbf{v}_1,\mathbf{v}_2)=\sum_{b=\lceil -j/2\rceil}^{j}k^{b}Q_{j,j+2b}(x;\mathbf{v}_1,\mathbf{v}_2).
\end{equation}

It follows from the previous discussion, and may readily checked directly, 
that (\ref{eqn:scaling asymptotics rescaled 3}) is an asymptotic expansion
for $k\rightarrow+\infty$, in the range $\|\mathbf{v}_j\|\le C\,k^{\epsilon-1/2}$ for $\epsilon\in (0,1/6)$.
Indeed, we have $|A_j|\le C_j\,k^{-3\delta j}$ for $\delta=1/6-\epsilon$.

\subsection{An integral formula for Gutzwiller-T\"{o}plitz kernels}

Equation (\ref{eqn:spectral unitary}) is a representation of $U(\tau)$ in terms of spectral data; on the other hand,
by definition $U(\tau)$ is a pull-back operator restricted to the Hardy space (recall (\ref{eqn:defn Utau})), and therefore
its Schwartz kernel is also given by
\begin{eqnarray}
\label{eqn:unitary and szego}
 U(\tau)_k(x,y)=\sum_j 
 \,e_{kj}(x_\tau)\cdot\overline{e_{kj}(y)}=\Pi_k(x_\tau,y)\,\,\,\,\,\,\,(x,y\in X).
\end{eqnarray}

Similarly, equation (\ref{eqn:spectral gutzwiller}) is a representation of $\mathcal{G}_k$ in terms of spectral data.
We shall now combine (\ref{eqn:spectral gutzwiller}) and (\ref{eqn:unitary and szego}) to obtain
an integral representation of $\mathcal{G}_k$ in terms of $\Pi_k$.
By definition of Fourier transform, we have
\begin{eqnarray}
 \label{eqn:gutzwiller integral}
\mathcal{G}_k(x,y)&=&\int_{-\infty}^{+\infty}e^{-i\tau kE}\,\chi(\tau)\,\left[\sum_{j=1}^{N_k} e^{i\tau \lambda_{kj}}\,
e_{kj}(x)\cdot \overline{e_{kj}(y)}\right]\,\mathrm{d}\tau\nonumber\\
&=&\int_{-\infty}^{+\infty}\,e^{-i\tau kE}\,\chi(\tau)\,U(\tau)_k(x,y)\,\mathrm{d}\tau\nonumber\\
&=&\int_{-\infty}^{+\infty}\,e^{-i\tau kE}\,\chi(\tau)\,\Pi_k(x_\tau,y)\,\mathrm{d}\tau
\end{eqnarray}

\subsection{A handy distance estimate}

Let $B_r(\mathbf{0},\delta)\subseteq \mathbb{R}^r$ be the open ball of radius $\delta>0$ centered
at the origin. Also, let 
$\langle\cdot,\cdot\rangle_{\mathrm{st}}:\mathbb{R}^r\times\mathbb{R}^r\rightarrow\mathbb{R}$ be the
standard Euclidean product, and let $\|\cdot\|$ be the corresponding norm.

\begin{lem}
 \label{lem:distance estimate local chart}
 Let $(M,g)$ be an $r$-dimensional Riemannian manifold, $m\in M$, and let 
 $$
 \eta:B_r(\mathbf{0},\delta)\rightarrow U=\eta\big(B_r(\mathbf{0},\delta)\big)\subseteq M
 $$ 
 be a local coordinate chart, satisfying
 $\eta(\mathbf{0})=m$ and $\eta^*(g)_\mathbf{0}=\langle\cdot,\cdot\rangle_{\mathrm{st}}$.
 Then, perhaps after passing to  a smaller $\delta$ of $m$ we have 
 $$
 2\,\|\mathbf{w}-\mathbf{w}'\|\ge \mathrm{dist}_M\big(\eta(\mathbf{w}),\eta(\mathbf{w}')\big)\ge \frac{1}{2}\,\|\mathbf{w}-\mathbf{w}'\|
 $$ 
 for any $\mathbf{w},\mathbf{w}'\in B_r(\mathbf{0},\delta)$.
\end{lem}

Thus, $\mathrm{d}_\mathbf{0}\eta:(\mathbb{R}^r,\langle\cdot,\cdot\rangle_{\mathrm{st}})
\rightarrow (T_mM,g_m)$ is required to be an isometry.

\begin{proof}
 [Proof of Lemma \ref{lem:distance estimate local chart}]
 Let us denote by $\|\cdot\|_p^*$ the norm at a point $p\in B_r(\mathbf{0},\delta)$ associated to the pulled-back
 Riemannian metric $\eta^*(g)$. Perhaps after passing to a smaller $\delta$ we may assume, by continuity,
 that 
 $2\,\|\cdot\|\ge \|\cdot\|_p^*\ge (1/2)\,\|\cdot\|$ for any $p\in B_r(\mathbf{0},\delta)$.
 
 If $\gamma:[a,b]\rightarrow B_r(\mathbf{0},\delta)$ is any piecewise smooth curve, let us denote
 by  $\ell^*(\gamma)$ its length with
 respect to $\eta^*(g)$, which is the same as the length $\ell(\gamma')$ of $\gamma'=:\eta\circ \gamma:[a,b]\rightarrow M$ 
 with respect to $g$. If $\mathbf{w},\mathbf{w}'\in B_r(\mathbf{0},\delta)$ and $\gamma (a)=\mathbf{w}$,
 $\gamma (b)=\mathbf{w}'$, then 
 \begin{eqnarray}
  \ell^*(\gamma)&=&
  \int_a^b\left\|\dot{\gamma}(\tau)\right\|^*_{\gamma(\tau)}\,\mathrm{d}\tau
  \ge \frac{1}{2}\,\int_a^b\left\|\dot{\gamma}(\tau)\right\|\,\mathrm{d}\tau\nonumber\\
 &\ge& 
  \frac{1}{2}\,\left\|\int_a^b\,\dot{\gamma}(\tau)\,\mathrm{d}\tau\right\| =\frac{1}{2}\,\left\|\mathbf{w}-\mathbf{w}'\right\|.
 \end{eqnarray}

Suppose now $\mathbf{w},\mathbf{w}'\in B_r(\mathbf{0},\delta/3)$. If $\gamma':[a,b]\rightarrow M$ is any piecewise smooth
curve joining $\gamma(\mathbf{w})$ and $\gamma(\mathbf{w}')$, we can distinguish two cases.

If $\gamma'([a,b])\subseteq U$, the previous argument applied to $\gamma=:\eta^{-1}\circ \gamma'$ implies
that $\ell (\gamma')\ge (1/2)\,\left\|\mathbf{w}-\mathbf{w}'\right\|$, where $\ell$ is the length with respect to $g$.

If $\gamma'([a,b])\not\subseteq U$, then applying the same argument to the portion of the curve preceding the 
first boundary point we see that
$\ell(\gamma')\ge (1/2)\big(\delta-(\delta/3)\big)=\delta/3\ge (1/2)\,\|\mathbf{w}-\mathbf{w}'\|$.

Thus at any rate, passing to inf over $\gamma$, we conclude that the right inequality in the statement holds for
any $\mathbf{w},\mathbf{w}'\in B_r(\mathbf{0},\delta/3)$.

To prove the left inequality, consider $\mathbf{w},\mathbf{w}'\in B_r(\mathbf{0},\delta)$
and the curve $\gamma:[0,1]\rightarrow B_r(\mathbf{0},\delta)$ given by
$\gamma(t)=(1-t)\,\mathbf{w}+t\,\mathbf{w}'$; define $\gamma'=:\eta\circ \gamma$.
Then
\begin{eqnarray*}
 \left\|\mathbf{w}-\mathbf{w}'\right\|&=&\int_0^1\left\|\dot{\gamma}(\tau)\right\|\,\mathrm{d}\tau\ge 
 \frac{1}{2}\,\int_0^1\left\|\dot{\gamma}(\tau)\right\|^*_{\gamma(\tau)}\,\mathrm{d}\tau\\
 &=&\frac{1}{2}\,\ell^*(\gamma)=\frac{1}{2}\,\ell(\gamma')\ge 
 \frac{1}{2}\,\mathrm{dist}_M\big(\eta(\mathbf{w}),\eta(\mathbf{w}')\big).
\end{eqnarray*}

\end{proof}

\section{Proof of Theorem \ref{thm:rapid decrease}}

\begin{proof}
 [Proof of Theorem \ref{thm:rapid decrease}]
 By Proposition \ref{prop:rapid decrease}, there exist constants $C_j>0$,
 $j=1,2,\ldots$, such that for any choice of a sequence
 $(x_k,y_k)\in X\times X$ with 
 $$\mathrm{dist}_M\left(\pi(y_k),\pi(x_k)^\chi\right)\ge C\,k^{\epsilon-1/2}$$
 we have 
 \begin{equation}
  \label{eqn:rapid decay szego tau}
\big|\Pi_k(x_\tau,y)\big|<C_j\,k^{-j}
 \end{equation}
 for $k\ge 1$, for any $\tau\in \mathrm{supp}(\chi)$.
 Since $\chi$ is compactly supported, we then deduce from (\ref{eqn:gutzwiller integral}) 
and (\ref{eqn:rapid decay szego tau}) that 
 $\mathcal{G}_k(x_k,y_k)=O\left(k^{-\infty}\right)$ for $k\rightarrow+\infty$.

We are then reduced to assuming that 
$$
 \mathrm{dist}_M\left(n_k,m_k^\chi\right)\le C\,k^{\epsilon-\frac{1}{2}},
 $$
where we have set $m_k=:\pi(x_k)$, $n_k=:\pi(y_k)$.

To begin with, let us fix some small $\epsilon_1>0$ and consider a bump function
$\varrho_1\in\mathcal{C}^\infty_0(\mathbb{R})$ such that $\varrho_1\equiv 1$ on $(-\epsilon_1,\epsilon_1)$
and $\varrho_1\equiv 0$ on $\mathbb{R}\setminus (-2\epsilon_1,2\epsilon_1)$.
Then, by the previous considerations, 
\begin{eqnarray}
 \label{eqn:gutzwiller integral 1}
\mathcal{G}_k(x,y)\sim
\int_{-\infty}^{+\infty}\,e^{-i\tau kE}\,\chi(\tau)\,\varrho_1\big(\mathrm{dist}_M(m_{k\tau},n_{k})\big)\,
\Pi_k(x_{k\tau},y_k)\,\mathrm{d}\tau,
\end{eqnarray}
where $\sim$ means \lq has the same asymptotics as\rq.

Since our analysis is local in the neighborhood of $y_k$, we can find a (smoothly varying) HLC system
centered at $y_k$, and in terms of the latter we shall write
$$
x_{k\tau}=y_k+\big(\theta_k(\tau),\mathbf{v}_k(\tau)\big),\,\,\,\,\,\,
m_{k\tau}=n_k+\mathbf{v}_k(\tau),
$$
where the second equality should be interpreted in terms of the associated adapted local coordinates on $M$ 
(see the discussion on page \pageref{pageHLC}). Since HLC are isometric at the origin, we have for sufficiently
small $\epsilon_1$ that
\begin{equation}
 \label{eqn:estimate distance norm}
\frac{1}{2}\,\|\mathbf{v}_k(\tau)\|\le \mathrm{dist}_M(m_{k\tau},n_k)\le 2\,\|\mathbf{v}_k(\tau)\|.
\end{equation}
In particular, we have 
\begin{equation}
 \label{eqn:bound on v_k}
\|\mathbf{v}_k(\tau)\|\le 2C\,k^{\epsilon-1/2}.
\end{equation}

Let us then choose another bump function $\varrho\in \mathcal{C}^\infty_0\left(\mathbb{R}^{2d}\right)$, 
such that $\varrho\equiv 1$ on $B_{2d}(\mathbf{0},2C)$
and $\varrho\equiv 0$ on $\mathbb{R}\setminus B_{2d}(\mathbf{0},3C)$. 
If $\varrho \left(k^{1/2-\epsilon}\,\mathbf{v}_k(\tau)\right)\neq 1$, then
$\|\varrho_k(\tau)\|\ge 2C\,k^{\epsilon-1/2}$, and so by (\ref{eqn:estimate distance norm}) we have
\begin{equation}
 \label{eqn:estimate distance norm 1}
\mathrm{dist}_M(m_{k\tau},n_k)\ge C\,k^{\epsilon-1/2}.
\end{equation}

Again, we conclude therefore that $\Pi_k(x_{k\tau},y_k)=O\left(k^{-\infty}\right)$ in the range where 
$\varrho \left(k^{1/2-\epsilon}\,\mathbf{v}_k(\tau)\right)\neq 1$. This implies
\begin{eqnarray}
 \label{eqn:gutzwiller integral 2}
\lefteqn{\mathcal{G}_k(x,y)}\\
&\sim&
\int_{-\infty}^{+\infty}\,e^{-i\tau kE}\,\chi(\tau)\,\varrho_1\big(\mathrm{dist}_M(m_{k\tau},n_{k})\big)\,
\varrho \left(k^{1/2-\epsilon}\,\mathbf{v}_k(\tau)\right)\,
\Pi_k(x_{k\tau},y_k)\,\mathrm{d}\tau\nonumber\\
&=&
\int_{-\infty}^{+\infty}\,e^{-i\tau kE}\,\chi(\tau)\,
\varrho \left(k^{1/2-\epsilon}\,\mathbf{v}_k(\tau)\right)\,
\Pi_k\Big(y_k+\big(\theta_k(\tau),\mathbf{v}_k(\tau)\big),y_k\Big)\,\mathrm{d}\tau.\nonumber
\end{eqnarray}
The latter equality follows for $k\gg 0$ because $\varrho \left(k^{1/2-\epsilon}\,\mathbf{v}_k(\tau)\right)\neq 0$
implies, in view of (\ref{eqn:estimate distance norm}), that
$$
\mathrm{dist}_M(m_{k\tau},n_k)\le 2\,\|\mathbf{v}_k(\tau)\|\le 6\,C\,k^{\epsilon-1/2}\ll \epsilon_1\,\,\,
\Longrightarrow\,\,\,\varrho_1\big(\mathrm{dist}_M(m_{k\tau},n_{k})\big)=1.
$$

Now we have, in view of (\ref{eqn:scaling asymptotics rescaled 3}),
\begin{eqnarray}
\label{eqn:scaling asymptotics rescaled 4}
 \lefteqn{\Pi_k\Big(y_k+\big(\theta_k(\tau),\mathbf{v}_k(\tau)\big),y_k\Big)} \\
&\sim&\left(\frac{k}{\pi}\right)^{d}\,e^{ik[\theta_k(\tau)+\frac{i}{2}\,\|\mathbf{v}_k(\tau)\|^2]}
 \cdot\left[1+\sum_{j=1}^{+\infty}\,A_j(k,x;\mathbf{v}_k(\tau),\mathbf{0})\right]. \nonumber
\end{eqnarray}

Inserting (\ref{eqn:scaling asymptotics rescaled 4}) in (\ref{eqn:gutzwiller integral 2}), we obtain
\begin{eqnarray}
 \label{eqn:gutzwiller integral 3}
\mathcal{G}_k(x_k,y_k)
&\sim&\left(\frac{k}{\pi}\right)^{d}\,
\int_{-\infty}^{+\infty}\,e^{ik\Gamma(x_k,y_k,\tau)}\,\chi(\tau)\,
\varrho \left(k^{1/2-\epsilon}\,\mathbf{v}_k(\tau)\right)\nonumber\\
&&\cdot 
\left[1+\sum_{j=1}^{+\infty}\,A_j(k,x;\mathbf{v}_k(\tau),\mathbf{0})\right]\,\mathrm{d}\tau,
\end{eqnarray}
where 
\begin{equation}
 \label{eqn:phase}
\Gamma(x_k,y_k,\tau)=:\theta_k(\tau)+\frac{i}{2}\,\|\mathbf{v}_k(\tau)\|^2-\tau\,E,
\end{equation}
and the asymptotic expansion can be integrated term by term.

To proceed further, we need to make $\Gamma$ more explicit. Let us fix $\tau_0$ in the support of the integrand,
and make the change of variables $\tau\rightsquigarrow \tau_0+\tau$ in a small neighborhood of $\tau_0$, so that
now $\tau\sim 0$. 
Then by Corollary 2.2 of \cite{pao_ijm_2012} we have
\begin{eqnarray*}
 \lefteqn{y_k+\big(\theta_k(\tau_0+\tau),\mathbf{v}_k(\tau_0+\tau)\big)
=\phi^X_{-\tau-\tau_0}(x_k)}\\
&=&\phi^X_{-\tau}\left(\phi^X_{-\tau_0}(x_k)\right)=\phi^X_{-\tau}\Big(y_k+\big(\theta_k(\tau_0),\mathbf{v}_k(\tau_0)\big)\Big)\\
&=&y_k+\Big(\theta_k(\tau_0)+\tau\,f(n_k)+\tau\,\omega_{n_k}\big(\upsilon_f(n_k),\mathbf{v}_k(\tau_0)\big)+R_3\big(\tau,\mathbf{v}_k(\tau_0)\big),\\
&&\mathbf{v}_k(\tau_0)-\tau\,\upsilon_f(n_k)+R_2\big(\tau,\mathbf{v}_k(\tau_0)\big)\Big).
\end{eqnarray*}

Therefore, we obtain
\begin{eqnarray}
 \label{eqn:phase Gamma}
\Gamma(x_k,y_k,\tau_0+\tau)&=&\theta_k(\tau_0)+\tau\,\big[f(n_k)+\omega_{n_k}\big(\upsilon_f(n_k),\mathbf{v}_k(\tau_0)\big)-E\big]
+R_3\big(\tau,\mathbf{v}_k(\tau_0)\big)\nonumber\\
&&+\frac{i}{2}\,\big\|\mathbf{v}_k(\tau_0)-\tau\,\upsilon_f(n_k)+R_2\big(\tau,\mathbf{v}_k(\tau_0)\big)\big\|^2\nonumber\\
&=&\theta_k(\tau_0)+\frac{i}{2}\,\big\|\mathbf{v}_k(\tau_0)\big\|^2\nonumber\\
&&+\tau\,\Big[f(n_k)+\omega_{n_k}\big(\upsilon_f(n_k),\mathbf{v}_k(\tau_0)\big)-E-i\,\big\langle\mathbf{v}_k(\tau_0)\upsilon_f(n_k)\big\rangle\Big]
\nonumber\\
&&+\frac{i}{2}\,\tau^2\,\big\|\upsilon_f(n_k)\big\|^2+R_3\big(\tau,\mathbf{v}_k(\tau_0)\big).
\end{eqnarray}
Since the $R_3$ term on the first line is real, the first equality in (\ref{eqn:phase Gamma}) implies $\Im \Gamma \ge 0$.

In view of (\ref{eqn:Aj definition}), applied with $\tau_0+\tau$ in place of $\tau$,
the $j$-th summand in the asymptotic expansion for the amplitude in the integrand in (\ref{eqn:gutzwiller integral 3})
is given by 
\begin{eqnarray}
 \label{eqn:jth summand expansion}
\lefteqn{\mathcal{A}_j(x_k,y_k,\tau_0+\tau)}\\
&=:&\left(\frac{k}{\pi}\right)^{d}\,\chi(\tau_0+\tau)\,\varrho \left(k^{1/2-\epsilon}\,\mathbf{v}_k(\tau_0+\tau)\right)
\,A_j(k,x;\mathbf{v}_k(\tau_0+\tau),\mathbf{0})\nonumber\\
&=&\left(\frac{k}{\pi}\right)^{d}\,\chi(\tau_0+\tau)\,\varrho \left(k^{1/2-\epsilon}\,\Big[\mathbf{v}_k(\tau_0)-\tau\,\upsilon_f(n_k)+R_2\big(\tau,\mathbf{v}_k(\tau_0)\big)\Big]\right)
\nonumber\\
&&\cdot\sum_{b=\lceil -j/2\rceil}^{j}k^{b}Q_{j+2b}\Big(y_k;
\mathbf{v}_k(\tau_0)-\tau\,\upsilon_f(n_k)+R_2\big(\tau,\mathbf{v}_k(\tau_0)\big),\mathbf{0}\Big).\nonumber
\end{eqnarray}
This is defined for $j\ge 1$, but for convenience we shall extend to $j=0$ by summing over $j\ge 0$ with $A_0=1$ in (\ref{eqn:gutzwiller integral 3}).

In view of the factor $k^{1/2-\epsilon}$ appearing in (\ref{eqn:jth summand expansion}), and the considerations following
(\ref{eqn:Aj definition}), we have for $a\ge 0$
\begin{eqnarray}
 \label{eqn:estimate derivative amplitude}
\left.\frac{\partial^a \mathcal{A}_j}{\partial\tau^a}(x_k,y_k,\tau)\right|_{\tau=\tau_0}&=&
\left.\frac{\partial^a \mathcal{A}_j}{\partial\tau^a}(x_k,y_k,\tau_0+\tau)\right|_{\tau=0}\nonumber\\
&=&O\left(k^{d-3j(1/6-\epsilon)+a(1/2-\epsilon)}\right).
\end{eqnarray}
On the other hand, given (\ref{eqn:phase Gamma}) we have
\begin{eqnarray}
 \label{eqn:derivative amplitude}
\lefteqn{\left.\frac{\partial}{\partial\tau}\Gamma(x_k,y_k,\tau)\right|_{\tau=\tau_0}=
\left.\frac{\partial}{\partial\tau}\Gamma(x_k,y_k,\tau_0+\tau)\right|_{\tau=0}}\nonumber\\
&=&f(n_k)+\omega_{n_k}\big(\upsilon_f(n_k),\mathbf{v}_k(\tau_0)\big)-E-i\,\big\langle\mathbf{v}_k(\tau_0)\upsilon_f(n_k)\big\rangle
+R_2\big(\mathbf{v}_k(\tau_0)\big)\nonumber\\
&=&f(n_k)-E+R_1\big(\mathbf{v}_k(\tau_0)\big).
\end{eqnarray}

Furthermore, in view of (\ref{eqn:bound on v_k}), for $k\gg 0$ 
the $R_1$ term on the last line of (\ref{eqn:derivative amplitude}) is bounded 
by $C_1C\,k^{\epsilon-1/2}$, where $C_1$ only depends on $f$ and $E$.

If then $\big|f(n_k)-E\big|\ge 2C_1C\,k^{\epsilon-1/2}$, we conclude from (\ref{eqn:derivative amplitude}) that 
for any $\tau\in \mathrm{supp}(\chi)$ and $k\gg 0$ one has
\begin{eqnarray}
 \label{eqn:derivative amplitude bound}
\left|\frac{\partial}{\partial\tau}\Gamma(x_k,y_k,\tau)\right|\ge C_1C\,k^{\epsilon-1/2}.
\end{eqnarray}
Let us consider the $k$-dependent differential operator on $\mathbb{R}$
$$
L_k=:\frac{1}{\partial_\tau\Gamma(x_k,y_k,\tau)}\,\frac{\partial}{\partial\tau}=
\frac{1}{f(n_k)-E+R_1\big(\mathbf{v}_k(\tau)\big)}\,\frac{\partial}{\partial\tau},
$$
where of course $\partial_\tau\Gamma=\partial\Gamma/\partial\tau$.

Then 
\begin{equation}
 \label{eqn:key step integration by parts}
e^{ik\Gamma(x_k,y_k,\tau)}=\frac{1}{ik}\,L_k\left(e^{ik\Gamma(x_k,y_k,\tau)}\right).
\end{equation}
Iterating $r$ times integration by parts, 
we obtain for the $j$-th summand in (\ref{eqn:gutzwiller integral 3}) (with $A_0=1$)
\begin{eqnarray}
 \label{eqn:integration by parts}
\lefteqn{
\int_{-\infty}^{+\infty}\,e^{ik\Gamma(x_k,y_k,\tau)}\,\mathcal{A}_j(x_k,y_k,\tau)\,\mathrm{d}\tau}\\
&=&\frac{1}{ik}\,
\int_{-\infty}^{+\infty}\,L_k\left(e^{ik\Gamma(x_k,y_k,\tau)}\right)\,\mathcal{A}_j(x_k,y_k,\tau)\,\mathrm{d}\tau\nonumber\\
&=&-\frac{1}{ik}\,
\int_{-\infty}^{+\infty}\,e^{ik\Gamma(x_k,y_k,\tau)}\widetilde{L}_k\left(\mathcal{A}_j\right)(x_k,y_k,\tau)\,\mathrm{d}\tau\nonumber\\
&=&\left(\frac{i}{k}\right)^r\,\int_{-\infty}^{+\infty}\,e^{ik\Gamma(x_k,y_k,\tau)}\,\widetilde{L}_k^r\big(\mathcal{A}_j\big)(x_k,y_k,\tau)\,
\mathrm{d}\tau,
\nonumber
\end{eqnarray}
where
$$
\widetilde{L}_k(h)=:\frac{\partial}{\partial\tau}\left(\frac{1}{\partial_\tau\Gamma(x_k,y_k,\tau)}\cdot h\right)\,\,\,\,\,
\,\,\,\,\,\,\,\,\,\left(h\in \mathcal{C}^\infty(\mathbb{R})\right).
$$

In general, given $\gamma \in \mathcal{C}^\infty(\mathbb{R})$ not vanishing identically,
let us denote by $U_\gamma\subseteq\mathbb{R}$ the open set where $\gamma\neq 0$, and
define 
$$
L_\gamma(h)=:\frac{\partial}{\partial\tau}\left(\frac{1}{\gamma}\cdot h\right)\,\,\,\,\,
\,\,\,\,\,\,\,\,\,\left(h\in \mathcal{C}^\infty(U_\gamma)\right),
$$
viewed as an operator $\mathcal{C}^\infty(U_\gamma)\rightarrow\mathcal{C}^\infty(U_\gamma)$.
\begin{lem}
\label{lem:decomposition L^r}
For any integer $r\ge 0$, 
\begin{equation}
 \label{eqn:decomposition L^r}
 \widetilde{L}_k^r\big(h\big)=\sum_{a+b\le 2r}P_{a,b}(\gamma)\,
\frac{\partial_\tau^a h}{\gamma^b}\,\,\,\,\,\,\,\,\,\,\,\left(h\in \mathcal{C}^\infty(U_\gamma)\right),
\end{equation}
where $P_{a,b}(\gamma)$ is a differential operator of degree $d_{a,b}$ satisfying 
$d_{ab}+a\le r$.
\end{lem}

\begin{proof}[Proof of Lemma \ref{lem:decomposition L^r}]
The claim is obvious for $r=0$, and for $r=1$ we have
\begin{eqnarray*}
 L_\gamma(h)=\frac{\partial}{\partial\tau}\left(\frac{1}{\gamma}\cdot h\right)=
\dfrac{\gamma\,\partial_\tau\,h-h\partial_\tau\gamma}{\gamma^2}=\frac{\partial_\tau\,h}{\gamma}
 -\partial_\tau\gamma\,\frac{h}{\gamma^2}.
\end{eqnarray*}
Let us assume that the claim is true for $r\ge 1$. Then
\begin{eqnarray*}
 \lefteqn{\widetilde{L}_k^{r+1}\big(h\big)=\frac{\partial}{\partial_\tau}\left(\sum_{a+b\le 2r}P_{a,b}(\gamma)\,
\frac{\partial_\tau^a h}{\gamma^{b+1}}\right)}\\
&=&\sum_{a+b\le 2r}\left\{\dfrac{1}{\gamma^{b+1}}\,\left[\big(\partial_\tau P_{a,b}(\gamma)\big)\cdot\partial_\tau^a h +P_{a,b}(\gamma)\,
\partial_\tau^{a+1} h\right]-\dfrac{(b+1)}{\gamma^{b+2}}\,\partial_\tau\gamma\,P_{a,b}(\gamma)\,\partial_\tau^a h\right\},
\end{eqnarray*}
which clearly implies the statement for $r+1$.
\end{proof}

Let us apply Lemma \ref{lem:decomposition L^r} with $h=\mathcal{A}_j$ and $\gamma=\partial\Gamma_k$.
In view of (\ref{eqn:estimate derivative amplitude}) and (\ref{eqn:derivative amplitude bound}) with obtain
for each of the summands in (\ref{eqn:decomposition L^r}) the estimate
\begin{eqnarray}
 \label{eqn:estimate summand}
 \left|P_{a,b}\big(\partial\Gamma_k\big)\,
\frac{\partial_\tau^a \mathcal{A}_j}{(\partial\Gamma_k)^b}\right|&\le& 
C_{ab}\,k^{d-3j(1/6-\epsilon)+(a+b)(1/2-\epsilon)}\nonumber\\
&\le& C_{ab}\,k^{d-3j(1/6-\epsilon)+r-2r\epsilon)}.
\end{eqnarray}

Inserting this estimate in (\ref{eqn:integration by parts}), we deduce that for every $r$ and $j$ there exist
constants $C_{rj}>0$ such that
$$
\left|\int_{-\infty}^{+\infty}\,e^{ik\Gamma(x_k,y_k,\tau)}\,\mathcal{A}_j(x_k,y_k,\tau)\,\mathrm{d}\tau\right|
\le C_{jr}\,k^{d-3j(1/6-\epsilon)-2r\epsilon}.
$$

We have thus proved the following. Given any $C>0$ and $\epsilon\in (0,1/6)$, 
we have $\mathcal{G}_k(x_k,y_k)=O\left(k^{-\infty}\right)$ uniformly for all sequences $(x_k,y_k)\in X\times X$
such that either

\begin{description}
 \item[A):] $\mathrm{dist}_M\left(n_k,m_k^\chi\right)\ge C\,k^{\epsilon-1/6}$, or
 \item[B):] $|f(n_k)-E|\ge DC\,k^{\epsilon-1/6}$, for an appropriate constant $D>0$,
  depending only on $f$ and $E$.
\end{description}

Now suppose that:

$$
\max\left\{\mathrm{dist}_M\left(n_k,m_k^\chi\right),\,|f(n_k)-E|\right\}\ge C\,k^{\epsilon-\frac{1}{2}}.
$$
In order to prove that $\mathcal{G}_k(x_k,y_k)=O\left(k^{-\infty}\right)$, after passing to subsequences we
are reduced to assuming that either 
\begin{enumerate}
 \item $\mathrm{dist}_M\left(n_k,m_k^\chi\right)\ge (C/D)\,k^{\epsilon-1/6}$, or else
 \item $\mathrm{dist}_M\left(n_k,m_k^\chi\right)\le (C/D)\,k^{\epsilon-1/6}$ and $|f(m_k)-E|\ge C\,k^{\epsilon-1/6}$.
\end{enumerate}
The statement then follows by replacing $C$ by $C/D$ in A) and B) above.

Finally, let us assume that a sequence $(x_k,y_k)\in X\times X$ satisfies
\begin{equation}
 \label{eqn:hypothesis theorem}
\max\left\{\mathrm{dist}_M\left(n_k,m_k^\chi\right),\,|f(n_k)-E|,\,|f(m_k)-E|\right\}\ge C\,k^{\epsilon-\frac{1}{2}}.
\end{equation}

There is a constant $D'\ge 1$ such that $|f(m)-f(n)|\le D'\,\mathrm{dist}_M(m,n)$ for all $m,n\in M$. Since
$f$ is $\phi^M$ invariant, we deduce that in fact
\begin{equation}
 \label{eqn:estimate f distance}
|f(m)-f(n)|\le D'\,\mathrm{dist}_M\left(n,m^\chi\right)\,\,\,\,\,\,\,(m,n\in M).
\end{equation}

Fix $R>0$ a large positive constant. After passing to subsequences, we fall under one of the following the three cases:

\begin{description}
 \item[A):] $\mathrm{dist}_M\left(n_k,m_k^\chi\right)\ge (C/RDD')\,k^{\epsilon-1/6}$, or
 \item[B):] $\mathrm{dist}_M\left(n_k,m_k^\chi\right)\le (C/RDD')\,k^{\epsilon-1/6}$
and $|f(n_k)-E|\ge (C/RD')\,k^{\epsilon-1/6}$, or
\item[C):] $\mathrm{dist}_M\left(n_k,m_k^\chi\right)\le (C/RDD')\,k^{\epsilon-1/6}$, $|f(m_k)-E|\ge C\,k^{\epsilon-1/6}$.
\end{description}

If either A) or B) hold, we are reduced to the case above, with $C$ replaced by $C/RD'$. 
Suppose on the other hand that C) holds. In view of (\ref{eqn:estimate f distance}), we get:
\begin{eqnarray*}
 |f(n_k)-E|&\ge& |f(m_k)-E|-|f(n_k)-f(m_k)|\\
&\ge&C\,k^{\epsilon-1/6}-D'\,\mathrm{dist}_M\left(n_k,m_k^\chi\right)\ge \left[C-\frac{C}{RD}\right]\,k^{\epsilon-1/6}
\end{eqnarray*}
Since for $R\gg 0$ we have
$$
C-\frac{C}{RD}\ge D\cdot \frac{C}{RDD'}=\frac{C}{RD'},
$$
we are reduced to the previous case. 

Hence in each case
$\mathcal{G}_k(x_k,y_k)=O\left(k^{-\infty}\right)$, and the proof of Theorem \ref{thm:rapid decrease}
is complete.
\end{proof}

\section{Proof of Theorem \ref{thm:scaling asymptotics}}

\begin{proof}[Proof of Theorem \ref{thm:scaling asymptotics}]
 
We shall first work in unrescaled coordinates and set
$x+\mathbf{v}_1=x+(0,\mathbf{v}_1)$, $y+\mathbf{v}_2=y+(0,\mathbf{v}_2)$, where $\|\mathbf{v}_1\|,\,
\|\mathbf{v}_2\|\le C\,k^{\epsilon-1/2}$.
We get
\begin{eqnarray}
 \label{eqn:gutzwiller integral 4}
\lefteqn{\mathcal{G}_k\big(x+(\theta_1,\mathbf{v}_1),y+(\theta_2,\mathbf{v}_2)\big)
=
e^{ik(\theta_1-\theta_2)}\,\mathcal{G}_k\big(x+\mathbf{v}_1,y+\mathbf{v}_2\big)}\nonumber\\
&=&e^{ik(\theta_1-\theta_2)}\,
\int_{-\infty}^{+\infty}\,e^{-i\tau kE}\,\chi(\tau)\,\Pi_k\big((x+\mathbf{v}_1)_\tau,y+\mathbf{v}_2\big)
\,\mathrm{d}\tau.
\end{eqnarray}
 
 Let us set $m=m_x$, $n=m_y$. 
 
 \begin{prop}
  \label{prop:localizzazione in tau}
  There is a constant $\beta=\beta(x,y)>0$ such that the the contribution
  of the locus in $\mathbb{R}$ where
  \begin{equation}
   \label{eqn:bound on distance period}
   \mathrm{dist}_\mathbb{R}\left(\tau,\mathrm{Per}^\mathbb{R}_M(n\mapsto m)\right)\ge \beta C\,k^{\epsilon-1/2} 
  \end{equation}
to the asymptotics for $k\rightarrow +\infty$
  of (\ref{eqn:gutzwiller integral 4}) is rapidly decreasing. If $E$ is a regular value of $f$, then $\beta$ can be chosen uniformly
  on $\mathcal{R}(f,E)$.
 \end{prop}

 \begin{proof}[Proof of Proposition \ref{prop:localizzazione in tau}]
 Given any $\delta>0$ for $k\gg 0$ we have (in adapted local coordinates, see page \pageref{pageHLC})
 \begin{equation}
  \label{eqn:estimate distance norm 1-1}
  (1-\delta)\,\|\mathbf{v}_1\|\le \mathrm{dist}_M\big(m_x,m_x+\mathbf{v}_1\big)
  \le (1+\delta)\,\|\mathbf{v}_1\|,
 \end{equation}
and similarly for $m_y$ and $\mathbf{v}_2$. Therefore, using that $\phi^M_\tau$ is a Riemannian isometry, we get
\begin{eqnarray}
 \label{eqn:estimate distance norm 2}
 \lefteqn{\mathrm{dist}_M(m_{\tau},n)}\nonumber\\
 &\le& \mathrm{dist}_M\big(m_{\tau},(m+\mathbf{v}_1)_\tau)+
 \mathrm{dist}_M\big((m+\mathbf{v}_1)_\tau,n+\mathbf{v}_2)+\mathrm{dist}_M\big(n+\mathbf{v}_2,n)\nonumber\\
 &=&\mathrm{dist}_M\big(m,m+\mathbf{v}_1)+
 \mathrm{dist}_M\big((m+\mathbf{v}_1)_\tau,n+\mathbf{v}_2)+\mathrm{dist}_M\big(n+\mathbf{v}_2,n)\nonumber\\
 &\le&C\,(1+\delta)\,\|\mathbf{v}_1\|+
 \mathrm{dist}_M\big((m+\mathbf{v}_1)_\tau,n+\mathbf{v}_2)+C\,(1+\delta)\,\|\mathbf{v}_2\|\nonumber\\
 &\le&2C\,(1+\delta)\,k^{\epsilon-1/2}+
 \mathrm{dist}_M\big((m+\mathbf{v}_1)_\tau,n+\mathbf{v}_2).
\end{eqnarray}

On the other hand, if (\ref{eqn:bound on distance period}) holds then $\tau=\sigma+\tau'$,
where $\sigma\in \mathrm{Per}^\mathbb{R}_M(n\mapsto m)$ and $|\tau'|\ge \beta C\,k^{\epsilon-1/2}$.
Thus $m_\tau=\phi^M_{-\tau'}(n)$, whence
\begin{equation}
 \label{eqn:estimate distance norm 3}
 \mathrm{dist}_M(m_{\tau},n)\ge |\tau'|\,(1-\delta)\,\|\upsilon_f(m)\|\ge \beta C\,(1-\delta)\,\|\upsilon_f(m)\|\,
 k^{\epsilon-1/2}.
\end{equation}

 From (\ref{eqn:estimate distance norm 2}) and (\ref{eqn:estimate distance norm 3}) we obtain that if
 $ \mathrm{dist}_\mathbb{R}\left(\tau,\mathrm{Per}^\mathbb{R}_M(n\mapsto m)\right)\ge \beta C\,k^{\epsilon-1/2} $
 then
 \begin{eqnarray}
  \label{eqn:estimate distance norma 4}
  \mathrm{dist}_M\big((m+\mathbf{v}_1)_\tau,n+\mathbf{v}_2)&\ge&\mathrm{dist}_M(m_{\tau},n)-2C\,(1+\delta)\,k^{\epsilon-1/2}\\
  &\ge&\beta C\,(1-\delta)\,\|\upsilon_f(m)\|\,
 k^{\epsilon-1/2}-2C\,(1+\delta)\,k^{\epsilon-1/2}\nonumber\\
 &=&C\,\big[\beta\,(1-\delta)\,\|\upsilon_f(m)\|-2\,(1+\delta)\big]\,k^{\epsilon-1/2}.\nonumber
 \end{eqnarray}
Thus it suffices to choose $\beta>3/\|\upsilon_f(m)\|$, say, to conclude that 
in (\ref{eqn:gutzwiller integral 4}) we have
$\Pi_k\big((x+\mathbf{v}_1)_\tau,y+\mathbf{v}_2\big)=O\left(k^{-\infty}\right)$, for $k\rightarrow +\infty$,
uniformly where (\ref{eqn:bound on distance period}) is satisfied.

 \end{proof}
 
 Proposition \ref{prop:localizzazione in tau} means the following. First of all, let $\varrho_0\in \mathcal{C}^\infty(\mathbb{R})$
\label{page defn varrho_0} be $\ge 0$, $\equiv 1$ on $(-1,1)$ and $\equiv 0$ on $\mathbb{R}\setminus (-2,2)$. Given some very small $\epsilon_1>0$, let us set 
$\varrho(\cdot)=:\sum_a \varrho_0\big((\cdot -\tau_a)/\epsilon_1\big)$ (recall (\ref{eqn:period set})). 
Then $\varrho$ is supported in a small neighborhood
of $\mathrm{Per}^\mathbb{R}_M(n\mapsto m)$, namely 
\begin{equation}
 \label{eqn:support varrho}
 \mathrm{supp}(\varrho)\subseteq \bigcup_{a\in \mathcal{A}}(a-2\epsilon_1,a+2\epsilon_1),
\end{equation}
and 
\begin{equation}
 \label{eqn:support varrho=1}
 \varrho\equiv 1\,\,\,\mathrm{on}\,\,\, \bigcup_{a\in \mathcal{A}}(a-\epsilon_1,a+\epsilon_1),
\end{equation}
Thus first of all we have:
\begin{eqnarray}
 \label{eqn:gutzwiller integral 5}
\lefteqn{\mathcal{G}_k\big(x+(\theta_1,\mathbf{v}_1),y+(\theta_2,\mathbf{v}_2)\big)}\\
&=&e^{ik(\theta_1-\theta_2)}\,
\int_{-\infty}^{+\infty}\,e^{-i\tau kE}\,\varrho(\tau)\,
\chi(\tau)\,\Pi_k\big((x+\mathbf{v}_1)_\tau,y+\mathbf{v}_2\big)
\,\mathrm{d}\tau\nonumber\\
&\sim&e^{ik(\theta_1-\theta_2)}\,
\sum_a\int_{-\infty}^{+\infty}\,e^{-i\tau kE}\,\varrho_0\big((\tau -\tau_a)/\epsilon_1\big)\,
\chi(\tau)\,\Pi_k\big((x+\mathbf{v}_1)_\tau,y+\mathbf{v}_2\big)
\,\mathrm{d}\tau.\nonumber
\end{eqnarray}
By the Proposition, only a rapidly decreasing contribution to the asymptotics is lost, if the integrand is further multiplied
by $\varrho_0\left(k^{1/2-\epsilon}\,(\tau-\tau_a)/\beta C\right)$; in addition, for $k\gg 0$
this factor is supported where $\varrho_0\big((\tau -\tau_a)/\epsilon_1\big)=1$. Therefore, we conclude
\begin{eqnarray}
 \label{eqn:gutzwiller integral 6}
\lefteqn{\mathcal{G}_k\big(x+(\theta_1,\mathbf{v}_1),y+(\theta_2,\mathbf{v}_2)\big)\sim e^{ik(\theta_1-\theta_2)}}\nonumber\\
&&\cdot 
\sum_a\int_{-\infty}^{+\infty}\,e^{-i\tau kE}\,\varrho_0\left(k^{1/2-\epsilon}\,(\tau-\tau_a)/\beta C\right)\,
\chi(\tau)\,\Pi_k\big((x+\mathbf{v}_1)_\tau,y+\mathbf{v}_2\big)
\,\mathrm{d}\tau\nonumber\\
&=&e^{ik(\theta_1-\theta_2)}\,\sum_a\,\mathcal{G}_k^{(a)}\big(x+\mathbf{v}_1,y+\mathbf{v}_2\big),
\end{eqnarray}
where we have set, with $\nu(\cdot)=:\varrho_0(\cdot /\beta C)$,
\begin{eqnarray}
 \label{eqn:ath summand}
\lefteqn{\mathcal{G}_k^{(a)}\big(x+\mathbf{v}_1,y+\mathbf{v}_2\big)}\\
&=:&\int_{-\infty}^{+\infty}\,e^{-i\tau kE}\,\nu\left(k^{1/2-\epsilon}\,(\tau-\tau_a)\right)\,
\chi(\tau)\,\Pi_k\big((x+\mathbf{v}_1)_\tau,y+\mathbf{v}_2\big)
\,\mathrm{d}\tau\nonumber\\
&=&e^{-i\tau_a kE}\,\int_{-\infty}^{+\infty}\,e^{-i\tau kE}\,\nu\left(k^{1/2-\epsilon}\,\tau\right)\,
\chi(\tau+\tau_a)\,\Pi_k\big((x+\mathbf{v}_1)_{\tau+\tau_a},y+\mathbf{v}_2\big)
\,\mathrm{d}\tau.\nonumber
\end{eqnarray}
In the latter integral, the integrand is supported where $|\tau|\le 2\,\beta C\,k^{\epsilon-1/2}$; in addition, 
$\nu\left(k^{1/2-\epsilon}\,\tau\right)\equiv 1$ for $|\tau|\le \beta C\,k^{\epsilon-1/2}$. The constant $\beta$ 
can be chosen arbitrarily large.

In order to proceed further, we need to give an explicit expression for $(x+\mathbf{v}_1)_{\tau+\tau_a}$.
Given (\ref{eqn:lifted period}), a slight modification of Lemma 3.2 of \cite{pao-ltfII} shows that, 
with $A_a$ as on page \pageref{eqn:differentialHLC},
\begin{equation}
 \label{eqn:HLC_tau_a}
\phi^X_{-\tau_a}(x+\mathbf{v}_1)=y+\big(\vartheta_a+R_3(\mathbf{v}_1),A_a\mathbf{v}_1+R_2(\mathbf{v}_1)\big).
\end{equation}
Therefore, applying Corollary 2.2 of \cite{pao_ijm_2012} we have:
\begin{eqnarray}
 \label{eqn:estimate displacement}
(x+\mathbf{v}_1)_{\tau+\tau_a}&=&\phi^X_{-(\tau+\tau_a)}(x+\mathbf{v}_1)=\phi^X_{-\tau}\left(\phi^X_{-\tau_a}(x+\mathbf{v}_1\right)\Big)
\nonumber\\
&=&\phi^X_{-\tau}\Big(y+\big(\vartheta_a+R_3(\mathbf{v}_1),A_a\mathbf{v}_1+R_2(\mathbf{v}_1\big)\Big)\nonumber\\
&=&y+\Big(\vartheta_a+\tau\,\Big[f(n)+\omega_n\big(\upsilon_f(n),A_a\mathbf{v}_1\big)\big]+R_3(\tau,\mathbf{v}_1),\nonumber\\
&&
A_a\mathbf{v}_1-\tau\,\upsilon_f(n)+R_2(\tau,\mathbf{v}_1)\Big)\nonumber\\
&=&y+\big(\vartheta_a+\theta_a(\tau),\mathbf{v}_a(\tau)\big),
\end{eqnarray}
where $R_j$ is as usual a $\mathcal{C}^\infty$ function vanishing to $j$-th order at the origin, possibly depending on $a$,
and $\theta_a(\tau)$, $\mathbf{v}_a(\tau)$ are defined by the latter equality.

Thus in view of (\ref{eqn:scaling asymptotics rescaled 3}) we get
\begin{eqnarray}
 \label{eqn:szego asymptotic integral}
\lefteqn{\Pi_k\big((x+\mathbf{v}_1)_{\tau+\tau_a},y+\mathbf{v}_2\big)=
\Pi_k\big(y+\big(\vartheta_a+\theta_a(\tau),\mathbf{v}_a(\tau)\big),y+\mathbf{v}_2\big)}\nonumber\\
&\sim&\left(\frac{k}{\pi}\right)^{d}\,e^{k[i\vartheta_a+i\theta_a(\tau)+\psi_2(\mathbf{v}_a(\tau),\mathbf{v}_2)]}
 \cdot
\left[1+\sum_{j=1}^{+\infty}\,A_j(k,y;\mathbf{v}_a(\tau),\mathbf{v}_2)\right]. 
\end{eqnarray}

We deduce from (\ref{eqn:ath summand}) and (\ref{eqn:szego asymptotic integral}) that
\begin{eqnarray}
 \label{eqn:ath summand 1}
\lefteqn{\mathcal{G}_k^{(a)}\big(x+\mathbf{v}_1,y+\mathbf{v}_2\big)}\\
&=&e^{ik\,(\vartheta_a-\tau_a E)}\,\left(\frac{k}{\pi}\right)^d\int_{-\infty}^{+\infty}\,e^{ik[\theta_a(\tau)-\tau E-i\psi_2
(\mathbf{v}_a(\tau),\mathbf{v}_2)]}\,\mathcal{A}_k(y,\mathbf{v}_a(\tau),\mathbf{v}_2)
\,\mathrm{d}\tau,\nonumber
\end{eqnarray}
where
\begin{eqnarray}
 \label{eqn:definition of amplitude}
\lefteqn{\mathcal{A}_k(\tau,\mathbf{v}_a(\tau),\mathbf{v}_2)}\\
&\sim&\nu\left(k^{1/2-\epsilon}\,\tau\right)\,
\chi(\tau+\tau_a)\cdot\left[1+\sum_{j=1}^{+\infty}\,\sum_{b=\lceil -j/2\rceil}^{j}k^{b}Q_{j+2b}(y;\mathbf{v}_a(\tau),\mathbf{v}_2)\right],
\nonumber
\end{eqnarray}
and the asymptotic expansion can be integrated term by term. 

We can expand the exponent in (\ref{eqn:ath summand 1}) as follows:
\begin{eqnarray}
\lefteqn{\theta_a(\tau)-\tau E-i\psi_2\big
(\mathbf{v}_a(\tau),\mathbf{v}_2\big)}\\
&=&\tau\,\Big[f(n)-E+\omega_n\big(\upsilon_f(n),A_a\mathbf{v}_1\big)\big]+R_3(\tau,\mathbf{v}_1)\nonumber\\
&&-\omega_n\big(A_a\mathbf{v}_1-\tau\,\upsilon_f(n)+R_2(\tau,\mathbf{v}_1),\mathbf{v}_2\big)
+\frac{i}{2}\,\left\|A_a\mathbf{v}_1-\mathbf{v}_2-\tau\,\upsilon_f(n)+R_2(\tau,\mathbf{v}_1)\right\|^2\nonumber\\
&=&-\omega_n(A_a\mathbf{v}_1,\mathbf{v}_2)
+\frac{i}{2}\,\left\|A_a\mathbf{v}_a-\mathbf{v}_2\right\|^2+R_3(\tau,\mathbf{v}_1,\mathbf{v}_2)\nonumber\\
&&\tau\,\Big[f(n)-E+\omega_n\big(\upsilon_f(n),A_a\mathbf{v}_1+\mathbf{v}_2\big)
\big]\nonumber\\
&&+\frac{i}{2}\,\left[\tau^2\,\|\upsilon_f(n)\|^2-2\,\tau\,g_n\big(\upsilon_f(n),A_a\mathbf{v}_1-\mathbf{v}_2\big)\right]\nonumber\\
&=&-i\,\psi_2(A_a\mathbf{v}_1,\mathbf{v}_2)+\Upsilon _a(y,\mathbf{v}_1,\mathbf{v}_2;\tau)+
R_3(\tau,\mathbf{v}_1,\mathbf{v}_2),\nonumber
\end{eqnarray}
where we have set, using now that $f(n)=E$,
\begin{eqnarray}
 \label{eqn:defn of phase}
 \lefteqn{\Upsilon _a(y,\mathbf{v}_1,\mathbf{v}_2;\tau)}\\
&=:&\tau\,\omega_n\big(\upsilon_f(n),A_a\mathbf{v}_1+\mathbf{v}_2\big)
+\frac{i}{2}\,\left[\tau^2\,\|\upsilon_f(n)\|^2-2\,\tau\,g_n\big(\upsilon_f(n),A_a\mathbf{v}_1-\mathbf{v}_2\big)\right].
\nonumber
\end{eqnarray}

We can then rewrite (\ref{eqn:ath summand 1}) in the following manner:
\begin{eqnarray}
 \label{eqn:ath summand 2}
\lefteqn{\mathcal{G}_k^{(a)}\big(x+\mathbf{v}_1,y+\mathbf{v}_2\big)}\\
&=&e^{ik\,(\vartheta_a-\tau_a E)}\,\left(\frac{k}{\pi}\right)^d\int_{-\infty}^{+\infty}\,
e^{ik\widetilde{\Upsilon}_a(y,\mathbf{v}_1,\mathbf{v}_2;\tau)}\,\mathcal{B}_k(y,\mathbf{v}_a(\tau),\mathbf{v}_2)
\,\mathrm{d}\tau,\nonumber
\end{eqnarray}
where
\begin{equation}
 \label{eqn:new phase}
 \widetilde{\Upsilon}_a(y,\mathbf{v}_1,\mathbf{v}_2;\tau)=
 -i\,\psi_2(A_a\mathbf{v}_1,\mathbf{v}_2)+\Upsilon _a(y,\mathbf{v}_1,\mathbf{v}_2;\tau),
\end{equation}
and
\begin{equation}
 \label{eqn:new amplitude}
\mathcal{B}_k(y,\mathbf{v}_a(\tau),\mathbf{v}_2)=:e^{ik\,R_3(\tau,\mathbf{v}_1,\mathbf{v}_2)}\,
\mathcal{A}_k(y,\mathbf{v}_a(\tau),\mathbf{v}_2)
\end{equation}

We shall view (\ref{eqn:ath summand 2}) as an oscillatory integral in $\mathrm{d}\tau$, with phase 
$\widetilde{\Upsilon}_a$ and amplitude $\mathcal{B}_k(y,\mathbf{v}_a(\tau),\mathbf{v}_2)$. 
The imaginary part of $\widetilde{\Upsilon}_a$ satisfies:
\begin{eqnarray}
 \label{eqn:imaginary part phase}
 \lefteqn{\Im \widetilde{\Upsilon}_a(y,\mathbf{v}_1,\mathbf{v}_2;\tau)}\nonumber\\
 &=&\frac{1}{2}\,\left[\left\|A_a\mathbf{v}_a-\mathbf{v}_2\right\|^2+
 \tau^2\,\|\upsilon_f(n)\|^2-2\,\tau\,g_n\big(\upsilon_f(n),A_a\mathbf{v}_1-\mathbf{v}_2\big)\right]\nonumber\\
 &=&\frac{1}{2}\,\left\|A_a\mathbf{v}_a-\mathbf{v}_2-\tau\,\upsilon_f(n)\right\|^2\ge 0.
\end{eqnarray}

Regarding the phase, momentarily viewing $k$ as a continuous parameter we have:
\begin{lem}
 \label{lem:amplitude in S1:2}
$\mathcal{B}_\cdot(y,\mathbf{v}_a(\cdot),\mathbf{v}_2)\in S^0_{1/2-\epsilon}(\mathbb{R}\times \mathbb{R}_+)$,
as a function of $(\tau,k)\in \mathbb{R}\times \mathbb{R}_+$, uniformly in $y$, $\mathbf{v}_1$ and $\mathbf{v}_2$
with $\|\mathbf{v}_j\|\le C\,k^{\epsilon-1/2}$.
\end{lem}

\begin{proof}
[Proof of Lemma \ref{lem:amplitude in S1:2}]
In view of (\ref{eqn:definition of amplitude}) and (\ref{eqn:new amplitude}),
we have
$$
\mathcal{B}_k(y,\mathbf{v}_a(\tau),\mathbf{v}_2)=
\sum_{j=0}^{+\infty}\sum_{b=\lceil -j/2\rceil}^{j}F_{j,b}(\tau,k),
$$
where 
$$
F_{j,b}(\tau,k)=:\nu\left(k^{1/2-\epsilon}\,\tau\right)\,
\chi(\tau+\tau_a)\cdot e^{ikR_3(\tau,\mathbf{v}_1,\mathbf{v}_2)}\,
k^{b}Q_{j+2b}(y;\mathbf{v}_a(\tau),\mathbf{v}_2),
$$
Since the asymptotic expansion for $\Pi_k$ can be differentiated any number of times,
it suffices to prove that every $F_{j,b}\in S^{-3j\,(1/6-\epsilon)}_{1/2-\epsilon}(\mathbb{R}\times \mathbb{R}_+)$.

Thus we need to show that for any $j\ge 0$ and $\lceil -j/2\rceil\le b\le j$,
and every $l\in \mathbb{N}$ we have
\begin{equation}
 \label{eqn:bound on iterated derivatives}
 F_{j,b}^{(l)}(\tau,k)=:\frac{\partial^lF_{j,b}}{\partial\tau^l} (\tau,k)=O\left(k^{l(1/2-\epsilon)}\right) 
\end{equation}
as $k\rightarrow+\infty$.

On the support of $F_{j,b}$, we have
$$
\max\big\{|\tau|,\,\|\mathbf{v}_1\|,\,\|\mathbf{v}_2\|\big\}\le
D\,k^{\epsilon-1/2}
$$
for some $D\gg 0$. Therefore, the exponent in (\ref{eqn:new amplitude}) satisfies 
\begin{equation}
 \label{eqn:bound on exponent}
\big|k\,R_3(\tau,\mathbf{v}_1,\mathbf{v}_2)\big|=O\left(k^{-3(1/6-\epsilon)}\right).
\end{equation}
Similarly, on the same domain we have
\begin{eqnarray*}
 \left|k^{b}Q_{j+2b}(y;\mathbf{v}_a(\tau),\mathbf{v}_2)\right|\le D_1\,k^{b+(j+2b)\,(\epsilon-1/2)}\le 
D_1\,k^{-3j\,(1/6-\epsilon)}.
\end{eqnarray*}
Thus the statement is clear when $l=0$.

Let us make the inductive hypothesis that for $l\le l_0$
$F_{j,b}^{(l)}$ is a linear combination of terms of the form 
\begin{equation}
 \label{eqn:inductive form}
k^{b'+a'(1/2-\epsilon)}\,\widetilde{\nu}\left(k^{1/2-\epsilon}\,\tau\right)\,R_{c'}(\tau,\mathbf{v}_1,\mathbf{v}_2)
\cdot e^{ikR_3(\tau,\mathbf{v}_1,\mathbf{v}_2)},
\end{equation}
where 
\begin{equation}
 \label{eqn:inductive condition}
b'+(a'-c')\,\left(\frac{1}{2}-\epsilon\right)\le -3j\,\left(\frac{1}{6}-\epsilon\right)+ l\,\left(\frac{1}{2}-\epsilon\right).
\end{equation}
If we differentiate $F_{j,b}^{(l_0)}$ with respect to $\tau$, thus passing from $l_0$ to $l_0+1$,
we obtain by the Leibnitz rule the sum of three terms, as follows. The first has the form

\begin{equation}
 \label{eqn:inductive form1}
k^{b'+(a'+1)(1/2-\epsilon)}\,\widetilde{\nu}'\left(k^{1/2-\epsilon}\,\tau\right)\,R_{c'}(\tau,\mathbf{v}_1,\mathbf{v}_2)
\cdot e^{ikR_3(\tau,\mathbf{v}_1,\mathbf{v}_2)};
\end{equation}
the second has the form
\begin{equation}
 \label{eqn:inductive form2}
k^{b'+a'(1/2-\epsilon)}\,\widetilde{\nu}\left(k^{1/2-\epsilon}\,\tau\right)\,R_{c'-1}(\tau,\mathbf{v}_1,\mathbf{v}_2)
\cdot e^{ikR_3(\tau,\mathbf{v}_1,\mathbf{v}_2)},
\end{equation}
since $\partial_\tau R_{c'}=R_{c'-1}$; finally the third has the form
\begin{equation}
 \label{eqn:inductive form3}
k^{b'+a'(1/2-\epsilon)}\,\widetilde{\nu}\left(k^{1/2-\epsilon}\,\tau\right)\,R_{c'}(\tau,\mathbf{v}_1,\mathbf{v}_2)\,
k\,R_2(\tau,\mathbf{v}_1,\mathbf{v}_2)
\cdot e^{ikR_3(\tau,\mathbf{v}_1,\mathbf{v}_2)},
\end{equation}
since $\partial_\tau[\tau\,R_3]=k\,R_2$.

In case (\ref{eqn:inductive form1}), $(a',b',c')$ get replaced by $(a'+1,b',c')$, in case (\ref{eqn:inductive form2})
by $(a',b',c'-1)$, and in case (\ref{eqn:inductive form3}) by $(a',b'+1,c'+2)$. At any rate, the inductive step is complete.

\end{proof}

We can in fact rewrite the expansion for $\mathcal{B}_k$ as we did for $\mathcal{A}_k$.

\begin{lem}
\label{lem:expansion Bk}
We have the following asymptotic expansion:
\begin{eqnarray*}
\lefteqn{\mathcal{B}_k(\tau,\mathbf{v}_a(\tau),\mathbf{v}_2)}\\
&\sim&\nu\left(k^{1/2-\epsilon}\,\tau\right)\,
\chi(\tau+\tau_a)\cdot\left[1+\sum_{j=1}^{+\infty}\,\sum_{b=\lceil -j/2\rceil}^{j}k^{b}P_{j+2b}^{(j,b)}(y;\mathbf{v}_a(\tau),\mathbf{v}_2)\right],
\nonumber
\end{eqnarray*}
where $P_l^{(j,b)}$ is a homogeneous polynomial of degree $l$.
\end{lem}

\begin{proof}
[Proof of Lemma \ref{lem:expansion Bk}]
First we notice that
for every $N\ge 0$ we have:
\begin{equation}
 \label{eqn:esponente espanso}
 e^{ikR_3(\tau,\mathbf{v}_1,\mathbf{v}_2)}=\sum_{l=0}^N\frac{k^l}{l!}\,R_{3l}(\tau,\mathbf{v}_1,\mathbf{v}_2)+
 O\left(k^{N+1}\,R_{3(N+1)}(\tau,\mathbf{v}_1,\mathbf{v}_2)\right).
\end{equation}
The remainder term satisfies
\begin{eqnarray}
 \label{eqn:bound on remainder}
\lefteqn{\left|k^{N+1}\,R_{3(N+1)}(\tau,\mathbf{v}_1,\mathbf{v}_2)\right|}\\
&=&O\left(k^{N+1+3(N+1)(\epsilon-1/2)}\right)=O\left(k^{-1/2+3\epsilon-3N(1/6-\epsilon)}\right).\nonumber
\end{eqnarray}

We can thus multiply the asymptotic expansions for $\mathcal{A}_k$ and for the exponential term.
Now the general term in the asymptotic expansion for $\mathcal{A}_k$ is a multiple
of 
$$
k^b\,\tau^{r_0}\,\mathbf{v}_1^{\mathbf{r}_1}\,\mathbf{v}_2^{\mathbf{r}_2},
$$
where $r_0\in \mathbb{N}_0$, $\mathbf{r}_1$ and $\mathbf{r}_2$ are multi-indexes
with $r_0+|\mathbf{r}_1|+|\mathbf{r}_2|\ge j+2b$,
and $\lceil -j/2\rceil\le b\le j$. On the other hand, the general term in the expansion 
for the exponential is a multiple of
$$
k^s\,\tau^{r'_0}\,\mathbf{v}_1^{\mathbf{r}'_1}\,\mathbf{v}_2^{\mathbf{r}'_2},
$$
where now 
$$
r'_0+|\mathbf{r}'_1|+|\mathbf{r}'_2|\ge 3s,
$$
with $s\ge 0$. If we then multiply the two expansions, we get that 
$$
e^{ikR_3(\tau,\mathbf{v}_1,\mathbf{v}_2)}\cdot \mathcal{A}_k(\tau,\mathbf{v}_1,\mathbf{v}_2)
$$
is a linear combination of terms of the form
\begin{equation}
 \label{eqn:general factored term}
k^{b+s}\,\tau^{r_0+r'_0}\,\mathbf{v}_1^{\mathbf{r}_1+\mathbf{r}'_1}\,\mathbf{v}_2^{\mathbf{r}_2+\mathbf{r}'_2}.
\end{equation}

Let us set $\rho_0=r_0+r'_0$, $\rho_j=\mathbf{r}_j+\mathbf{r}_j'$ for $j=1,2$,
$b'=b+s$, $j'=j+s$.
We can then rewrite (\ref{eqn:general factored term}) in the form
\begin{eqnarray}
 \label{eqn:factored term 1}
 \lefteqn{k^{b'}\,\tau^{\rho_0}\,\mathbf{v}_1^{\rho_1}\,\mathbf{v}_2^{\rho_2}},
\end{eqnarray}
with 
\begin{eqnarray}
 \label{eqn:total length multiindexes}
 \rho_0+|\rho_1|+|\rho_2|&=&(r_0+|\mathbf{r}_1|+|\mathbf{r}_2|)+(r'_0+|\mathbf{r}'_1|+|\mathbf{r}'_2|)\nonumber\\
& \ge&j+2b+3s=(j+s)+2\,(b+s)=j'+2\,b'.
\end{eqnarray}

On the other hand, we have
$$
\lceil -j'/2\rceil=\lceil -(j+s)/2\rceil\le \lceil -j/2\rceil +s \le b'=b+s\le j+s=j'.
$$
\end{proof}

Summing up, we have the following situation:
\begin{eqnarray}
 \label{eqn:summing up}
\lefteqn{\mathcal{G}_k\big(x+(\theta_1,\mathbf{v}_1),y+(\theta_2,\mathbf{v}_2)\big)\sim e^{ik(\theta_1-\theta_2)}}\\
&&\cdot \sum_a e^{ik(\vartheta_a-\tau_a E)}\,\left(\frac{k}{d}\right)^d\,
\int_{-\infty}^{+\infty}e^{ik\Upsilon_a(y,\mathbf{v}_1,\mathbf{v}_2;\tau)}\,\mathcal{B}_k(\tau,\mathbf{v}_a(\tau),\mathbf{v}_2)
\mathrm{d}\tau,\nonumber\end{eqnarray}
where, in view of Lemma \ref{lem:expansion Bk} and (\ref{eqn:estimate displacement}) we have
\begin{eqnarray}
\label{eqn:expansion Bk1}
\lefteqn{\mathcal{B}_k(\tau,\mathbf{v}_a(\tau),\mathbf{v}_2)}\\
&\sim&\nu\left(k^{1/2-\epsilon}\,\tau\right)\,
\chi(\tau+\tau_a)\cdot\left[1+\sum_{j=1}^{+\infty}\,\sum_{b=\lceil -j/2\rceil}^{j}k^{b}\sum_{|\rho|\ge j+2b}
\gamma_{jb\rho}\,\tau^{\rho_0}\,\mathbf{v}_1^{\rho_1}\,\mathbf{v}_2^{\rho_2}\right]
\nonumber\\
&\sim&\nu\left(k^{1/2-\epsilon}\,\tau\right)\,
\chi(\tau+\tau_a)\cdot\sum_{j=0}^{+\infty}\,\sum_{b=\lceil -j/2\rceil}^{j}k^{b}\sum_{|\rho|\ge j+2b}
\gamma_{jb\rho}\,\tau^{\rho_0}\,\mathbf{v}_1^{\rho_1}\,\mathbf{v}_2^{\rho_2},
\nonumber
\end{eqnarray}
where $\rho=(\rho_0,\rho_1,\rho_2)\in \mathbb{N}_0\times \mathbb{N}_0^{2d}\times \mathbb{N}_0^{2d}$ is a multi-index,
and $|\rho|=:\rho_0+|\rho_1|+|\rho_2|$ denotes its length; here $\gamma_{00\mathbf{0}}=1$.

Since the expansion may be integrated term by term, we are reduced to considering oscillatory integrals of the form
\begin{eqnarray}
 \label{eqn:ath summand 3}
I_\rho^{(a)}(y,\mathbf{v}_1,\mathbf{v}_2)&=& e^{ik\,(\vartheta_a-\tau_a E)}\,\left(\frac{k}{\pi}\right)^d\,k^b\int_{-\infty}^{+\infty}\,
e^{ik\widetilde{\Upsilon}_a(y,\mathbf{v}_1,\mathbf{v}_2;\tau)}\nonumber\\
&&\cdot\nu\left(k^{1/2-\epsilon}\,\tau\right)\,
\chi(\tau+\tau_a)\cdot\tau^{\rho_0}\,\mathbf{v}_1^{\rho_1}\,\mathbf{v}_2^{\rho_2}
\,\mathrm{d}\tau,
\end{eqnarray}
where $|\rho|\ge j+2b$, with $\lceil-j/2\rceil\le b\le j$ for some $j\ge 0$.
By the previous considerations, furthermore, we are in a position to apply the stationary phase Lemma 
for complex valued phase functions (\cite{MelSjo}, \cite{hor}).

In view of (\ref{eqn:new phase}), the study of the critical points of $\widetilde{\Upsilon}_a$ is clearly reduced to the
study of critical points of $\Upsilon_a$ (recall (\ref{eqn:defn of phase}) and (\ref{eqn:new phase}). 
The following Lemma is proved by a straightforward computation, that we shall leave to the reader:

\begin{lem}
 \label{lem:critical point}
 The phase $\Upsilon _a(y,\mathbf{v}_1,\mathbf{v}_2;\cdot)$ has a unique critical point 
 $\tau_c=\tau_c(n,\mathbf{v}_1,\mathbf{v}_2)$, given by 
 $$
\tau_c=\frac{1}{\|\upsilon_f(n)\|^2}\cdot \Big[g_n\big(\upsilon_f(n),A_a\mathbf{v}_1-\mathbf{v}_2\big)
+i\,\omega_n\big(\upsilon_f(n),A_a\mathbf{v}_1+\mathbf{v}_2\big)\Big].
$$ Furthermore, 
 $$\partial^2_\tau \Upsilon _a(y,\mathbf{v}_1,\mathbf{v}_2;\tau_c)=i\,\|\upsilon_f(n)\|^2,$$
 so that the critical point is non-degenerate. In addition,
 at the critical point we have
\begin{eqnarray}
 \label{eqn:critical value phase}
\lefteqn{\Upsilon _a(y,\mathbf{v}_1,\mathbf{v}_2;\tau_c)}\\
&=&-\dfrac{i}{2\,\|\upsilon_f(n)\|^2}\,\Big[g_n\big(\upsilon_f(n),A_a\mathbf{v}_1-\mathbf{v}_2\big)^2
 -\omega_n\big(\upsilon_f(n),A_a\mathbf{v}_1+\mathbf{v}_2\big)^2\nonumber\\
 &&+2i\,
g_n\big(\upsilon_f(n),A_a\mathbf{v}_1-\mathbf{v}_2\big)\cdot 
\omega_n\big(\upsilon_f(n),A_a\mathbf{v}_1+\mathbf{v}_2\big) \Big].\nonumber
\end{eqnarray}

\end{lem}

Let us clarify (\ref{eqn:critical value phase}) in light of the decomposition in Definition
\ref{defn:tvh components}. To this end let us write: 
\begin{equation}
 \label{eqn:basic splitting vj}
 \mathbf{v}_j=\mathbf{v}_{j\mathrm{t}}+\mathbf{v}_{j\mathrm{v}}+\mathbf{v}_{j\mathrm{h}} \,\,\,\,\,\,
 \,\,\,\,\,\,\,\,\,\,\,
 (j=1,\,2).
\end{equation}
Since $A$ is unitary and leaves $\upsilon_f(n)$ invariant, it preserves the previous decomposition, meaning
that $(A\mathbf{v}_{j})_\mathrm{t}=A\mathbf{v}_{j\mathrm{t}}$, and so forth.
Hence we have
\begin{eqnarray}
\label{eqn:psi2 decomposition}
 \psi_2(A_a\mathbf{v}_1,\mathbf{v}_2) &=&\psi_2(A_a\mathbf{v}_{1\mathrm{h}},\mathbf{v}_{2\mathrm{h}})+
 \psi_2(A_a\mathbf{v}_{1\mathrm{t}}+A_a\mathbf{v}_{1\mathrm{v}},\mathbf{v}_{2\mathrm{t}}+\mathbf{v}_{2\mathrm{v}})
 \nonumber\\
 &=&\psi_2(A_a\mathbf{v}_{1\mathrm{h}},\mathbf{v}_{2\mathrm{h}})
 -i\,\omega_n (A_a\mathbf{v}_{1\mathrm{t}},\mathbf{v}_{2\mathrm{v}})-i\,\omega_n (A_a\mathbf{v}_{1\mathrm{v}},\mathbf{v}_{2\mathrm{t}})
 \nonumber\\
 &&-\frac{1}{2}\,\|A_a\mathbf{v}_{1\mathrm{t}}-\mathbf{v}_{2\mathrm{t}}\|^2
 -\frac{1}{2}\,\|A_a\mathbf{v}_{1\mathrm{v}}-\mathbf{v}_{2\mathrm{v}}\|^2.
\end{eqnarray}
Furthermore, it is evident that
\begin{eqnarray}
 \label{eqn:g_n decomposed}
 g_n\big(\upsilon_f(n),A_a\mathbf{v}_1-\mathbf{v}_2\big)^2=\|\upsilon_f(n)\|^2\cdot 
 \|A_a\mathbf{v}_{1\mathrm{v}}-\mathbf{v}_{2\mathrm{v}}\|^2
\end{eqnarray}
and that
\begin{equation}
 \label{eqn:omega_n decomposed}
 \omega_n\big(\upsilon_f(n),A_a\mathbf{v}_1+\mathbf{v}_2\big)^2=\|\upsilon_f(n)\|^2\cdot 
 \|A_a\mathbf{v}_{1\mathrm{t}}+\mathbf{v}_{2\mathrm{t}}\|^2.
\end{equation}

Therefore, we can write
\begin{eqnarray}
 \label{eqn:cirtical value total phase}
i\widetilde{\Upsilon}_a(y,\mathbf{v}_1,\mathbf{v}_2;\tau_c)&=&
 \psi_2(A_a\mathbf{v}_1,\mathbf{v}_2)+i\Upsilon _a(y,\mathbf{v}_1,\mathbf{v}_2;\tau)\nonumber\\
&=&\psi_2(A_a\mathbf{v}_{1\mathrm{h}},\mathbf{v}_{2\mathrm{h}})-\|A_a\mathbf{v}_{1\mathrm{t}}\|^2-\|\mathbf{v}_{2\mathrm{t}}\|^2
+i\,\mathcal{H},
\end{eqnarray}
where
\begin{eqnarray}
 \label{eqn:defn di H}
\mathcal{H}&=:&-\omega_n (A_a\mathbf{v}_{1\mathrm{t}},\mathbf{v}_{2\mathrm{v}})-\omega_n (A_a\mathbf{v}_{1\mathrm{v}},\mathbf{v}_{2\mathrm{t}})\\
&&
+\frac{1}{\|\upsilon_f(n)\|^2} \, g_n\big(\upsilon_f(n),A_a\mathbf{v}_1-\mathbf{v}_2\big)\cdot 
\omega_n\big(\upsilon_f(n),A_a\mathbf{v}_1+\mathbf{v}_2\big). \nonumber
\end{eqnarray}

To compute $\mathcal{H}$, let us write
$$
A_a\mathbf{v}_1=\alpha_1\,\upsilon_f(n)+\beta_1\,J_n\,\upsilon_f(n)+A_a\mathbf{v}_{1\mathrm{h}},
$$
$$
\mathbf{v}_2=\alpha_2\,\upsilon_f(n)+\beta_2\,J_n\,\upsilon_f(n)+\mathbf{v}_{2\mathrm{h}},
$$
for appropriate $\alpha_j,\,\beta_j\in \mathbb{R}$. Then clearly
$$
\omega_n (A_a\mathbf{v}_{1\mathrm{t}},\mathbf{v}_{2\mathrm{v}})=-\alpha_2\,\beta_1\,\|\upsilon_f(n)\|^2,
\,\,\,\,\,\,
\omega_n (A_a\mathbf{v}_{1\mathrm{v}},\mathbf{v}_{2\mathrm{t}})=
\alpha_1\,\beta_2\,\|\upsilon_f(n)\|^2;
$$
furthermore,
$$
g_n\big(\upsilon_f(n),A_a\mathbf{v}_1-\mathbf{v}_2\big)=
(\alpha_1-\alpha_2)\,\|\upsilon_f(n)\|^2,
$$
and
$$
\omega_n\big(\upsilon_f(n),A_a\mathbf{v}_1+\mathbf{v}_2\big)=(\beta_1+\beta_2)\,\|\upsilon_f(n)\|^2.
$$
Hence we may rewrite (\ref{eqn:defn di H}) as follows:
\begin{eqnarray}
\label{eqn:computation of H}
 \mathcal{H}
&=&\|\upsilon_f(n)\|^2\,\big(\alpha_2\,\beta_1-\alpha_1\,\beta_2+(\alpha_1-\alpha_2)\,(\beta_1+\beta_2)\big)\\
&=&\|\upsilon_f(n)\|^2\,\big(\alpha_1\beta_1
-\alpha_2\,\beta_2\big)=\omega_n\big(A_a\mathbf{v}_{1\mathrm{v}},A_a\mathbf{v}_{1\mathrm{t}}\big)-
\omega_n\big(A_a\mathbf{v}_{2\mathrm{v}},A_a\mathbf{v}_{2\mathrm{t}}\big).\nonumber
\end{eqnarray}

Recalling (\ref{eqn:defn of Q}), we conclude from (\ref{eqn:cirtical value total phase})
and (\ref{eqn:computation of H}) that
\begin{eqnarray}
 \label{eqn:cirtical value total phase 1}
i\widetilde{\Upsilon}_a(y,\mathbf{v}_1,\mathbf{v}_2;\tau_c)&=&\psi_2(A_a\mathbf{v}_{1\mathrm{h}},\mathbf{v}_{2\mathrm{h}})
-\|A_a\mathbf{v}_{1\mathrm{t}}\|^2-\|\mathbf{v}_{2\mathrm{t}}\|^2\nonumber\\
&&+i\,\Big[\omega_n\big(A_a\mathbf{v}_{1\mathrm{v}},A_a\mathbf{v}_{1\mathrm{t}}\big)-
\omega_n\big(A_a\mathbf{v}_{2\mathrm{v}},A_a\mathbf{v}_{2\mathrm{t}}\big)\Big]\nonumber\\
&=&\mathcal{Q}(A_a\mathbf{v}_{1},\mathbf{v}_{2}).
\end{eqnarray}

Since $\widetilde{\Upsilon}_a$ is a polynomial of degree two in $\tau$, there is no third order remainder to take care
of in the stationary phase expansion for (\ref{eqn:ath summand 3}).
In addition, upon choosing the constant $\beta$ in Proposition \ref{prop:localizzazione in tau}
sufficiently large, we may assume that $\rho\left(k^{1/2-\epsilon}\,\tau\right)\equiv 1$
near the critical point.

Thus the asymptotic expansion coming from each summand (\ref{eqn:ath summand 3})
is
\begin{eqnarray}
 \label{eqn:ath summand expanded}
\lefteqn{I_\rho^{(a)}(y,\mathbf{v}_1,\mathbf{v}_2)
\sim e^{k\,[i(\vartheta_a-\tau_a E)+\mathcal{Q}(A_a\mathbf{v}_{1},\mathbf{v}_{2})]}\,
\left(\frac{k}{\pi}\right)^d\,k^{b-1/2}\,\frac{\sqrt{2\pi}}{\|\upsilon_f(n\|}}\nonumber\\
&&\cdot\sum_{r=0}^{\infty}\frac{1}{r!}\,\left(-2k\,\|\upsilon_f(n)\|^2\right)^{-r}\,
\left.\frac{\partial^{2r}}{\partial\tau^{2r}}\left[\chi(\tau+\tau_a)\cdot\tau^{\rho_0}
\right]\right|_{\tau=\tau_c}\,\mathbf{v}_1^{\rho_1}\,\mathbf{v}_2^{\rho_2}.
\end{eqnarray}
Sice $\tau_c$ is a complex number, in the previous expression $\chi$ should be really replaced by 
an almost analytic extension, that we still denote by $\chi$ for notational simplicity.
On the other hand, with the same abuse of notation, we have
\begin{eqnarray}
 \label{eqn:derivative term}
 \left.\frac{\partial^{2r}}{\partial\tau^{2r}}\left[\chi(\tau+\tau_a)\cdot\tau^{\rho_0}
\right]\right|_{\tau=\tau_c}
&=&\sum_{l=0}^{\min\{2r,\rho_0\}}C_{rl}\,\chi^{(2r-l)}(\tau_c+\tau_a)\,\tau_c^{\rho_0-l}.
\end{eqnarray}
If we Taylor expand $\chi^{(2r-l)}(\cdot+\tau_a)$ at $\tau=0$ to estimate
$\chi^{(2r-l)}(\tau_c+\tau_a)$ asymptotically, we end up with an asymptotic expansion of the form
\begin{equation}
 \label{eqn:expand chi derivative}
\chi^{(2r-l)}(\tau_c+\tau_a)\sim \sum_{s\ge 0}\frac{1}{s!}\,\chi^{(2r-l+s)}(\tau_a)\,\tau_c^s .
\end{equation}

Therefore,
\begin{eqnarray}
 \label{eqn:derivative term1}
 \left.\frac{\partial^{2r}}{\partial\tau^{2r}}\left[\chi(\tau+\tau_a)\cdot\tau^{\rho_0}
\right]\right|_{\tau=\tau_c}
\sim\sum_{l=0}^{\min\{2r,\rho_0\}}\sum_{s\ge 0}\frac{C_{rl}}{s!}\,\chi^{(2r-l+s)}(\tau_a)\,\tau_c^{\rho_0+s-l}.
\end{eqnarray}
Inserting (\ref{eqn:derivative term1}) in (\ref{eqn:ath summand expanded}), we obtain
\begin{eqnarray}
 \label{eqn:ath summand expanded 1}
I_\rho^{(a)}(y,\mathbf{v}_1,\mathbf{v}_2)
&\sim& e^{k\,[i(\vartheta_a-\tau_a E)+\mathcal{Q}(A_a\mathbf{v}_{1},\mathbf{v}_{2})]}\,
\left(\frac{k}{\pi}\right)^d\,k^{b-1/2}\,\frac{\sqrt{2\pi}}{\|\upsilon_f(n\|}\\
&&\cdot\sum_{r=0}^{\infty}k^{-r}
\sum_{l=0}^{\min\{2r,\rho_0\}}\sum_{s\ge 0}C_{rls}^{(a)}\,\chi^{(2r-l+s)}(\tau_a)\,\tau_c^{\rho_0+s-l}\,
\mathbf{v}_1^{\rho_1}\,\mathbf{v}_2^{\rho_2}.\nonumber
\end{eqnarray}

Now by Lemma \ref{lem:critical point} $\tau_c$ is linear in $\mathbf{v}_1$ and $\mathbf{v}_2$;
therefore, we can rewrite (\ref{eqn:ath summand expanded 1}) as
\begin{eqnarray}
 \label{eqn:ath summand expanded 2}
I_\rho^{(a)}(y,\mathbf{v}_1,\mathbf{v}_2)
&\sim& e^{k\,[i(\vartheta_a-\tau_a E)+\mathcal{Q}(A_a\mathbf{v}_{1},\mathbf{v}_{2})]}\,
\left(\frac{k}{\pi}\right)^d\,k^{b-1/2}
\\
&&\cdot\sum_{r=0}^{\infty}k^{-r}\,
\sum_{l=0}^{\min\{2r,\rho_0\}}\sum_{s\ge 0}\,\chi^{(2r-l+s)}(\tau_a)\,P_{|\rho|+s-l}^{(a)}(n;\mathbf{v}_1,\mathbf{v}_2),\nonumber
\end{eqnarray}
where $P_{\ell}^{(a)}(n;\cdot,\cdot)$ is a homogeneous polynomial of degree $\ell$.

If we insert (\ref{eqn:ath summand expanded 2}) in (\ref{eqn:summing up}), with 
the amplitude $\mathcal{B}_k$ given by the asymptotic expansion
(\ref{eqn:expansion Bk1}), we obtain the following asymptotic expansion:
\begin{eqnarray}
 \label{eqn:summing up1}
\lefteqn{\mathcal{G}_k\big(x+(\theta_1,\mathbf{v}_1),y+(\theta_2,\mathbf{v}_2)\big)}\\
&\sim& e^{ik(\theta_1-\theta_2)}\,\left(\frac{k}{\pi}\right)^d
\cdot \frac{\sqrt{2\pi}}{\|\upsilon_f(n)\|}\,\sum_a e^{k\,[i(\vartheta_a-\tau_a E)+\mathcal{Q}(A_a\mathbf{v}_{1},\mathbf{v}_{2})]}\nonumber\\
&&\cdot
\sum_{j=0}^{+\infty}\,\sum_{b=\lceil -j/2\rceil}^{j}k^{b-1/2}\sum_{|\rho|\ge j+2b}\sum_{r=0}^{\infty}k^{-r}\nonumber\\
&&\cdot
\sum_{l=0}^{\min\{2r,\rho_0\}}\sum_{s\ge 0}\,\chi^{(2r-l+s)}(\tau_a)\,P_{|\rho|+s-l}^{(ajb\rho ls)}(n;\mathbf{v}_1,\mathbf{v}_2),\nonumber\end{eqnarray}
where $P_{\ell}^{(ajb\rho ls)}(n;\cdot,\cdot)$ is homogeneous of degree $\ell$.

Let us now pass to rescaled coordinates, thus replacing $\mathbf{v}_j$ by $\mathbf{v}_j/\sqrt{k}$, we may express the 
previous expansion in the following manner:
\begin{eqnarray}
 \label{eqn:summing up2}
\lefteqn{\mathcal{G}_k\left(x+\left(\theta_1,\frac{\mathbf{v}_1}{\sqrt{k}}\right),y+\left(\theta_2,\dfrac{\mathbf{v}_2}{\sqrt{2}}\right)\right)}\\
&\sim& e^{ik(\theta_1-\theta_2)}\,\left(\frac{k}{\pi}\right)^{d-1/2}
\cdot \frac{\sqrt{2}}{\|\upsilon_f(n)\|}\,\sum_a e^{ik\,(\vartheta_a-\tau_a E)+\mathcal{Q}(A_a\mathbf{v}_{1},\mathbf{v}_{2})}\nonumber\\
&&\cdot
\sum_{j=0}^{+\infty}\,\sum_{b=\lceil -j/2\rceil}^{j}\sum_{|\rho|\ge j+2b}\sum_{r=0}^{\infty}
\sum_{l=0}^{\min\{2r,\rho_0\}}\sum_{s\ge 0}k^{-\frac{1}{2}(2r-2b+|\rho|+s-l)}\nonumber\\
&&\cdot\chi^{(2r-l+s)}(\tau_a)\,
P_{|\rho|+s-l}^{(ajb\rho ls)}(n;\mathbf{v}_1,\mathbf{v}_2),\nonumber
\end{eqnarray}

Let us remark that, since $|\rho|\ge j+2b$,
\begin{equation}
 \label{eqn:only nonnegative powers}
2r-2b+|\rho|+s-l= (2r-l)+(|\rho|-2b)+s\ge (2r-l)+j+s\ge 0;
\end{equation}
therefore, only integer or half-integer powers of $k$ of 
the form $k^{d-1/2-\ell/2}$ with $\ell\ge 0$ contribute to the asymptotic expansion.

It is also clear from (\ref{eqn:summing up2}) and (\ref{eqn:only nonnegative powers}) 
that the coefficient of $k^{d-1/2-\ell/2}$ is the evaluation at $\tau_a$ of a differential polynomial in $\chi$
of degree $\le \ell$.

Let us first consider the contribution with $\ell=0$.

\begin{lem}
 \label{lem:ell=0}
$2r-2b+|\rho|+s-l=0$ if and only if $r=l=j=b=s=0$ and $\rho=\mathbf{0}$. 
\end{lem}

Thus for each $a$ there is only one summand in (\ref{eqn:summing up2})
contributing to the power $k^{d-1/2}$, and it is readily seen from (\ref{eqn:expand chi derivative}) that its coefficient 
is $\chi(\tau_a)$.

\begin{proof}
[Proof of Lemma \ref{lem:ell=0}.]
By (\ref{eqn:only nonnegative powers}), if $2r-2b+|\rho|+s-l=0$ then $2r=l$, $j=s=0$ and $|\rho|=2b$; given that 
$\lceil -j/2\rceil\le b\le j$, then also $b=0$ and so $|\rho|=0$. On the other hand
$l\le \min\{2r,\rho_0\}$,
and so $l=0$ and thus $r=0$. 
\end{proof}

\begin{lem}
 \label{lem:finitely many terms}
For any $n_0\in \mathbb{N}$ there are only finitely many summands
in (\ref{eqn:summing up2}) with $2r-2b+|\rho|+s-l=n_0$.
\end{lem}

Thus the sum of all polynomials $P_{|\rho|+s-l}^{(ajb\rho ls)}(n;\mathbf{v}_1,\mathbf{v}_2)$ with 
$2r-2b+|\rho|+s-l=n_0$ is itself a polynomial.

\begin{proof}
[Proof of Lemma \ref{lem:finitely many terms}.]
If
$2r-2b+|\rho|+s-l=n_0$, then $0\le j\le (2r-l)+s+j\le (2r-l)+s+(|\rho|-2b)=n_0$, and so 
$\lceil-n_0/2\rceil\le b\le n_0$. Consequently, $|\rho|-2n_0\le |\rho|-2b\le 
(2r-l)+s+(|\rho|-2b)=n_0$.
Thus, $|\rho|\le 3n_0$, and so also $l\le \rho_0\le |\rho|\le 3n_0$.
Therefore, we also have $2r\le l+n_0\le 4n_0$. Similarly,
$0\le s\le (2r-l)+s+(|\rho|-2b)=n_0$.
\end{proof}

\begin{lem}
 \label{lem:3bound}
For any summand in (\ref{eqn:summing up2}), we have
$|\rho|+s-l\le 3\,(2r-2b+|\rho|+s-l)$.
\end{lem}

Thus, the polynomials in $(\mathbf{v}_1,\mathbf{v}_2)$ contributing to a given power
$k^{d-1/2-\ell/2}$ in (\ref{eqn:summing up2}) all have degree $\le 3\ell$.

\begin{proof}
[Proof of Lemma \ref{lem:3bound}]
The claimed inequality is equivalent to 
$
6b+2l\le 6r+2|\rho|+2s$. 
On the other hand, given that $l\le 2r$ and $|\rho|\ge j+2b\ge 3b$,
$$6b+2l\le 6b+4r\le 6b +6r \le 2|\rho|+6r\le 6r+2|\rho|+2s.$$
\end{proof}

It is furthermore also evident that $|\rho|+s-l$ and $2r-2b+|\rho|+s-l$ have the same parity.
Thus each polynomial in $(\mathbf{v}_1,\mathbf{v}_2)$ contributing to a given power
$k^{d-1/2-\ell/2}$ in (\ref{eqn:summing up2}) has parity $(-1)^\ell$.

The proof of Theorem \ref{thm:scaling asymptotics} is complete.

\end{proof}

\section{Proof of Theorem \ref{thm:trace of G_k}}

\begin{proof}[Proof of Theorem \ref{thm:trace of G_k}]
 
Let us define 
\begin{equation}
 \label{eqn:defn of Gk diagonal}
 \mathfrak{G}_k(m)=:\mathcal{G}_k(x,x)\,\,\,\,\mathrm{if}\,\,\,\,x\in X_m.
\end{equation}
The definition is well-posed, and $\mathfrak{G}_k\in \mathcal{C}^\infty (M)$.

The following is an immediate consequence of Theorem \ref{thm:rapid decrease}.
\begin{cor}
 \label{cor:rapid decay diagonal}
 Suppose that $E$ is a regular value of $f$. Then,
 for any $C>0$ and $\epsilon>0$, we have
 $\mathfrak{G}_k(m)=O\left(k^{-\infty}\right)$, 
 uniformly for 
 $$
 \max\left\{\mathrm{dist}_M(m,M_E),\,\mathrm{dist}_M\left(m,m^\chi\right)\right\}\ge C\,k^{\epsilon-1/2}.
 $$
\end{cor}

\begin{proof}
[Proof of Corollary \ref{cor:rapid decay diagonal}]
When $E$ is a regular value of $f$, any $m$ in a small tubular neighborhood of $M_E$
may be written $m=\exp_{m_0}\left(\lambda\,J_{m_0}\big(\upsilon_f(m_0)\big)\right)$, for unique
$m_0\in M_E$, $\lambda\in \mathbb{R}$. Here $\exp$ is of course the exponential map for the 
Riemannian manifold $(M,g)$. Clearly, for sufficiently small $\lambda$ one
has 
$a\,|\lambda|\le \mathrm{dist}_M(m,M_E)\le A\,|\lambda|$ and
$a\,|\lambda|\le |f(m)-E|\le A\,|\lambda|$, for suitable constants $0<a<A$; we may take, for example,
$a=\|\upsilon_f(m_0)\|^2/2$, $A=2\,\|\upsilon_f(m_0)\|^2$. It follows
 that $|f(m)-E|>(b/A)\,\mathrm{dist}_M(m,M_E)$, and this implies the statement.
\end{proof}

We want to estimate
\begin{eqnarray}
 \mathrm{trace}(\mathcal{G}_k)&=&\int_X \mathcal{G}_k(x,x)\,\mathrm{d}V_X(x)\\
 &=&\int_M\mathfrak{G}_k(m)\,\mathrm{d}V_M(m).\nonumber
\end{eqnarray}

Integration in $\mathrm{d}V_M(m)$, by Corollary \ref{cor:rapid decay diagonal},
asymptotically localizes to a neighborhood of $U_E$ of $M_E$ in $M$. Let us normalize the previous
parametrization by the map:
\begin{equation}
 \label{eqn:defn of eta}
 \eta:(m,\lambda)\in M_E\times (-\delta,\delta)\mapsto 
\exp_m\Big((\lambda/\|\upsilon_f(m)\|)\,J_m\big(\upsilon_f(m)\big)\Big)\in M.
\end{equation}

Let us define $\mathcal{V}\in \mathcal{C}^\infty\big(M_E\times (-\delta,\delta)\big)$ by setting
$$\mathcal{V}(m,\lambda)\,\mathrm{d}\lambda\,\mathrm{d}V_{M_E}=\eta^*(\mathrm{d}V_M\big),$$
where $\mathrm{d}V_{M_E}$ is the volume form of the oriented manifold $M_E$ (for our practical purposes,
we may as well deal with Riemannian densities); in particular, $\mathcal{V}(m,0)\equiv 1$
identically.
Thus
\begin{equation}
 \label{eqn:integration on M}
  \mathrm{trace}(\mathcal{G}_k)\sim\int_{M_E}\int_{-\delta}^{\delta}
  \mathfrak{G}_k\big(\eta(m,\lambda)\big)\,\mathcal{V}(m,\lambda)\,\mathrm{d}\lambda\,\mathrm{d}V_{M_E}(m).
\end{equation}

Furthermore, if $\varrho_0$ is as on page \pageref{page defn varrho_0} then, by Corollary \ref{cor:rapid decay diagonal}
and its proof, we can 
insert the cut-off $\varrho_0 \left(k^{1/2-\epsilon}\,\lambda\right)$ without affecting
the asymptotics:
\begin{equation}
 \label{eqn:integration on M 1}
  \mathrm{trace}(\mathcal{G}_k)\sim\int_{M_E}\int_{-\delta}^{\delta}
  \mathfrak{G}_k\big(\eta(m,\lambda)\big)\,\mathcal{V}(m,\lambda)\,\varrho_0 \left(k^{1/2-\epsilon}\,\lambda\right)\,
  \mathrm{d}\lambda\,\mathrm{d}V_{M_E}(m),
\end{equation}
so that integration in $\mathrm{d}\lambda$ is now over the interval $\left(-2\,k^{\epsilon-1/2},2\,k^{\epsilon-1/2}\right)$.

Here $\mathcal{G}_k=\mathcal{G}_k^{(\chi)}$. Let us define
$
\mathfrak{P}_k:\mathbb{R}\times M\rightarrow \mathbb{C}
$
by setting
\begin{equation}
 \label{eqn:defn mathfrak Pk}
 \mathfrak{P}_k(\tau,m)=:\Pi_k\big(x_\tau,x)
\end{equation}
for an arbitrary choice of $x\in X_m$. Then by (\ref{eqn:gutzwiller integral})
$$
\mathfrak{G}_k\big(\eta(m,\lambda)\big)=\int_{\mathbb{R}}\,
e^{-ik\tau E}\,\chi(\tau)\,\mathfrak{P}_k\big(\tau,\eta(m,\lambda)\big)\,\mathrm{d}\tau,
$$
so that given (\ref{eqn:integration on M}) we have
\begin{eqnarray}
 \label{eqn:integration on M 2}
  \lefteqn{\mathrm{trace}(\mathcal{G}_k)}\\
  &\sim&\int_{M_E}\int_{-\delta}^{\delta}
  \int_{\mathbb{R}}\,e^{-ik\tau E}\,\chi(\tau)\,\mathfrak{P}_k\big(\tau,\eta(m,\lambda)\big)
  \,\mathcal{V}(m,\lambda)\,\varrho_0 \left(k^{1/2-\epsilon}\,\lambda\right)\nonumber\\
  &&\cdot\mathrm{d}\tau\,
  \mathrm{d}\lambda\,\mathrm{d}V_{M_E}(m),\nonumber
\end{eqnarray}

In view of Proposition \ref{prop:hypersurface periods}, $\mathrm{Per}^\mathbb{R}_M(M_E)\cap \mathrm{supp}(\chi)$
is a finite set $\{\sigma_b\}_{b\in \mathcal{B}_0}$ ($\mathcal{B}_0$ is a finite subset of $\mathcal{B}$
in the discussion on page \pageref{pageperiods}). Let us choose $\epsilon_1$ such that 
\begin{equation}
  \label{eqn:choice of epsilon_1}
 \min\{|\sigma_l-\sigma_h|\,:l,h\in \mathcal{B}_0,\,h\neq l\}/5>\epsilon_1>0, 
\end{equation}
and set  
\begin{equation*}
\widetilde{\varrho}_0(\tau)=:\sum_{b\in \mathcal{B}_0}\varrho_0\big((\tau-\sigma_b)/\epsilon_1\big)
\,\,\,\,\,(\tau\in \mathbb{R}).
\end{equation*}
Then $\widetilde{\varrho}_0\in \mathcal{C}_0^\infty(\mathbb{R})$ is (compactly) supported in
$\bigcup_{b\in \mathcal{B}_0}(\sigma_b-2\epsilon_1,\sigma_b+2\epsilon_1)$,
and it is $\equiv 1$ in $\bigcup_{b\in \mathcal{B}_0}(\sigma_b-\epsilon_1,\sigma_b+\epsilon_1)$.

\begin{lem}
\label{lem:1st reduction}
Only a rapidly decreasing contribution to
the asymptotics of $\mathrm{trace}(\mathcal{G}_k)$ is lost, if the integrand in (\ref{eqn:integration on M 2})
is multiplied by $\widetilde{\varrho}_0(\tau)$. 
\end{lem}

\begin{proof}
[Proof of Lemma \ref{lem:1st reduction}]
It suffices to show, in view of (\ref{eqn:integration on M 2}), (\ref{eqn:defn mathfrak Pk})
and Proposition \ref{prop:rapid decrease}, that there exists 
$\epsilon_2>0$ such that 
$
\mathrm{dist}_M\big(\eta(m,\lambda)_\tau,\eta(m,\lambda)\big)\ge
\epsilon_2
$
whenever $m\in M_E$, $|\lambda|\le 2\,k^{\epsilon-1/2}$, $\tau\in \mathrm{supp}(\chi)$, and $|\tau-\sigma_b|>\epsilon_1$
for each $b\in \mathcal{B}_0$.

Let us first show that there is $\epsilon_3>0$ such that 
$
\mathrm{dist}_M\big(m_\tau,m\big)\ge
\epsilon_3
$
whenever $m\in M_E$, $\tau\in \mathrm{supp}(\chi)$, and $|\tau-\sigma_b|>\epsilon_1$ for each $b\in \mathcal{B}_0$.

Indeed, if not we can find pairs $(m_h,\tau_h)\in M_E\times \mathrm{supp}(\chi)$ for $h=1,2,\ldots$,
such that $|\tau_h-\sigma_b|>\epsilon_1$ for each $b\in \mathcal{B}_0$ and $\mathrm{dist}_M(m_{h\,\tau_h},m_h)<1/h$.
By compactness of $M_E$ and $\mathrm{supp}(\chi)$, perhaps after passing to a subsequence we may assume
that $m_h\rightarrow m_\infty\in M_E$ and $\tau_h\rightarrow \tau_\infty\in \mathrm{supp}(\chi)$. 

By continuity, $m_{\infty\,\tau_\infty}=m_\infty$, 
whence $\tau_\infty\in \mathrm{Per}^\mathbb{R}_M(M_E)\cap \mathrm{supp}(\chi)$.
Thus $\tau_\infty=\sigma_b$ for some $b\in \mathcal{B}_0$; however this is clearly absurd since, again by continuity,
$|\tau_\infty-\sigma_b|\ge \epsilon_1$ for every $b\in \mathcal{B}_0$.

Given this, we have whenever $m\in M_E$, 
$|\lambda|\le 2\,k^{\epsilon-1/2}$, $\tau\in \mathrm{supp}(\chi)$, and $|\tau-\sigma_b|\ge \epsilon_1$
for each $b\in \mathcal{B}_0$:
\begin{eqnarray*}
\epsilon_3&\le& \mathrm{dist}_M(m,m_\tau)\\
&\le& \mathrm{dist}_M\big(m,\eta (m,\lambda)\big)+\mathrm{dist}_M\big(\eta (m,\lambda),
\eta (m,\lambda)_\tau\big)+\mathrm{dist}_M\big(\eta (m,\lambda)_\tau,
m_\tau\big),\nonumber\\
&=&2\,\mathrm{dist}_M\big(m,\eta (m,\lambda)\big)+\mathrm{dist}_M\big(\eta (m,\lambda),
\eta (m,\lambda)_\tau\big)\nonumber\\
&\le&5\,k^{\epsilon-1/2}+\mathrm{dist}_M\big(\eta(m,\lambda),
\eta(m,\lambda)_\tau\big).\nonumber
\end{eqnarray*}
Therefore, for $k\gg 0$ we have, say, 
$
\mathrm{dist}_M\big(\eta(m,\lambda),
\eta(m,\lambda)_\tau\big)\ge \epsilon_3/2
$.
\end{proof}

Let us set 
$\chi_b(\tau)=:\varrho_0\big((\tau-\sigma_b)/\epsilon_1\big)\cdot \chi(\tau)$.
Then by Lemma \ref{lem:1st reduction} we may rewrite
(\ref{eqn:integration on M 2}) as follows:

\begin{eqnarray}
 \label{eqn:trace computation3}
\lefteqn{
\sum_{j=1}^{N_k} \widehat{\chi}(kE-\lambda_{kj})}\\
&\sim&\sum_{b\in \mathcal{B}_0}\int_{M_E}\int_{-\delta}^{+\delta}\,\left[\int_{-\infty}^{+\infty}\,e^{-i\tau kE}\,\chi_b(\tau)\,
\mathfrak{P}_k\big(\tau,\eta(m,\lambda)\big)\,\mathrm{d}\tau\right]\nonumber\\
&&
\cdot \mathcal{V}(m,\lambda)\,\varrho_0\left (k^{1/2-\epsilon}\lambda\right)
\,\mathrm{d}\lambda\,\mathrm{d}V_{M_E}(m)\nonumber\\
&=& 
 \sum_{b\in \mathcal{B}_0}\int_{M_E}\int_{-\delta}^{+\delta}\,\mathfrak{G}_k^{(\chi_b)}\big(\eta(m,\lambda)\big)\cdot
 \mathcal{V}(m,\lambda)\,\varrho_0\left (k^{1/2-\epsilon}\lambda\right)
\,\mathrm{d}\lambda\,\mathrm{d}V_{M_E}(m),\nonumber\\
&\sim& \sum_{b\in \mathcal{B}_0}\mathrm{trace}\left(\mathcal{G}_k^{(\chi_b)}\right)\end{eqnarray}
where $\mathcal{G}_k^{(\chi_b)}$ is the Gutzwiller-T\"{o}plitz operator with $(f,E,\chi)$ replaced by $(f,E,\chi_b)$.

An argument similar to those used in the proof of Lemma \ref{lem:1st reduction} proves the following further
reduction.
Let $M_E(\sigma_b)=\bigcup_l M_E(\sigma_b)_l$ be as in (\ref{eqn:defn of MEsigma}).

\begin{lem}
 \label{lem:2nd reduction}
Only a rapidly decreasing contribution to the asymptotics is lost in (\ref{eqn:trace computation3}), if 
in the $b$-th summand integration in $\mathrm{d}V_{M_E}(m)$ 
is restricted to an arbitrarily small tubular neighborhood $U_E(\sigma_b)\subseteq M_E$
of $M_E(\sigma_b)$.
\end{lem}

\begin{proof}
 [Proof of Lemma \ref{lem:2nd reduction}]
 It suffices to prove that given $\delta_1>0$ there exists $\delta_2>0$ such that
 uniformly for $(m,\tau,\lambda)\in M_E\times \mathrm{supp}(\chi)\times \mathbb{R}$ satisfying:
 \begin{equation}
  \label{eqn:conditions on mtaulambda}
  |\tau-\sigma_b|\le 2\,\epsilon_1,\,\,
 \mathrm{dist}_M\big(m,M_E(\sigma_b)\big)\ge \delta_1,\,\,|\lambda|\le 2\,k^{\epsilon-1/2}
 \end{equation}
we have
 $\mathrm{dist}_M\big(\eta(m,\lambda),
\eta(m,\lambda)_\tau\big)\ge \delta_2$.
 
 Assume, to the contrary, that for $k=1,2,\ldots$ we can find 
 $$(m_k,\tau_k,\lambda_k)\in M_E\times \mathrm{supp}(\chi)\times
 \left(-2\,k^{\epsilon-1/2},2\,k^{\epsilon-1/2}\right)$$
 satisfying (\ref{eqn:conditions on mtaulambda}), and such that $\mathrm{dist}_M\big(\eta(m_k,\lambda_k),
\eta(m_k,\lambda_k)_{\tau_k}\big)\rightarrow 0$ when $k\rightarrow 0$.
Again by compactness, we may assume, perhaps after passing to a subsequence $k_j$, that 
$m_{k_j}\rightarrow m_\infty\in M_E$ and $\tau_{k_j}\rightarrow \tau_\infty\in [\sigma_b-2\epsilon_1,
\sigma_b+2\,\epsilon_1]$ when $j\rightarrow +\infty$.

It follows clearly that $\eta(m_{k_j},\lambda_{k_j})\rightarrow m_\infty$ and that
$m_{\infty\tau_\infty}=m_\infty$. Therefore, $\tau_\infty\in \mathrm{Per}^\mathbb{R}_M(M_E)\cap \mathrm{supp}(\chi)$.
Necessarily then we have $\tau_\infty=\sigma_b$, since by continuity $|\tau_k-\sigma_{b'}|\ge 3\,\epsilon_1$
for any $b'\in \mathcal{B}_0\setminus \{b\}$ (recall (\ref{eqn:choice of epsilon_1})). 
It thus follows that $m_\infty\in M_E(\sigma_b)$. But 
by continuity we also have $\mathrm{dist}_M\big(m_\infty,M_E(\sigma_b)\big)\ge \delta_1$,
and this is
absurd.
\end{proof}

We may assume that $U_E(\sigma_b)$
is the disjoint union of tubular neighborhoods $U_E(\sigma_b)_l\subseteq M_E$ of each
$M_E(\sigma_b)_l$. Thus, we can rewrite (\ref{eqn:trace computation3}) as follows: 
\begin{eqnarray}
 \label{eqn:trace computation5}
\lefteqn{
\sum_{j=1}^{N_k} \widehat{\chi}(kE-\lambda_{kj})}\\
&\sim&\sum_{b\in \mathcal{B}_0}\sum_{l=1}^{n_b}\int_{U_E(\sigma_b)_l}
\int_{-\delta}^{+\delta}\,\mathfrak{G}_k^{(\chi_b)}\big(\eta(m,\lambda)\big)\cdot \mathcal{V}(m,\lambda)
\,\varrho_0\left (k^{1/2-\epsilon}\lambda\right)
\,\mathrm{d}\lambda\,\mathrm{d}V_{M_E}(m).\nonumber\end{eqnarray}

As discussed in \S \ref{sctn:transverse intersections}, $M_E(\sigma_b)_l$
is a $\phi^M$-invariant submanifold of $M$, of dimension $2d_{bl}-1$.
In particular, $\upsilon_f(m)\in T_mM_E(\sigma_b)_l$ for each 
$m\in M_E(\sigma_b)_l$. Furthermore, the normal bundle $N\big(M_E(\sigma_b)_l/M_E)\big)$
of $M_E(\sigma_b)_l$ in $M_E$ is the restriction of the normal bundle of $M(\sigma_b)_l$
in $M$. In view of (\ref{eqn:tangent and normal}), at each $m\in M(\sigma_b)_l$ the normal space 
of $M_E(\sigma_b)_l$ in $M_E$ at $m$ is thus given by 
\begin{equation}
 \label{eqn:normal MEsigmaME}
 N_m\big(M_E(\sigma_b)_l/M_E)\big)=\mathrm{im}\left(\mathrm{d}\phi^M_{\sigma_b}-\mathrm{id}_{T_mM}\right).
\end{equation}

Hence we can find a finite open cover $\mathcal{U}=\{U_{blj}\}$ of
$M_E(\sigma_b)_l$ such that for each $j$ there are unitary
trivializations $N\big(M_E(\sigma_b)_l/M_E)\big)\cong U_{blj}\times \mathbb{C}^{c_{bl}}$.
Writing $\exp^{M_E}_m$ for the exponential map of $M_E$ at a point $m\in M_E(\sigma_b)_l$, we obtain
a local parametrization of $U_E(\sigma_b)_l\subseteq M_E$ along $U_{blj}$ by setting 
$$
\beta_{blj}:(m,\mathbf{n})\in U_{blj}\times B_{2c_{bl}}(\mathbf{0},\delta)\mapsto \exp^{M_E}_m(\mathbf{n})
\in U_E(\sigma_b)_l,
$$
where $\mathbf{n}\in B_{2c_{bl}}(\mathbf{0},\delta)\subseteq\mathbb{R}^{2c_{bl}}\cong \mathbb{C}^{c_{bl}}$
is identified with a normal vector through the previous unitary isomorphism.

Let $(\gamma_{blj})$ be a partition of unity on $M_E(\sigma_b)_l$ subordinate to the open cover
$\mathcal{U}$. We can then write $m=\exp^{M_E}_m(\mathbf{n})$ in (\ref{eqn:trace computation5}), and
express (\ref{eqn:trace computation5}) in the form
\begin{eqnarray}
 \label{eqn:partition of unity}
 \lefteqn{\sum_{j=1}^{N_k} \widehat{\chi}(kE-\lambda_{kj})}\\
 &\sim&\sum_{b\in \mathcal{B}_0}\sum_{l=1}^{n_b}\sum_j\int_{U_{blj}}\,\gamma_{blj}(m)\,
 \int_{B_{2c_{bl}}(\mathbf{0},\delta)}\,
\int_{-\delta}^{+\delta}\,\mathfrak{G}_k^{(\chi_b)}\left(\eta\left(\exp^{M_E}_m(\mathbf{n}),\lambda\right)\right)\nonumber\\
&&\cdot \mathcal{V}_1(m,\mathbf{n},\lambda)\,\varrho_0\left (k^{1/2-\epsilon}\lambda\right)
\,\mathrm{d}\lambda\,\mathrm{d}\mathbf{n}\,\mathrm{d}V_{M_E(\sigma_b)_l}(m),\nonumber
\end{eqnarray}
where $\mathcal{V}_1(\cdot,\mathbf{0},0)\equiv 1$, and $\mathrm{d}\mathbf{n}$ denotes the Lebesgue
measure on $\mathbb{R}^{2c_{bl}}$.
In view of (\ref{eqn:defn of eta}), we have
\begin{equation}
 \label{eqn:eta in argument}
 \eta\left(\exp^{M_E}_m(\mathbf{n}),\lambda\right)=
 \exp^M_{\exp^{M_E}_m(\mathbf{n})}\Big((\lambda/\|\upsilon_f(m)\|)\,J_m\big(\upsilon_f(m)\big)\Big)
\end{equation}
in (\ref{eqn:partition of unity}).
This is a parametrization 
$$\widetilde{\beta}_{blj}:U_{blj}\times B_{2c_{bl}}(\mathbf{0},\delta)\times
(-\delta,\delta)\rightarrow U_E(\sigma_b)_l;$$
we obtain another parametrization (perhaps changing $\delta$ and restricting
$U_E(\sigma_b)_l$),
$$\widehat{\beta}_{blj}:U_{blj}\times B_{2c_{bl}}(\mathbf{0},\delta)\times
(-\delta,\delta)\rightarrow U_E(\sigma_b)_l,$$
by using the unitary
trivialization of the normal bundle $N\big(M_E(\sigma_b)_l\big)$
to $M_E(\sigma_b)_l$ on $U_{blj}$, as in (\ref{eqn:normal space bl}), and setting instead:
\begin{equation}
 \label{eqn:eta in argument1}
\widehat{\beta}_{blj}:(m,\mathbf{n},\lambda)\mapsto 
 \exp_m^M\left(\frac{\lambda}{ \left\|\upsilon_f\left(m\right)\right\| }\,
J_{m}\big(\upsilon_f\left(m\right)\big)+\mathbf{n}\right).
\end{equation}
We have $\widehat{\beta}_{blj}^*(\mathrm{d}V_M)=
\widehat{\mathcal{V}}(m,\mathbf{n},\lambda)\,\mathrm{d}V_{M_E(\sigma_b)_l}
\,\mathrm{d}\mathbf{n}\,\mathrm{d}\lambda$ with $\mathcal{V}(\cdot,\mathbf{0},0)\equiv 1$.
Also, $\widetilde{\beta}_{blj}$ the same $\widehat{\beta}_{blj}$ induce the same differential
for $\mathbf{n}=\mathbf{0}$ and $\lambda=0$. Thus, if we apply the change of integration variables
$\widehat{\beta}_{blj}^{-1}\circ \widetilde{\beta}_{blj}$ we can reformulate (\ref{eqn:partition of unity})
in the following form:
\begin{eqnarray}
 \label{eqn:partition of unity1}
 \lefteqn{\sum_{j=1}^{N_k} \widehat{\chi}(kE-\lambda_{kj})}\\
 &\sim&\sum_{b\in \mathcal{B}_0}\sum_{l=1}^{n_b}\sum_j\int_{U_{blj}}\,\gamma_{blj}(m)\,
 \int_{B_{2c_{bl}}(\mathbf{0},\delta)}\,
\int_{-\delta}^{+\delta}\,\mathfrak{G}_k^{(\chi_b)}\left(\widehat{\beta}_{blj}(m,\mathbf{n},\lambda)\right)\nonumber\\
&&\cdot \mathcal{V}_2(m,\mathbf{n},\lambda)\,\varrho_0\left (k^{1/2-\epsilon}\lambda\right)
\,\mathrm{d}\lambda\,\mathrm{d}\mathbf{n}\,\mathrm{d}V_{M_E(\sigma_b)_l}(m),\nonumber
\end{eqnarray}
where again $\mathcal{V}_2(\cdot,\mathbf{0},0)\equiv 1$ ($\mathcal{V}_2$ also depends on $b,l,j$).

Perhaps after passing to a finer open cover, we may assume that the restriction to
$M_E(\sigma_b)_l$ of the tangent bundle $TM$ is unitarily trivialized
on each $U_{blj}$, and that this trivialization is consistent with the orthogonal direct sum 
decomposition in Definition \ref{defn:tvh components}. Paired with this trivialization,
$\exp_m$ provides for each $m\in U_{blj}$ a system of preferred local coordinates on $M$ centered
at $m$ (see (\ref{eqn:local parametrization})), 
for which $(\lambda/\|\upsilon_f(m)\|)\,J_m\big(\upsilon_f(m)\big)$ is a transverse tangent vector, 
of norm $|\lambda|$, and $\mathbf{n}$ is horizontal (see the discussion on page \pageref{page tubular neighborhood},
especially Remark \ref{rem:orthogonal decomposition}). It is however more convenient to use a slightly
different choice of preferred local coordinates centered at $m$, as follows.

Let us unitarily identify 
$\mathrm{span}\big(\upsilon_f(m)\big)$ and $\mathrm{span}\big(J_m\upsilon_f(m)\big)$ with $\mathbb{R}$
by $\tau\mapsto (\tau/\|\upsilon_f(m)\|)\,\upsilon_f(m)$, and 
$\lambda\mapsto (\lambda/\|\upsilon_f(m)\|)\,J_m\upsilon_f(m)$. We have also a smoothly varying unitary
isomorphism $S_m\cong \mathbb{R}^{2d-2}$, where $S_m$ is as in (\ref{eqn:RmSm}); recall that
$\mathrm{im}\left(\mathrm{d}\phi^M_{\sigma_b}-\mathrm{id}_{T_mM}\right)\subseteq S_m$, hence it corresponds to
a copy of $\mathbb{R}^{2c_{bl}}$ in $\mathbb{R}^{2d-2}$. 

Under these identifications,
we shall denote by $\mathbf{s}\in \mathbb{R}^{2d-2}$ the general element of
$S_m$, and by $\mathbf{n}\in \mathbb{R}^{2c_{bl}}\subseteq \mathbb{R}^{2d-2}$ 
the general element of $\mathrm{im}\left(\mathrm{d}\phi^M_{\sigma_b}-\mathrm{id}_{T_mM}\right)$.

We shall always adopt
additive notation, but set
\begin{eqnarray}
 \label{eqn:additive notation}
 m+(\tau,\lambda,\mathbf{s})
=:\phi^M_{-\tau/\|\upsilon_f(m)\|}\left(\exp_m^M\left(\frac{\lambda}{ \left\|\upsilon_f\left(m\right)\right\| }\,
J_{m}\big(\upsilon_f\left(m\right)\big)+\mathbf{s}\right)\right).
\end{eqnarray}
with 
$(\tau,\lambda,\mathbf{s})\in (-\delta,\delta)^2\times B_{2d-2}(\mathbf{0},\delta)\subseteq
\mathbb{R}\times\mathbb{R}^{2d-2}$.

Now let $\varrho_1\in \mathcal{C}^\infty_0\left(\mathbb{R}^{2c_{bl}}\right)$ satisfy
$\varrho_1(\mathbf{n})=0$ if $\|\mathbf{n}\|\ge 2$ and $\varrho_1(\mathbf{n})=1$ if $\|\mathbf{n}\|\le 1$.
Recall that $\chi_b$ in (\ref{eqn:partition of unity1}) is supported in $(\sigma_b-2\epsilon_1,\sigma_b+2\epsilon_1)$.

\begin{lem}
 \label{lem:reduction in mathbfn}
 Given that $\epsilon_1>0$ is sufficiently smal, 
only a negligible contribution to the asymptotics is lost in (\ref{eqn:partition of unity1}), if
 the integrand is multiplied by $\varrho_1\left(k^{1/2-\epsilon}\,\mathbf{n}\right)$.
\end{lem}

\begin{proof}
[Proof of Lemma \ref{lem:reduction in mathbfn}]
Consider $m\in U_{blj}\subseteq M_E(\sigma_b)_l$, and adopt notation (\ref{eqn:additive notation}). 
Since $\phi_{\sigma_b}$ is a holomorphic Riemannian isometry, it preserves
geodesics. Furthermore, it leaves $\upsilon_f$ invariant. 
Given that $\phi^M_{\sigma_b}(m)=m$, we then have
\begin{eqnarray}
 \label{eqn:azione di phi tau a}
\phi^M_{-\sigma_b}\big( m+\left(0,\lambda,
\mathbf{n}\right)\big)=m+\left(0,\lambda,
A_b\mathbf{n}\right),
\end{eqnarray}
where $A_b$ is the unitary (orthogonal and symplectic) $(2 c_{bl})\times (2 c_{bl})$ matrix
representing the restriction to $N_m\big(M(\sigma_b)_l\big)\cong \mathbb{R}^{2c_{bl}}$ of
$\mathrm{d}_{m}\phi^M_{-\sigma_b}:T_mM\rightarrow T_mM$. 

Given (\ref{eqn:additive notation}), 
we then have that for $\tau\sim 0$
\begin{eqnarray}
 \label{eqn:azione di phi tau }
\phi^M_{-(\tau+\sigma_b)}\big( m+(0,\lambda,\mathbf{n})\big)&=&
\phi^M_{-\tau}\big(m+(0,\lambda,A_b\mathbf{n})\big)\\
&=&m+\big(\tau\,\|\upsilon_f(m)\|,\lambda,A_b\mathbf{n}\big).\nonumber
\end{eqnarray}

Since the local chart (\ref{eqn:additive notation})
is isometric at the origin,  by Lemma \ref{lem:distance estimate local chart} we have,
perhaps after passing to a smaller $\delta$,
\begin{equation}
 \label{eqn:distance comparison}
 \mathrm{dist}_M\left(m+(\tau,\lambda,\mathbf{s}),m+(\tau',\lambda',\mathbf{s}')\right)\ge \frac{1}{2}\,
 \left\|(\tau-\tau',\lambda-\lambda',\mathbf{s}-\mathbf{s}')\right\|
\end{equation}
for all 
$(\tau,\lambda,\mathbf{s}),\,(\tau',\lambda',\mathbf{s}')\in (-\delta,\delta)^2\times B_{2d-2}(\mathbf{0},\delta)$.
 
 Let us apply this with $\mathbf{s}=\mathbf{n}, \,\mathbf{s}'=A_b\mathbf{n}\in \mathbb{R}^{2c_{bl}}\subseteq 
 \mathbb{R}^{2d-2}$. Recalling that $A_b-I_{2c_{bl}}$ is invertible, 
 we have  that for $\tau\sim 0$
 \begin{eqnarray}
  \label{eqn:bound on distance action}
  \lefteqn{\mathrm{dist}_M\left(\phi^M_{-(\tau+\sigma_b)}\big( m+(0,\lambda,\mathbf{n})\big),
  m+(0,\lambda,\mathbf{n})\right)}\nonumber\\
  &=&\mathrm{dist}_M\left(m+\big(\tau\,\|\upsilon_f(m)\|,\lambda,A_b\mathbf{n}\big),
  m+(0,\lambda,\mathbf{n})\right)\nonumber\\
  &\ge&
  \frac{1}{2}\,\big\|\big(\tau\,\|\upsilon_f(m)\|,A_b\mathbf{n}-\mathbf{n})\big\|\ge C_1\,\|\mathbf{n}\|,
 \end{eqnarray}
where $C_1>0$ depends only on $E$, $b$ and $l$.

Hence if $\epsilon_1>0$ is sufficiently small, we deduce that
\begin{eqnarray}
 \label{eqn:bound on distance action 2}
 \mathrm{dist}_M\left(m+(0,\lambda,\mathbf{n}),
 \big(m+(0,\lambda,\mathbf{n})\big)^{(\chi_b)}\right)
 \ge C_1\,\|\mathbf{n}\|
\end{eqnarray}
for some constant $C_1>0$, and the statement follows from Theorem \ref{thm:rapid decrease}.
 
\end{proof}

With the rescaling $\lambda\mapsto \lambda/\sqrt{k}$ and $\mathbf{n}\mapsto \mathbf{n}/\sqrt{k}$,
we can then rewrite (\ref{eqn:partition of unity1}) as

\begin{eqnarray}
 \label{eqn:partition of unity2}
 \lefteqn{\sum_{j=1}^{N_k} \widehat{\chi}(kE-\lambda_{kj})}\\
 &\sim&\sum_{b\in \mathcal{B}_0}\sum_{l=1}^{n_b}k^{-c_{bl}-\frac{1}{2}}\,\sum_j\int_{U_{blj}}\,\gamma_{blj}(m)\,
 \int_{\mathbb{R}^{2c_{bl}}}\,
\int_{-\infty}^{+\infty}\,\mathfrak{G}_k^{(\chi_b)}\left(m+\left(0,\frac{\lambda}{\sqrt{k}},\frac{\mathbf{n}}{\sqrt{k}}\right)\right)
\nonumber\\
&&\cdot \mathcal{V}_2\left(m,\frac{\mathbf{n}}{\sqrt{k}},\frac{\lambda}{\sqrt{k}}\right)\,
\varrho_0\left (k^{-\epsilon}\lambda\right)\,\varrho_1\left(k^{-\epsilon}\,\mathbf{n}\right)
\,\mathrm{d}\lambda\,\mathrm{d}\mathbf{n}\,\mathrm{d}V_{M_E(\sigma_b)_l}(m)\nonumber\\
&=&\sum_{b\in \mathcal{B}_0}\sum_{l=1}^{n_b}\,\int_{M_E(\sigma_b)_l}\,\sum_j\gamma_{blj}(m)\,
 I_{lbj}(m,k)\,\mathrm{d}V_{M_E(\sigma_b)_l}(m),\nonumber
\end{eqnarray}
where we have set
\begin{eqnarray}
 \label{eqn:defnIlbjk}
  I_{lbj}(m,k)
 &=:&k^{-c_{bl}-\frac{1}{2}}\,\int_{\mathbb{R}^{2c_{bl}}}\,
\int_{-\infty}^{+\infty}\,\mathfrak{G}_k^{(\chi_b)}\left(m+\left(0,\frac{\lambda}{\sqrt{k}},\frac{\mathbf{n}}{\sqrt{k}}\right)\right)
\nonumber\\
&&\cdot \mathcal{V}_2\left(m,\frac{\mathbf{n}}{\sqrt{k}},\frac{\lambda}{\sqrt{k}}\right)\,
\varrho_0\left (k^{-\epsilon}\lambda\right)\,\varrho_1\left(k^{-\epsilon}\,\mathbf{n}\right)
\,\mathrm{d}\lambda\,\mathrm{d}\mathbf{n};
\end{eqnarray}
integration in $\mathrm{d}\lambda\,\mathrm{d}\mathbf{n}$ in (\ref{eqn:defnIlbjk})
is over an open ball centered at the origin
and of radius $O\left(k^\epsilon\right)$ in $\mathbb{R}^{1+2c_{bl}}$.
The dependence on $j$ of the right hand side of (\ref{eqn:defnIlbjk}) is implicit through the above choices of 
trivializations and HLC systems.

To find an asymptotic expansion for (\ref{eqn:defnIlbjk}), let us notice that for an arbitrary choice 
of $x\in X_m$ and of a preferred local frame of $A$ centered at
$x$, we obtain a system of HLC centered at $x$. Let
$\mathbf{w}\in T_mM$ be defined by $\mathbf{w}_\mathrm{t}=(\lambda/\|\upsilon_f(m)\|)\,J_m\upsilon_f(m)$, 
$\mathbf{w}_\mathrm{h}=\mathbf{n}$,
and $\mathbf{w}_\mathrm{v}=\mathbf{0}$.
Recalling (\ref{eqn:defn of Gk diagonal}), we have
\begin{eqnarray}
 \label{eqn:lift to HLC}
 \mathfrak{G}_k^{(\chi_b)}\left(m+\left(0,\frac{\lambda}{\sqrt{k}},\frac{\mathbf{n}}{\sqrt{k}}\right)\right)
 =\mathcal{G}_k^{(\chi_b)}\left(x+\frac{\mathbf{w}}{\sqrt{k}},x+\frac{\mathbf{w}}{\sqrt{k}}\right).
\end{eqnarray}

Theorem \ref{thm:scaling asymptotics} provides an asymptotic expansion for the right
hand side of
(\ref{eqn:lift to HLC}),
when we take $y=x$, $\theta_1=\theta_2=0$, $\mathbf{v}_1=\mathbf{v}_2=\mathbf{w}$,
$\tau_a=\sigma_b$ $\vartheta_a=\vartheta_{bl}$ ($\sigma_b$ is the only period in the support 
of $\chi_b$. 

By construction, we have $\chi_b(\sigma_b)=\chi(b)$ and 
\begin{eqnarray}
\label{eqn:defined exponent}
 \mathcal{Q}(A_a\mathbf{w},\mathbf{w})=\psi_2\big(A_b\mathbf{n},\mathbf{n})-2\,|\lambda|^2.
\end{eqnarray}

Given this, the asymptotic expansion of Theorem \ref{thm:scaling asymptotics} yields
\begin{eqnarray}
 \label{eqn:asymptotic expansion chib}
 \lefteqn{\mathcal{G}_k^{(\chi_b)}\left(x+\frac{\mathbf{w}}{\sqrt{k}},x+\frac{\mathbf{w}}{\sqrt{k}}\right)\sim
 \dfrac{\sqrt{2}}{\|\upsilon_f(m)\|}\,\left(\frac{k}{\pi}\right)^{d-1/2}}\\
 &&\cdot 
 e^{ik(\vartheta_{bl}-\sigma_b E)+\psi_2\big(A_b\mathbf{n},\mathbf{n})-2\,|\lambda|^2}\,\left[\chi(\sigma_b)
+\sum_{s=1}^{+\infty}k^{-s/2}\,P_{bs}(m;
\lambda,\mathbf{n})\right],\nonumber
\end{eqnarray}
where $P_{as}(m,\cdot)$ is a polynomial of degree $\le 3s$ and parity $s$.

Furthermore, if we Taylor expand $\mathcal{V}_2(m,\cdot,\cdot)$ at the origin, and recall that
$\mathcal{V}_1(m,\mathbf{0},0)=1$, we obtain an asymptotic expansion
\begin{equation}
 \label{eqn:taylor expand V2}
 \mathcal{V}_2\left(m,\frac{\mathbf{n}}{\sqrt{k}},\frac{\lambda}{\sqrt{k}}\right)\sim
 1+\sum_{r\ge 1}k^{-r/2}\,F_r(m;\lambda,\mathbf{n}),
\end{equation}
where $F_r$ is a homogeneous polynomial of degree $r$.
The step-$N$ remainder is bounded by $C_{N+1}\,k^{-(N+1)(1/2-\epsilon)}$. 

Multiplying (\ref{eqn:asymptotic expansion chib}) and (\ref{eqn:taylor expand V2}), we obtain for the product
of the first two factors in the integrand of (\ref{eqn:defnIlbjk}) 
an asymptotic expansion as in (\ref{eqn:asymptotic expansion chib}), but with the $P_{bs}$'s
($j\ge 1$)
replaced with polynomials $R_{bs}$'s with the same properties. Since integration is over a ball of
radius $O\left(k^\epsilon\right)$, the expansion can be integrated term by term.

On the other hand, the exponent in (\ref{eqn:defined exponent}) satisfies
\begin{equation}
 \label{eqn:bound real part}
 \Re \left(\psi_2\big(A_b\mathbf{n},\mathbf{n})-2\,|\lambda|^2\right)\le
 -c\,\left(\|\mathbf{n}\|^2+|\lambda|^2\right)
\end{equation}
for some constant $c>0$; therefore, only a rapidly decreasing contribution to the asymptotics is lost,
if the cut-off functions are omitted. Thus we obtain for $I_{lbj}(m,k)$ an asymptotic expansion of the 
form 
\begin{eqnarray}
 \label{eqn:defnIlbjk1}
  I_{lbj}(m,k)
 &\sim&\dfrac{\sqrt{2}}{\|\upsilon_f(m)\|}\,\left(\frac{k}{\pi}\right)^{d-1/2}\cdot k^{-c_{bl}-\frac{1}{2}}\,
 e^{ik(\vartheta_{bl}-\sigma_b E)}\nonumber\\
 &&\cdot\sum_{s=0}^{+\infty}\,k^{-s/2}\, I_{lbj}(m,s),
 \end{eqnarray}
 where we have set 
 \begin{eqnarray}
 \label{eqn:Ilbjs}
 I_{lbj}(m,s)=:\int_{\mathbb{R}^{2c_{bl}}}\,
\int_{-\infty}^{+\infty}\, 
 e^{\psi_2\big(A_b\mathbf{n},\mathbf{n})-2\,|\lambda|^2}\,
R_{bs}(m;
\lambda,\mathbf{n})
\,\mathrm{d}\lambda\,\mathrm{d}\mathbf{n},
\end{eqnarray}
with $R_{b0}(m;
\lambda,\mathbf{n})=:\chi(\sigma_b)$. Since $\psi_2\big(A_b\mathbf{n},\mathbf{n})-2\,|\lambda|^2$
is an even function of $(\mathbf{n},\lambda)$ and $R_{bs}(m;
\lambda,\mathbf{n})$ has the same parity as $s$, $I_{lbj}(m,s)=0$ if $s$ is odd.

Let us insert this in (\ref{eqn:partition of unity2}). We obtain
\begin{equation}
 \label{eqn:spectral asymptotics}
 \sum_{j=1}^{N_k} \widehat{\chi}(kE-\lambda_{kj})\sim \sum_{b\in \mathcal{B}_0}\sum_{l=1}^{n_b}\,\mathcal{I}(b,l,k)
\end{equation}
where each $\mathcal{I}(b,l,k)$ is an asymptotic expansion 
\begin{eqnarray}
 \label{eqn:spectral asymptotics summand}
 \lefteqn{\mathcal{I}(b,l,k)\sim\sqrt{2}\,\pi^{1/2-d}\,k^{d_{bl}-1}\,
 e^{ik(\vartheta_{bl}-\sigma_b E)}}\\
 &&\cdot\sum_{s=0}^{+\infty}\,k^{-s}\,\int_{M_E(\sigma_b)_l}\,\frac{1}{\|\upsilon_f(m)\|}\,\sum_j\gamma_{blj}(m)\, I_{lbj}(m,2s)
 \,\mathrm{d}V_{M_E(\sigma_b)_l}(m).\nonumber
\end{eqnarray}

Let us compute the leading order term in (\ref{eqn:spectral asymptotics summand}). We have
in view of (\ref{eqn:Ilbjs}) 
\begin{eqnarray}
 \label{eqn:Ilbj0}
 I_{lbj}(m,0)&=&\chi(\sigma_b)\,
 \left(\int_{\mathbb{R}^{2c_{bl}}}\,e^{\psi_2\big(A_b\mathbf{n},\mathbf{n})}\,\mathrm{d}\mathbf{n}\right)\cdot 
\left(\int_{-\infty}^{+\infty}\,  e^{-2\,|\lambda|^2}\,
\mathrm{d}\lambda\right)\\
&=&\frac{\chi(\sigma_b)}{\sqrt{2}}\cdot \frac{\pi^{c_{bl}+\frac{1}{2}}}{D(b,l)},
\end{eqnarray}
where $D(b,l)$ was defined on page \pageref{pageDeterminant}; we have made use of the computation on
page 235 of \cite{pao-trace}.

Thus the leading order term in (\ref{eqn:spectral asymptotics summand}) is
\begin{eqnarray}
 \label{eqn:leading order term1}
\frac{\chi(\sigma_b)}{D(b,l)}\,\left(\frac{k}{\pi}\right)^{d_{bl}-1}\,
 e^{ik(\vartheta_{bl}-\sigma_b E)}\cdot 
 \int_{M_E(\sigma_b)_l}\,\frac{1}{\|\upsilon_f(m)\|}
 \,\mathrm{d}V_{M_E(\sigma_b)_l}(m).
\end{eqnarray}

This completes the proof of Theorem \ref{thm:trace of G_k}.

\end{proof}


\begin{thebibliography}{Dillo99}



\bibitem[B]{ber}F. A. Berezin, {\em General concept of quantization}, Comm. Math. Phys., \textbf{40} (1975),
153--174



\bibitem[BSZ1]{bsz1} P. Bleher, B. Shiffman, S. Zelditch, {\em
Universality and scaling of correlations between zeros on complex
manifolds}, Invent. Math. \textbf{142} (2000), 351--395

\bibitem[BSZ2]{bsz2} P. Bleher, B. Shiffman, S. Zelditch, {\em
Universality and scaling of zeros on symplectic 
manifolds}, Random Matrices and Their Applications, MSRI Publications, Volume \textbf{40} (2001), 31--69

\bibitem[BMS]{bms} M. Bordemann, E. Meinrenken, M. Schlichenmaier,  
{\em Toeplitz quantization of K\"{a}hler manifolds and gl(N), $N\rightarrow \infty$ limits}, Comm. Math. Phys. \textbf{165} (1994), no. 2, 281--296


\bibitem[BPU1]{bpu1}
D. Borthwick, T. Paul, A. Uribe, {\em Legendrian distributions
with applications to relative Poincar\'{e} series}, Invent. Math.
\textbf{122} (1995), no. 2, 359--402

\bibitem[BPU2]{bpu2}
D. Borthwick, T. Paul, A. Uribe, {\em  Semiclassical spectral estimates for Toeplitz operators}, 
Ann. Inst. Fourier (Grenoble) \textbf{48} (1998), no. 4, 1189--1229

\bibitem[BG]{bdm-g} L. Boutet de Monvel, V. Guillemin,
{\em The spectral theory of Toeplitz operators},
Annals of Mathematics Studies, \textbf{99} (1981),
Princeton University Press, Princeton, NJ;
University of Tokyo Press, Tokyo



\bibitem[BS]{bs} L. Boutet de Monvel, J. Sj\"ostrand,
{\em Sur la singularit\'e des noyaux de Bergman et de Szeg\"o},
Ast\'erisque \textbf{34-35} (1976), 123--164

\bibitem[CGR]{cgr} 
M. Cahen, S. Gutt, J. Rawnsley, {\em Quantization of K\"{a}hler manifolds. I. Geometric interpretation of Berezin's quantization}, 
J. Geom. Phys. \textbf{7} (1990), no. 1, 45--62

\bibitem[Ch1]{ch1} L. Charles, {\em Berezin-Toeplitz operators, a semi-classical approach}, Comm. Math. Phys. \textbf{239} 
(2003), no. 1-2, 1--28

\bibitem[Ch2]{ch2} L. Charles, {\em Quantization of compact symplectic manifolds}, http://arxiv.org/abs/1409.8507

\bibitem[C]{christ} M. Christ, {\em Slow off-diagonal decay for Szeg\"{o} kernels associated to smooth Hermitian line bundles}, 
Harmonic analysis at Mount Holyoke (South Hadley, MA, 2001), 77--89, Contemp. Math., \textbf{320}, Amer. Math. Soc., Providence, RI, 2003



\bibitem[G]{g-star} V. Guillemin, {\em Star products on compact pre-quantizable symplectic manifolds},
Lett. Math. Phys. \textbf{35} (1995), no. 1, 85-89

\bibitem[H]{hor} L. H\"{o}rmander, {\em The analysis of linear partial differential operators. I.
Distribution theory and Fourier analysis}, Second edition. Springer Study Edition. Springer-Verlag, Berlin, 1990

\bibitem[KS]{ks}  A. Karabegov, M. Schlichenmaier, {\em Identification of Berezin-Toeplitz deformation quantization},
J. Reine Angew. Math. \textbf{540} (2001), 49--76

\bibitem[MM1]{mm1} X. Ma, G.Marinescu, {\em Holomorphic Morse inequalities and Bergman kernels},
Progress in Mathematics \textbf{254},
Birkhšauser Verlag, Basel, 2007

\bibitem[MM2]{mm2} X. Ma, G.Marinescu, {\em Generalized Bergman kernels on symplectic manifolds},
Adv. Math. \textbf{217} (2008), no. 4, 1756-1815




\bibitem[MZ]{mz} X. Ma, W. Zhang, {\em Bergman kernels and symplectic reduction}, Ast\'{e}risque No. \textbf{318} (2008)

\bibitem[McDS]{mcds} D. McDuff, D. Salamon, {\em Introduction to symplectic topology},
Second edition. Oxford Mathematical Monographs. The Clarendon Press, Oxford University Press, New York, 1998



\bibitem[MS]{MelSjo}  A. Melin, J. Sjöstrand, {\em Fourier integral operators with complex-valued phase functions}, 
Fourier integral operators and partial differential equations 
(Colloq. Internat., Univ. Nice, Nice, 1974), pp. 120--223. 
Lecture Notes in Math., Vol. \textbf{459}, Springer, Berlin, 1975


\bibitem[P1]{pao-trace} R. Paoletti, {\em Szeg\"{o} kernels, Toeplitz operators, and equivariant fixed point formulae},
J. Anal. Math. 106 (2008), 209236

\bibitem[P2]{pao_imrn} R. Paoletti, {\em Local asymptotics for slowly shrinking spectral bands of a Berezin-Toeplitz operator},
Int. Math. Res. Not. IMRN 2011, no. \textbf{5}, 1165–1204


\bibitem[P3]{pao-torus} R. Paoletti, {\em Asymptotics of Szeg\"{o} kernels under Hamiltonian
torus actions}, Israel Journal of Mathematics
2011, DOI: 10.1007/s11856-011-0212-4

\bibitem[P4]{pao-ltfII} R. Paoletti, {\em Local trace formulae and scaling asymptotics for general quantized Hamiltonian flows},
J. Math. Phys. \textbf{53}, 023501 (2012); doi: 10.1063/1.3679660 

\bibitem[P5]{pao_ijm_2012} R. Paoletti, {\em 
Scaling asymptotics for quantized Hamiltonian flows}, Internat. J. Math. \textbf{23} (2012), no. 10, 1250102, 25 pp

\bibitem[S]{s} M. Schlichenmaier, {\em Berezin-Toeplitz quantization for compact K\"{a}hler manifolds. A review of results}, 
Adv. Math. Phys. 2010, Art. ID 927280, 38 pp


\bibitem[SZ]{sz} B. Shiffman, S. Zelditch, {\em Asymptotics of almost
holomorphic sections of ample line bundles on symplectic
manifolds}, J. Reine Angew. Math. {\bf 544} (2002), 181--222



\bibitem[Z1]{z1} S. Zelditch,
{\em Index and dynamics of quantized contact transformations},
Ann. Inst. Fourier (Grenoble) \textbf{47} (1997), no. 1, 305--363

\bibitem[Z2]{z2} S. Zelditch, {\em Szeg\"o kernels and a theorem of Tian},
Int. Math. Res. Not. {\bf 6} (1998), 317--331





\end{thebibliography}
\end{document}